\documentclass[11pt,a4paper,twoside]{article}
\usepackage[caption=false]{subfig}
\usepackage{amsfonts,amsmath,amsthm}
\usepackage[margin=1.8cm]{geometry}
\usepackage[algo2e,vlined,ruled]{algorithm2e}
\usepackage{xcolor}
\usepackage{verbatim}
\usepackage{textcomp}
\usepackage{dsfont}
\usepackage{hyperref}
\usepackage{fancyhdr}
\usepackage{amssymb}
\usepackage{mathrsfs}
\usepackage{cases}
\usepackage{lineno}
\usepackage{authblk}

\usepackage[page,title]{appendix}
\newtheorem{theorem}{Theorem}[section]
\newtheorem{lemma}[theorem]{Lemma}
\newtheorem{assumption}{Assumption}
\theoremstyle{definition}
\newtheorem{definition}[theorem]{Definition}
\newtheorem{example}[theorem]{Example}

\newtheorem{proposition}[theorem]{Proposition}
\theoremstyle{remark}
\newtheorem{remark}[theorem]{Remark}

\numberwithin{equation}{section}
\usepackage{setspace,url}
\usepackage{graphics,graphicx,epstopdf}
\usepackage{float}
\usepackage{subfig}
\usepackage{tabularx}
\usepackage{longtable}
\usepackage{booktabs,multirow,array,multicol}
\usepackage{enumitem}
\newcommand{\R}{\mathbb{R}}
\newcommand{\N}{\mathbb{N}}

\newcommand{\X}{\mathbb{X}}
\newcommand{\Y}{\mathbb{Y}}

\newcommand{\dom}{\mathrm{dom}}

\newcommand{\ewmin}{\lambda_{\min}}
\newcommand{\cA}{\mathcal{A}}

\newcommand{\cC}{\mathcal{C}}
\newcommand{\cD}{\mathcal{D}}
\newcommand{\cF}{\mathcal{F}}

\newcommand{\cH}{\mathcal{H}}
\newcommand{\cI}{\mathcal{I}}

\newcommand{\cM}{\mathcal{M}}

\newcommand{\cL}{\mathcal{L}}

\newcommand{\cT}{\mathcal{T}}

\newcommand{\sL}{\mathscr{L}}

\newcommand{{\KL}  }{K\L{}}

\newcommand{\prox}{\mathrm{prox}}
\newcommand{\dist}{\mathrm{dist}}
\newcommand{\crit}{\mathrm{crit}}
\newcommand{\sign}{\mathrm{sgn}}
\newcommand{\Exp}{\mathbb{E}}
\newcommand{\Sto}{\widetilde{\nabla}}

\newcommand{\iprod}[2]{\langle #1, #2 \rangle}
\DeclareMathOperator*{\argmin}{arg\,min}

\newcommand{\be}{\begin{equation}}
\newcommand{\ee}{\end{equation}}
\newcommand{\bee}{\begin{equation*}}
\newcommand{\eee}{\end{equation*}}

\newcommand{\mycomment}[1]{}

\title{\textbf{Preconditioned  Primal-Dual Gradient Methods for Nonconvex Composite and Finite-Sum Optimization}}

\author[1]{Jiahong Guo}
\author[2]{Xiao Wang}
\author[1]{Xiantao Xiao}
\affil[1]{School of Mathematical Sciences, Dalian University of Technology, Dalian 116024, China}
\affil[2]{Peng Cheng Laboratory, Shenzhen 518066, China}
\begin{document}
\maketitle
\begin{abstract}
In this paper, we first introduce a preconditioned  primal-dual gradient algorithm based on conjugate duality theory. This algorithm is designed to solve  composite optimization problem
 whose objective function consists of two summands:  a continuously differentiable nonconvex function and the composition of a nonsmooth nonconvex function with a linear operator. In contrast to existing nonconvex primal-dual algorithms, our proposed algorithm, through the utilization of conjugate duality,  does not require the calculation of  proximal mapping of nonconvex functions. Under mild conditions, we prove that any cluster point of the generated sequence is a critical point of the composite optimization problem. In the context of Kurdyka-\L{}ojasiewicz property, we establish global convergence and convergence rates for the iterates.
Secondly, for nonconvex finite-sum optimization,  we propose a stochastic algorithm that combines  the preconditioned primal-dual gradient algorithm with a class of variance reduced stochastic gradient estimators. Almost sure global convergence and expected convergence rates are derived relying on the  Kurdyka-\L{}ojasiewicz inequality.
Finally, some preliminary numerical results are presented to demonstrate the effectiveness of the proposed algorithms.

\end{abstract}



\section{Introduction}\label{sec:intro}
In this paper, we first consider the following  composite optimization problem:
\be\label{eq:p}
\begin{aligned}
\min_{x\in \X }\ f(x)+h(Ax),
\end{aligned}
\ee
where $\X$ and $\Y$ are finite-dimensional vector spaces, $f:\X \to\R$ is a continuously differentiable and possibly nonconvex function, $A:\X \to\Y $ is a  linear operator and $h:\Y \to(-\infty, +\infty]$ is a simple and possibly nonsmooth, nonconvex function.  Problem (\ref{eq:p}) arises in a variety of practical applications from machine learning, statistics, image processing, and so on. In many applications, function $h$ is usually referred to the \textit{regularizer} which is used to guarantee certain regular properties of the solution. Recently,  nonconvex regularizers, such as $\ell_0$, $\ell_p$ ($0<p<1$), smoothly clipped absolute deviation (SCAD) and minimax concave penalty (MCP), have drawn a lot of attention and been witnessed to achieve significant improvement over convex regularizers, see \cite{WCLQ2018} and references therein.

For problem (\ref{eq:p}) in the fully nonconvex setting (both $f$ and $h$ are nonconvex),  there has been an intensive renewed interest in the convergence analysis of various algorithms based on the Kurdyka-\L{}ojasiewicz ({\KL}) property in recent years. Attouch et al. \cite{ABS2013} established the global convergence of a forward-backward splitting algorithm for (\ref{eq:p}) with $A$ being the identity operator and $(f+h)$ being a {\KL}  function. Li and Pong \cite{Li2015} demonstrated the convergence of  an alternating direction method of multipliers (ADMM) under assumptions that both $f$ and $h$ are semialgebraic, and $A$ is surjective.  A nonmonotone linesearch algorithm based on the forward-backward envelope was proposed in \cite{TSP2018} and  shown to own superlinear convergence rates. Geiping and Moeller \cite{GM2018} investigated a class of majorization-minimization methods for (\ref{eq:p}) with a nonlinear operator $A$, and derived the global convergence under the {\KL}  property and the uniqueness of R-minimizing solutions.  In \cite{BST2018},  the authors employed  a Lyapunov method to established the global convergence of a bounded sequence to a critical point  for several Lagrangian-based methods, including proximal multipliers method and proximal ADMM, within the semialgebraic setting. By assuming that the associated augmented Lagrangian possesses the {\KL}  property, Bo{\c{t}} and Nguyen  \cite{BN2020} proved that the iterates of proximal ADMM converge to a Karush-Kuhn-Tucker point. They also derived convergence rates for both the augmented Lagrangian and the iterates. The algorithms for problem (\ref{eq:p}) with $h$ being $\ell_0$ norm were reviewed in the survey paper \cite{TBLS2021}.  For  problem (\ref{eq:p}) with  convex $h$, there exists a vast literature on various nonconvex composite optimization algorithms, see, for instance, \cite{Nesterov2013,BNO2016,BLPPR2017,BPR2020}.

Motivated by a class of well-studied primal-dual hybrid gradient (PDHG) algorithms for convex optimization  \cite{CP2011,PC2011,LXY2021}, and drawing upon the conjugate duality theory for nonconvex optimization, we propose a preconditioned first-order primal-dual algorithm  for solving  nonconvex composite optimization problem (\ref{eq:p}). In most of the aforementioned related algorithms, it is required to compute the elements of the generalized proximal (set-valued) mapping for nonconvex function $h$ and/or nonconvex function $f$ at each iteration. In contrast, at each iteration of our proposed algorithm, we only need to calculate the proximal mapping of the conjugate function $h^*$ which is always convex and lower semicontinuous. This fact makes the proposed algorithm much easier to implement in many scenarios.

In the second part of this paper, we consider to extend the proposed algorithm for
 the following finite-sum optimization problem:
\be\label{eq:p-finite-sum}
\begin{aligned}
	\min_{x\in \X}\ \frac{1}{N}\sum_{i=1}^N f_i(x)+h(Ax),
\end{aligned}
\ee
where $f_i:\X \to\R, i=1,\cdots,N$ are  continuously differentiable and possibly nonconvex,  and $h(Ax)$ is the same as in (\ref{eq:p}). Problem (\ref{eq:p-finite-sum}) arises frequently in the fields of statistics \cite{HTF2009} and machine learning \cite{BCN2018}. In many applications, problem (\ref{eq:p-finite-sum}) is also called as \emph{regularized empirical risk minimization} and the component functions $f_i, i=1,\ldots,N$ correspond to certain loss models. Moreover, in various interesting problems such as deep learning, dictionary learning and classification with nonconvex activation functions, the loss functions $f_i$ exhibit nonconvexity. Since the number of components $N$ (usually represents the size of a dataset) can be extremely large, the  exact computation of the full gradient $\frac{1}{N}\sum_{i=1}^N \nabla f_i(x)$ becomes prohibitively expensive in practice. Consequently, the stochastic approximation techniques have gained increasing importance in designing efficient numerical  algorithms for problem (\ref{eq:p-finite-sum}), see \cite{Lan2020} for example. In particular, the success of many popular variance reduced stochastic algorithms for convex finite-sum optimization has been witnessed in recent years, such as SAG \cite{SLB2017}, SAGA \cite{DBL2014}, SVRG \cite{JOhnson2013} and SARAH \cite{NLST2017}.

For problem (\ref{eq:p-finite-sum}) with nonconvex $f_i$ and convex $h$, a large amount of algorithms have been developed over the past few years. We only name a few here. Li and Li \cite{LL2018}  introduced a stochastic proximal gradient algorithm based on variance reduction, and established a global linear convergence rate for nonconvex $f_i$ satisfying Polyak-\L{}ojasiewicz condition. Nhan et al. \cite{PNPT2020} presented a stochastic first-order algorithm by combining a proximal gradient step with the SARAH estimator, and analyzed the complexity bounds in terms of stochastic first-order oracle calls. Fort and Moulines \cite{FM2022} introduced a stochastic variable metric proximal gradient algorithm by using a mini-batch strategy with variance reduction called SPIDER \cite{FLLZ2018}. In  \cite{WWY2019}, a generic algorithmic framework for stochastic proximal quasi-Newton methods was  introduced. Milzarek et al. \cite{MXCWU2019,MXWU2022} proposed a stochastic semismooth Newton method  for nonsmooth nonconvex stochastic and finite-sum optimization, and established the almost sure global convergence as well as local convergence rates with high probability. Jin and Wang \cite{JW2022} studied a single-loop stochastic primal-dual method for problem (\ref{eq:p-finite-sum}) coupled with a large number of nonconvex functional constraints.

We next review the stochastic approximation algorithms for problem (\ref{eq:p-finite-sum}) in the fully nonconvex setting, where $f_i, i=1,\ldots,N$ and $h$ are nonconvex. Xu et al. \cite{XJY2019} showed that the stochastic proximal gradient methods for problem (\ref{eq:p-finite-sum}) with nonconvex $h$ enjoy the same complexities as their counterparts for convex regularized problem to find an approximate stationary point. Cheng et al. \cite{CWC2022} proposed an interior stochastic gradient method for bounded constrained optimization problems where the objective function is the sum of an expectation function and a nonconvex  $\ell_p$ regularizer.
A stochastic algorithm that combines ADMM with a class of variance reduced stochastic gradient estimators, including SAGA, SVRG and SARAH, was proposed in \cite{BLZ2021}. The global convergence in expectation was established under the condition that $f_i, i=1,\ldots,N$ and $h$ are semialgebraic, and the convergence rates in the expectation sense were derived based on \L{}ojasiewicz exponent.
In \cite{LTP2022}, by employing the forward-backward envelope serving as a Lyapunov function, Latafat et al.  proved that the cluster points of the iterates generated by the popular proximal Finito/MISO algorithm are the stationary points almost surely in the fully nonconvex case. They further  established the global and linear convergence under the assumption that  $f_i, i=1,\ldots,N$ and $h$ are {\KL}  functions.
By combining the proposed algorithm for problem (\ref{eq:p}) with the variance reduced stochastic gradient estimators (uniformly defined in \cite{BLZ2021,DTLDS2021}), we study a stochastic preconditioned first-order primal-dual algorithm for solving the fully nonconvex finite-sum optimization problem (\ref{eq:p-finite-sum}).


\textbf{Contributions.}
The main contributions of this paper can be summarized as follows.
\begin{itemize}
\item[$\bullet$] We propose a preconditioned primal-dual gradient (PPDG) method  for the composite optimization problem (\ref{eq:p}). This problem poses significant challenges due to its fully nonconvex structure including the smooth nonconvex function $f$ and the nonsmooth nonconvex regularizer $h$ coupled with linear operator $A$.
 Under certain mild assumptions that  the gradient of $f$ is Lipschitz continuous, the liner operator $A$ is surjective and the convex hull of $h$ is proper, we prove that any convergent subsequence of the iterates converges to a critical point of the Lagrange function associated with problem (\ref{eq:p}). This is realized based on establishing the nonincreasing property of a properly selected Lyapunov function. With the additional {\KL}  property of the Lyapunov function, we  demonstrate the global convergence of the generated sequence of iterates.  We further derive convergence rates for the sequence, provided that the Lyapunov function has the \L{}ojasiewicz property.

\item[$\bullet$] To address the challenge of  solving problem \eqref{eq:p-finite-sum} in the fully nonconvex setting, we introduce a stochastic preconditioned primal-dual gradient (SPPDG) method, which can be viewed as a stochastic variant of PPDG. To analyze the convergence  of SPPDG, we first establish a crucial descent property related to the expectation of a Lyapunov function based on the Lagrange function of problem \eqref{eq:p-finite-sum}. Moreover, the upper bound for the conditional expectation of the subgradient of the Lyapunov function is derived. Leveraging these important auxiliary results and assuming that the generated iterates of SPPDG are bounded almost surely, we establish  the subsequence convergence  in the almost sure sense. Moreover, if the Lyapunov function is a  {\KL}  function, we prove that the whole iteration sequence possesses the finite length property and  converges almost surely to a critical point. To the best of our knowledge, such almost sure global convergence result for stochastic algorithms applied to (\ref{eq:p-finite-sum}) in the fully nonconvex setting is new.


\item[$\bullet$] We report the numerical performances of the proposed methods with  PPDG being applied to image denoising via $\ell_0$ gradient minimization, as well as SPPDG being applied to image classification using deep neural network  and a nonconvex graph-guided fused  lasso problem. Compared with the existing popular algorithms, the numerical results verify the advantages of the proposed methods.
\end{itemize}

\textbf{Organization.}
The rest of this paper is organized as follows.  In Section \ref{sec:composite}, we explore the convergence of a preconditioned primal-dual gradient method  for composite optimization problem (\ref{eq:p}).
 In Section  \ref{sec:finite-sum}, we propose a stochastic preconditioned  primal-dual gradient method for finite-sum problem \eqref{eq:p-finite-sum}, and provide a convergence analysis. Numerical experiments are presented in Section \ref{sec:Numerical Experiments} to show the effectiveness of the proposed algorithms.

\textbf{Notation.}
Let $\X$ and $\Y$ be two finite-dimensional real vector spaces equipped with standard inner products $\iprod{\cdot}{\cdot}$ and norms $\|\cdot\|=\sqrt{\iprod{\cdot}{\cdot}}$.
Let $\X^*$ and $\Y^*$ be the dual spaces of $\X$ and $\Y$, respectively.
The operator norm of a linear operator $A:\X\rightarrow\Y$ is
\[
\|A\|:=\max\{\|Ax\|:x\in\X\ \mbox{with}\ \|x\|\leq 1\}.
\]
%
Given a closed set $\cC\subset\X $ and a vector $x\in\X $,  the \emph{distance} of $x$ to $\cC$ is given by $\dist(x,\cC):=\min_{y\in\cC}\|x-y\|$. Let $f:\X \to(-\infty,+\infty]$ be a proper lower semicontinuous convex function. The extended \emph{proximal mapping} of $f$ associated to a positive definite linear operator $M$ is defined as
\[
\prox_f^{M}(y):=\argmin_{x\in\X }\left\{f(x)+\frac{1}{2}\|x-y\|_M^2\right\}.
\]
Here, $\|x\|_M^2:=\langle Mx,x\rangle$.

For an extended real-valued function $f:\X \to(-\infty,+\infty]$, let $\dom f:=\{x\in\X :f(x)<+\infty\}$ be its domain and
\[f^*(y):=\sup_{x\in\X }\{\langle y,x\rangle-f(x)\},\ y\in\X^* \]
be its \emph{conjugate function}. The conjugate function $h^*$ is always convex and lower semicontinuous, see \cite[Theorem 4.3]{Beck2017}.
When $f$ is convex, let $\partial f$ denote its \emph{subdifferential}.
A set-valued mapping $F:\X \rightrightarrows\Y $ is said to be \emph{outer semicontinuous} at $\bar{x}$, if for any $u\in\Y$ satisfying that there exist $x^k\rightarrow\bar{x}$ and $u^k\rightarrow u$ with $u^k\in F(x^k)$, it holds that $u\in F(\bar{x})$.
 From \cite[Theorem 24.4]{Rockafellar1970}, if $f$ is lower semicontinuous, proper and convex, the set-valued mapping $\partial f$ is outer semicontinuous, or equivalently, its graph is closed.
\section{PPDG for nonconvex composite optimization}\label{sec:composite}
In this section, we propose PPDG,  a preconditioned  primal-dual first-order  method  based on conjugate duality,  for solving the nonconvex composite optimization problem \eqref{eq:p}, and  establish its convergence. We begin by reviewing preliminary conjugate duality results in Subsection  \ref{subsec:optim}. The algorithmic framework of PPDG and the main assumptions are described in Subsection \ref{subsec:PPDG}. Subsection \ref{subsec:Lyapunov} devotes to derive  the descent property of a Lyapunov function. The subsequence convergence is investigated in Subsection \ref{subsec:subsequent}. Finally, in the setting of {\KL}  property, the main theoretical results regarding global convergence and convergence rates are established in Subsection \ref{subsec:KL}.


\subsection{Conjugate duality and necessary optimality}\label{subsec:optim}
Going back to problem \eqref{eq:p} and drawing upon the conjugate duality theory presented in \cite[Section 2.5.3] {BF2000},  
we have that the dual problem  of  \eqref{eq:p} is
\be\label{eq:DP}
\max_{y\in\Y^* }\left\{\inf_{x\in\X } \cL(x,y)\right\},\ \mbox{where}\ \cL(x,y):=f(x)+\iprod{y}{Ax}-h^*(y).
\ee
As per \cite[Theorem 2.158] {BF2000},  if $\inf_{x\in\X } \cL(x,y)>-\infty$ for any $y\in\Y^* $,  then $\bar{x}$ and $\bar{y}$ are optimal solutions of  \eqref{eq:p} and \eqref{eq:DP}, respectively, if and only if the following relations hold true:
\be\label{eq:optimai-precondi}
\left\{
\begin{array}{ll}
\bar{x} \in \argmin_{x\in\X }\cL(x,\bar{y}),\\
\\
0=h(A\bar{x})+h^*(\bar{y})-\langle \bar{y}, A\bar{x} \rangle.
\end{array}
\right.
\ee
From the definition of conjugate function, it follows that, if (\ref{eq:optimai-precondi}) is satisfied, we have $0\in\partial\cL(\bar{x},\bar{y})$, i.e.,
\be\label{eq:optimal-condition}
	\left\{
	\begin{array}{ll}
		0 = \nabla f(\bar{x})+A^T\bar{y},\\
		\\
		0 \in -\partial h^*(\bar{y})+A\bar{x},
	\end{array}
	\right.
\ee
where $A^T$ is the adjoint operator of $A$.
Let us denote the set of critical points of $\cL$ by
\[
\crit\cL:=\{(\bar{x},\bar{y})\in\X \times\Y^* : 0\in\partial\cL(\bar{x},\bar{y})\}.
\]
Therefore,  our primary aim of this paper is to find a pair $(\bar{x},\bar{y})$ that satisfies the necessary optimality conditions of  the  nonconvex problem \eqref{eq:p}, that is, $(\bar{x},\bar{y})\in \crit\cL$. Similarly, for the nonconvex finite-sum problem (\ref{eq:p-finite-sum}), our goal is  to obtain a critical point of
$$
\cL_s(x,y):=\frac{1}{N}\sum_{i=1}^N f_i(x)+\langle y,Ax\rangle-h^*(y).
$$

\subsection{The PPDG algorithm}\label{subsec:PPDG}
The detail of PPDG is described in Algorithm \ref{alg:PPDG}. This algorithm can be viewed as a first-order primal-dual algorithm by observing the necessary optimality conditions (\ref{eq:optimal-condition}). In specific, \eqref{eq:iter-1} is a standard gradient step associated with the first relation in (\ref{eq:optimal-condition}), and \eqref{eq:iter-2} can be regarded as a proximal gradient step coupled with the  preconditioning technique introduced in \cite{PC2011} for the second relation $0 \in -\partial h^*(\bar{y})+A\bar{x}$.
We also point out that \eqref{eq:iter-2} is equivalent to
\[y^{k+1}=\argmin_{y\in\Y^*}\left\{h^*(y)-\langle y,A(2x^{k+1}-x^k)\rangle+\frac{1}{2}\|y-y^k\|_{M}^2\right\},\]
hence the inverse of $M$ is not actually required in practice.
Moreover, in view of \eqref{eq:iter-2}, from the definition of $\prox_{h^*}^{M}$, it follows that there exists a vector $g^{k+1}\in\partial{h^*}(y^{k+1})$ such that
\be\label{eq:optimal-h}
g^{k+1}=-M(y^{k+1}-y^{k})+A(2x^{k+1}-x^{k}).
\ee
If the sequence $\{(x^k,y^k)\}$ converges to $(\bar{x},\bar{y})$, then (\ref{eq:optimal-h}) immediately implies the second relation $A\bar{x}\in\partial h^*(\bar{y})$ due to the outer semicontinuity of $\partial h^*(\cdot)$.

\begin{algorithm2e}[htp]
\caption{PPDG }
\label{alg:PPDG}
\lnlset{alg:SA1}{1}{Initialization: ~Choose an initial point $(x^0,y^0) \in \X\times\Y^*$, a constant $\alpha>0$ and a positive definite matrix $M$. }
\\ \vspace{.5ex}
\lnlset{alg:SA2}{2}\For{$k=0,1,2,\ldots$}{
\lnlset{alg:SA3}{3}{Update $x^{k},y^{k}$ as follows,
\begin{subnumcases}{\label{eq:iter}}
	x^{k+1}=x^k-\alpha(\nabla f(x^k)+A^Ty^k), \label{eq:iter-1}\\
	\notag\\
	y^{k+1}=\prox_{h^*}^{M}(y^k+M^{-1}A(2x^{k+1}-x^k)).\label{eq:iter-2}
\end{subnumcases}
} \\ \vspace{.5ex}

\lnlset{alg:SA4}{4}{Set $k\leftarrow k+1$.}
\\ }
\end{algorithm2e}

Compared with the existing first-order algorithms for nonsmooth nonconvex optimization problems, one of the main features of Algorithm \ref{alg:PPDG} is that, we compute the proximal mapping of the conjugate function $h^*$ rather than dealing with $h$ directly. This is partially motivated by the popular PDHG algorithm \cite{CP2011} for convex optimization problems. However, let us emphasize that, in the nonconvex setting there are several additional advantages. Firstly,  many related algorithms \cite{ABS2013,ABRS2010,BST2014,BST2018} involve the calculation of the proximal mapping with respect to the nonconvex function $h$, i.e.,
\[
\prox_{h}(x)\in\argmin\limits_{u\in\Y}\left\{h(u)+\frac{1}{2}\|u-x\|^2\right\},
\]
which is usually more prohibitive than computing $\prox_{h^*}^{M}(x)$ because $h^*$ is lower semicontinuous and convex. Secondly, upon observing (\ref{eq:optimal-condition}), in both the definition of $\crit\cL$  and  the later  subsequence convergence analysis of Algorithm \ref{alg:PPDG}, we do not need to  introduce  complex generalized subdifferentials of nonconvex functions, as is often required in many well-studied first-order algorithms for nonsmooth nonconvex optimization problems, see, e.g., \cite{AB2009,ABRS2010,BST2014,BCN2019}.

We would also like to remark that the update of $x^k$ in Algorithm \ref{alg:PPDG} is different from the standard PDHG algorithm, in which,
\[
x^{k+1}=\prox_{\alpha f}(x^k-\alpha A^Ty^k).
\]
In the convex setting, this can be rewritten as the following implicit step,
\[
x^{k+1}=x^k-\alpha(\nabla f(x^{k+1})+A^Ty^k),
\]
due to the fact that $x=\prox_{\alpha f}(y)$ is equivalent to $x=y-\alpha\nabla f(x)$.
The success of  PDHG for convex optimization depends on the efficient calculation of the proximal mapping of $f$, that, however, is unrealistic in our nonconvex setting. The interested readers are referred to the survey paper \cite{CP2016} for more discussion of  PDHG in the convex setting.

In order to establish the convergence of Algorithm \ref{alg:PPDG}, we impose some standard assumptions throughout this section.
\begin{assumption} \label{ass:setting}
Suppose that:
\begin{itemize}
	\item[$(i)$] The function $f$ is $L$-smooth over $\X$, i.e., $f$ is continuously differentiable and there exists a  constant $L>0$ such that for any $x, z\in\X $,
	\[\|\nabla f(x)-\nabla f(z)\|\leq L\|x-z\|.\]
	\item[$(ii)$] $\inf_{x\in\X } \cL(x,y)>-\infty$ for any $y\in\Y^* $.
\item[(iii)] The linear operator $A$ is surjective.
\item[(iv)] The convex hull of $h$ is proper.
\end{itemize}
\end{assumption}

\begin{remark}\label{rem:ass1}
Some comments on Assumption \ref{ass:setting} are in order.
\begin{enumerate}
\item[(i)] A well-known gradient descent lemma  under Assumption (i) is that,
\be\label{eq:gradient-L}
f(x^{k+1})\leq f(x^k)+\langle\nabla f(x^k),x^{k+1}-x^k\rangle+\frac{L}{2}\|x^{k+1}-x^k\|^2.
\ee
Moreover, applying \eqref{eq:iter-1} and Assumption (i), we have
\be\label{eq:lambda-x}
\begin{aligned}
\|A^T(y^{k+1}-y^k)\|=\ &\left\|\left(\frac{x^{k+1}-x^{k+2}}{\alpha}-\nabla f(x^{k+1})\right)-\left(\frac{x^k-x^{k+1}}{\alpha}-\nabla f(x^k)\right)\right\|\\
\leq\ & \left(\frac{1}{\alpha}+L\right)\|x^{k+1}-x^{k}\|+\frac{1}{\alpha}\|x^{k+2}-x^{k+1}\|.
\end{aligned}
\ee
\item[(ii)]  Assumption (ii) ensures that the sequence generated by Algorithm \ref{alg:PPDG} is well-defined. It is also indispensable in the subsequence convergence analysis (see Theorem \ref{th:usual}).
\item[(iii)] The linear operator $A$ is surjective if and only if the matrix associated with $AA^T$ is positive definite. Thus, under Assumption (iii), a natural choice of $M$ in Algorithm \ref{alg:PPDG} is the associated matrix of $\alpha AA^T$. As a special case, if $A$ is the identity operator, then we can choose $M=\alpha I$ and hence the extended proximal mapping $\prox_{h^*}^{M}(\cdot)$ reduces to the classical proximal mapping $\prox_{\frac{1}{\alpha} h^*}(\cdot)$. Moreover, under Assumption (iii), for any $y\in\Y^*$ we have
    \be\label{eq:hat-lambda}
    \hat{\lambda}\|y\|\leq \|A^Ty\|,
    \ee
    where $\hat{\lambda}:=\sqrt{\ewmin(AA^T)}$ and $\ewmin(AA^T)$  denotes the smallest eigenvalue of $AA^T$.
\item[(iv)] It is known that, without any assumption on $h$, the conjugate function $h^*$ is lower semicontinuous and convex, see, e.g., \cite[Theorem 4.3]{Beck2017}. However, in order to guarantee that $h^*$ is proper, an additional assumption is required.
 It is shown in \cite[Theorem 11.1]{RW1998} that $h^*$ is proper if Assumption (iv) is satisfied.
\end{enumerate}
\end{remark}

\subsection{A Lyapunov function}\label{subsec:Lyapunov}
As discussed previously, the primary aim of this section is to establish the convergence result that  the sequence $(x^k,y^k)$ generated by Algorithm \ref{alg:PPDG} converges to a critical point of the Lagrange function  $\cL(x,y)$. However, this is difficult to fulfill for the nonconvex composite optimization problem \eqref{eq:p} through the usual approach owing to the lack of the descent property of $\cL$. Instead, we shall  work with an auxiliary function to alleviate this difficulty.

Let us define the following Lyapunov function
\[
\sL(x,y,u,v):=\cL(x,y)-a\|x-u\|^2+b\|x-v\|^2, \ \forall x,u,v\in\X, \ y\in\Y^*.
\]
Here,  with the stepsize $\alpha$ and the Lipschitz constant $L$ of $\nabla f$,  let
\be\label{eq:a-b}
a:=\frac{\delta}{\alpha},\quad b:=\frac{1}{2\alpha}-\frac{L}{4}-\frac{\delta}{\alpha}-\frac{\alpha\delta L^2}{2}-\delta L+\frac{\alpha L^2}{4\delta},
\ee
and $\delta$ be a properly selected constant such that $a>0$ and $b>0$. Let $c$ be a constant given by
\be\label{eq:c}
c:=b-\frac{\alpha L^2}{2\delta}.
\ee
By an elementary calculation,
if we choose $\delta=1/5$, then  the choice of stepsize $\alpha\in (0, 1/3L)$ is sufficient to guarantee $c>0$, and consequently, $a>0$ and  $b>0$. Therefore, we can safely assume that $a$, $b$ and $c$ are positive in the rest of this section.

The convergence analysis of Algorithm \ref{alg:PPDG} will significantly rely on the properties of $\sL$ which shall be investigated in this subsection. For a start, we show in the following lemma that the critical point set $\crit\sL$ is closely related to $\crit\cL$.
\begin{lemma}\label{lem:relation-crit}
	Let $x,u,v\in\X $, $y\in\Y^* $. Then, $(x,y,u,v)\in\crit\sL$ is equivalent to $(x,y)\in\crit\cL$ and $u=v=x$.
\end{lemma}
\begin{proof}
In view of the definition of $\sL$,  the condition $(x,y,u,v)\in\crit\sL$ reads
\[
\begin{array}{ll}
0=\nabla_x\sL(x,y,u,v)=\nabla_x\cL(x,y)-2a(x-u)+2b(x-v),\\
0\in  \partial_y\sL(x,y,u,v)=\partial_{y}\cL(x,y),\\
0=\nabla_u\sL(x,y,u,v)=2a(x-u),\\
0=\nabla_v\sL(x,y,u,v)=2b(v-x).
\end{array}
\]
The latter two relations imply that $u=v=x$. This, together with the first two relations, leads to $(0,0)\in\partial\cL(x,y)$, which means  $(x,y)\in\crit{\cL}$. The converse is obvious.
\end{proof}

Throughout the remainder of this section, we let $M$ be the matrix associated with $\alpha AA^T$, that is, for all $y,\hat{y}\in\Y^*$,  $\iprod{\hat{y}}{My}=\iprod{\hat{y}}{\alpha AA^Ty}$. The following lemma presents a recursive relation for $\cL$.
\begin{lemma}\label{lem:cL}
Under Assumption \ref{ass:setting},  for all $k\geq 1$, it holds
\[
\cL(x^{k+1},y^{k+1})\leq  \ \cL(x^{k},y^{k})-\left(\frac{1}{\alpha}-\frac{L}{2}\right)\|x^{k+1}-x^k\|^2+\langle y^{k+1}-y^k, \alpha A(\nabla f(x^{k-1})-\nabla f(x^k))\rangle.
\]
\end{lemma}
\begin{proof}
From \eqref{eq:iter-1} and \eqref{eq:gradient-L}, it follows that
\[
f(x^{k+1})\leq f(x^k)-\langle y^k,A(x^{k+1}-x^k)\rangle-\left(\frac{1}{\alpha}-\frac{L}{2}\right)\|x^{k+1}-x^k\|^2.
\]
Taking $k=k-1$ in \eqref{eq:optimal-h} and using the convexity of $h^*$, we have
\[
-h^*(y^{k+1})\leq -h^*(y^k)+\langle y^{k}-y^{k+1},-M(y^{k}-y^{k-1})+A(2x^{k}-x^{k-1})\rangle.
\]
Combining these two inequalities, adding  $\langle y^{k+1},Ax^{k+1}\rangle$ on both sides and recalling that $\cL(x,y)=f(x)+\iprod{y}{Ax}-h^*(y)$, we obtain
\be\label{eq:a1}
\cL(x^{k+1},y^{k+1})\leq \cL(x^{k},y^{k})-\left(\frac{1}{\alpha}-\frac{L}{2}\right)\|x^{k+1}-x^k\|^2+\langle y^{k+1}-y^k, A(x^{k+1}-x^k+x^{k-1}-x^{k})+M(y^k-y^{k-1})\rangle.
\ee
 Applying \eqref{eq:iter-1} again, one has
\bee
\begin{aligned}
 &\langle y^{k+1}-y^k, A(x^{k+1}-x^k+x^{k-1}-x^{k})+M(y^k-y^{k-1})\rangle\\
 &=\langle y^{k+1}-y^k, (\alpha AA^T-M)(y^{k-1}-y^k)\rangle
+\langle y^{k+1}-y^k, \alpha A(\nabla f(x^{k-1})-\nabla f(x^k))\rangle\\
&=\langle y^{k+1}-y^k, \alpha A(\nabla f(x^{k-1})-\nabla f(x^k))\rangle.
\end{aligned}
\eee
Substituting this relation into (\ref{eq:a1}), we complete the proof.
\end{proof}

Now, we establish the following descent property of $\sL$ which will play a pivotal role in the discussion of global convergence.
\begin{lemma}\label{lem:descent}
Let  Assumption \ref{ass:setting} hold. Then,
 for all $k\geq 1$,
\be\label{ieq:descent-result}
\sL(x^{k+1},y^{k+1},x^{k+2},x^k)+c(\|x^{k+1}-x^{k}\|^2+\|x^{k}-x^{k-1}\|^2)\leq \sL(x^{k},y^{k},x^{k+1},x^{k-1}),
\ee
where $c$ is defined in (\ref{eq:c}).
\end{lemma}
\begin{proof}
From Lemma \ref{lem:cL},  we have
\be\label{eq:pre-condition-descent}
\begin{aligned}
\cL(x^{k+1},y^{k+1})&\leq  \cL(x^{k},y^{k})-\left(\frac{1}{\alpha}-\frac{L}{2}\right)\|x^{k+1}-x^k\|^2+\langle y^{k+1}-y^k, \alpha A(\nabla f(x^{k-1})-\nabla f(x^k))\rangle\\
& \leq  \cL(x^{k},y^{k})-\left(\frac{1}{\alpha}-\frac{L}{2}\right)\|x^{k+1}-x^k\|^2+\frac{\alpha\delta}{2}\|A^T(y^{k+1}-y^k)\|^2+\frac{\alpha L^2}{2\delta}\|x^{k}-x^{k-1}\|^2,
\end{aligned}
\ee
where the second inequality is deduced from the Lipschitz continuity of $\nabla f $ and the fact that $\langle x,y\rangle\leq \frac{\delta}{2}\|x\|^2+\frac{1}{2\delta}\|y\|^2$.
From \eqref{eq:lambda-x},  it follows that
\bee
\|A^T(y^{k+1}-y^k)\|^2\leq 2\left(\frac{1}{\alpha}+L\right)^2\|x^{k+1}-x^{k}\|^2+\frac{2}{\alpha^2}\|x^{k+2}-x^{k+1}\|^2.
\eee
Substituting this inequality into  \eqref{eq:pre-condition-descent} and recalling the definitions of $a,b,c$, we conclude
\bee
\begin{aligned}
	\cL(x^{k+1},y^{k+1})\leq \cL(x^{k},y^{k})-(a+b+c)\|x^{k+1}-x^k\|^2+(b-c)\|x^{k}-x^{k-1}\|^2+a\|x^{k+2}-x^{k+1}\|^2.
\end{aligned}
\eee
Rewriting this inequality gives
\bee
\begin{aligned}
&\cL(x^{k+1},y^{k+1})-a\|x^{k+1}-x^{k+2}\|^2+b\|x^{k+1}-x^{k}\|^2+c(\|x^{k+1}-x^{k}\|^2+\|x^{k}-x^{k-1}\|^2)\\
&\leq  \cL(x^{k},y^{k})-a\|x^{k}-x^{k+1}\|^2+b\|x^{k}-x^{k-1}\|^2.
\end{aligned}
\eee
The proof is completed by recalling the definition of $\sL$.
\end{proof}

Denote
\[
z^k:=(x^k,y^k,x^{k+1},x^{k-1}).
\]
Lemma \ref{lem:descent} implies that the sequence $\{\sL(z^k)\}$ is nonincreasing. Let
\[
d^k:=(\nabla_x\sL(z^k),Ax^k-g^k,\nabla_u\sL(z^k),\nabla_v\sL(z^k)),
\]
where $g^k=-M(y^{k}-y^{k-1})+A(2x^{k}-x^{k-1})$.  From (\ref{eq:optimal-h}) we have that $g^k\in\partial h^*(y^k)$, and hence $d^k\in\partial\sL(z^k)$ by the definition of $\sL$.
In the following, we derive a bound of $d^k$.
\begin{lemma}\label{lem:gradient}
Under Assumption \ref{ass:setting}, it holds that
\[
\|d^{k}\|\leq \gamma_1\|x^{k}-x^{k-1}\|+\gamma_2\|x^{k+1}-x^{k}\|,
\]
where
\[\gamma_1:=2L+4b+\frac{2}{\alpha}+(2+\alpha L)\|A\|, \  \gamma_2:=4a+\frac{1}{\alpha}+\|A\|.\]
\end{lemma}
\begin{proof}
For the first component of $d^{k}$,  we have
\bee
\begin{aligned}
&\|\nabla_x\sL(z^{k})\|\\
&=  \|\nabla f(x^{k})+A^Ty^{k}-2a(x^{k}-x^{k+1})+2b(x^{k}-x^{k-1})\|\\
&= \|\nabla f(x^{k})-\nabla f(x^{k-1})+A^T(y^{k}-y^{k-1})+\nabla f(x^{k-1})+A^Ty^{k-1}-2a(x^{k}-x^{k+1})+2b(x^{k}-x^{k-1})\|\\
&\leq (L+2b)\|x^{k}-x^{k-1}\|+2a\|x^{k}-x^{k+1}\|+\|A^T(y^{k}-y^{k-1})\|+\|\nabla f(x^{k-1})+A^Ty^{k-1}\|,
\end{aligned}
\eee
which, together with \eqref{eq:iter-1} and \eqref{eq:lambda-x}, gives that
\be\label{eq:gradient-cond1}
\|\nabla_x\sL(z^{k})\| \leq 2\left(L+b+\frac{1}{\alpha}\right)\|x^{k}-x^{k-1}\|+\left(2a+\frac{1}{\alpha}\right)\|x^{k+1}-x^{k}\|.
\ee
For the second component of $d^{k}$, by \eqref{eq:optimal-h}  we obtain
\bee
\begin{aligned}
\|Ax^{k}-g^{k}\|=\|M(y^{k}-y^{k-1})-A(x^{k}-x^{k-1})\|
\leq \alpha\|A\|\|A^T(y^{k}-y^{k-1})\|+\|A\|\|x^{k}-x^{k-1}\|,
\end{aligned}
\eee
which, together with \eqref{eq:lambda-x}, yields that
\be\label{eq:gradient-cond2}
\begin{aligned}
	\|Ax^{k}-g^{k}\|\leq (2+\alpha L)\|A\|\|x^{k}-x^{k-1}\|+\|A\|\|x^{k+1}-x^{k}\|.
\end{aligned}
\ee
For $\nabla_u\sL(z^{k})$ and $\nabla_v\sL(z^{k})$, one has
\be\label{eq:gradient-cond3}
\|\nabla_{u}\sL(z^{k})\|=2a\|x^{k+1}-x^{k}\|,\quad
\|\nabla_{v}\sL(z^{k})\|=2b\|x^{k}-x^{k-1}\|.
\ee
Combining \eqref{eq:gradient-cond1}, \eqref{eq:gradient-cond2} and \eqref{eq:gradient-cond3} together, we have
\[\|d^{k}\|\leq \left(2L+4b+\frac{2}{\alpha}+(2+\alpha L)\|A\|\right)\|x^{k}-x^{k-1}\|+\left(4a+\frac{1}{\alpha}+\|A\|\right)\|x^{k+1}-x^{k}\|.\]
The proof is completed.
\end{proof}


\subsection{Subsequence convergence}\label{subsec:subsequent}
Let $\cC$  denote the set of cluster points of the sequence $\{(x^k,y^k)\}$ generated by Algorithm \ref{alg:PPDG}. Now, we establish the subsequence convergence based on the previous lemmas concerning $\sL$. These convergence results shall be proved under the assumption that the sequence $\{(x^k,y^k)\}$ is bounded, which is a standard assumption in the global convergence analysis of nonconvex optimization algorithms, see \cite{BST2014,BST2018,BCN2019} for instance.
\begin{theorem}\label{th:usual}
	Let the sequence $\{(x^k,y^k)\}$ be bounded and  Assumption \ref{ass:setting} hold. Then,
\begin{itemize}
	\item[(i)] $\sum_{k=1}^{\infty}\|x^{k+1}-x^k\|^2<\infty$ and  $\sum_{k=1}^{\infty}\|y^{k+1}-y^k\|^2<\infty$;
	\item[(ii)] $\cC$ is a nonempty compact set and
	\[\lim_{k\to\infty}\dist((x^k,y^k),\cC)=0;\] 
	\item[(iii)] $\cC\subseteq \crit\cL$;
	\item[(iv)] $\cL$ is finite and constant on $\cC$. 	
\end{itemize}
\end{theorem}
\begin{proof}
Assumption \ref{ass:setting}(ii) implies  $\inf_k\cL(x^k,y^k)>-\infty$, which, together with the boundedness of $\{x^k\}$, leads to  $\inf_k\sL(z^k)>-\infty$. Since the sequence $\{\sL(z^k)\}$ is nonincreasing (cf. Lemma \ref{lem:descent}) and bounded from below, $\sL(z^k)$ converges to a finite value denoted by $\bar{\sL}$. Summing \eqref{ieq:descent-result} over $k=1,\ldots,n$ yields that
\bee
c\sum_{k=1}^n (\|x^{k+1}-x^{k}\|^2+\|x^{k}-x^{k-1}\|^2)\leq \sL(z^1)-\sL(z^{n+1}).
\eee
Let $n\to\infty$,  by the convergence of $\{\sL(z^k)\}$ we have
\[
\sum_{k=1}^{\infty}\|x^{k+1}-x^k\|^2<\infty.
\]	
This, together with  \eqref{eq:lambda-x} and (\ref{eq:hat-lambda}), further gives
\[\sum_{k=1}^{\infty}\|y^{k+1}-y^k\|^2<\infty.
\]
Item (i) is derived. Moreover, it further implies
\be\label{eq:x-x}
\lim_{k\to\infty}\|x^{k+1}-x^{k}\|= 0 \text{ and }  \lim_{k\to\infty}\|y^{k+1}-y^{k}\|=0.
\ee

 The compactness of $\cC$ follows the proof of \cite[Lemma 5 (iii)]{BST2014}. Since the sequence $\{(x^k,y^k)\}$ is bounded, thus $\cC$ is nonempty and for any $(\bar{x},\bar{y})\in\cC$ there exists a subsequence $\{(x^{k_q},y^{k_q})\}$ of $\{(x^k,y^k)\}$ such that
\be\label{eq:limit-point}
\lim_{q\to\infty}\|x^{k_q}-\bar{x}\|= 0, \ \lim_{q\to\infty}\|y^{k_q}-\bar{y}\|=0.
\ee
By the definition of the distance function, we have
\[
\dist((x^k,y^k),\cC)\leq \|x^k-\bar{x}\|+\|y^k-\bar{y}\|\leq \|x^k-x^{k_q}\|+\|x^{k_q}-\bar{x}\|+\|y^k-y^{k_q}\|+\|y^{k_q}-\bar{y}\|.
\]
Combining this inequality with \eqref{eq:x-x} and \eqref{eq:limit-point}, we obtain the result that $\dist((x^k,y^k),\cC)$ converges to $0$ and hence item (ii) holds.

For item (iii), it is sufficient to prove $(\bar{x},\bar{y})\in \crit\cL$ for any $(\bar{x},\bar{y})\in\cC$. Let $\bar{z}:=(\bar{x},\bar{y},\bar{x},\bar{x})$.
Noting that $z^{k_q}\rightarrow\bar{z}$, $d^{k_q}\in\partial \sL(z^{k_q})$ and $d^{k_q}\rightarrow 0$ by Lemma \ref{lem:gradient}, we have from the outer semicontinuity of $\partial \sL$ that $0\in\partial \sL(\bar{z})$, i.e., $(\bar{x},\bar{y},\bar{x},\bar{x})\in\crit\sL$. Therefore, from Lemma \ref{lem:relation-crit}, it follows that $(\bar{x},\bar{y})\in\crit\cL$.

Recall from Remark \ref{rem:ass1} (iv)  that the conjugate function $h^*$ is proper, lower semicontinuous and convex.  Thus,  $h^*$ is continuous over its domain $\dom h^*$, as demonstrated in \cite[Theorem 2.22]{Beck2017}. Therefore, $\cL$ is continuous over $\X \times\dom h^*$
and hence
\[
\lim_{q\to\infty}\cL(x^{k_q},y^{k_q})=\cL(\bar{x},\bar{y}),
\]
which further implies
\be\label{eq:continuous-L}
\lim_{q\to\infty}\sL(z^{k_q})=\lim_{q\to\infty}(\cL(x^{k_q},y^{k_q})-a\|x^{k_q}-x^{k_q+1}\|^2+b\|x^{k_q}-x^{k_q-1}\|^2)=\cL(\bar{x},\bar{y})=\sL(\bar{z}).
\ee
In the proof of (i), we have shown that
\[
\lim_{k\to\infty} \sL(z^k)= \bar{\sL},
\]
which, together with \eqref{eq:continuous-L}, implies  $\cL(\bar{x},\bar{y})=\bar{\sL}$. Since $(\bar{x},\bar{y})$ is arbitrarily chosen in $\cC$, item (iv) is obtained.
\end{proof}


\subsection{Global convergence and rates under {\KL}  assumption}\label{subsec:KL}
In this subsection, we will establish the global convergence and convergence rates of Algorithm \ref{alg:PPDG} in the context of {\KL}  property, which has been extensively studied in recent years for  the convergence of algorithms for nonconvex optimization, see, e.g., \cite{ABRS2010,ABS2013,BST2014,BLPPR2017,BST2018,BN2020,LMQ2021}.

Given a proper lower semicontinuous function $f$ and real numbers $a,b$, let us denote $[a<f<b]:=\{x\in\X :a<f(x)<b\}$.
\begin{definition}\label{def:kl}
A proper lower semicontinuous function $f:\X \to(-\infty,+\infty]$ is said to have the \emph{Kurdyka-\L{}ojasiewicz (KL) property} at $\bar{x}\in\dom\partial f:=\{x\in\X :\partial f(x)\neq\emptyset\}$ if there exist $\eta\in(0,+\infty]$, a neighborhood $U$ of $\bar{x}$ and a continuous concave function $\varphi:[0,\eta)\to[0,+\infty)$ such that
\begin{itemize}
	\item[(i)] $\varphi(0)=0$;
	\item[(ii)] $\varphi$ is continuously differentiable and $\varphi'>0$ on $(0,\eta)$;
	\item[(iii)] for all $x\in U\cap [0<f-f(\bar{x})<\eta]$, the following {\KL}  inequality holds
\be\label{eq:KL-ineq}
	\varphi'(f(x)-f(\bar{x}))\cdot\dist(0,\partial f(x))\geq 1.
\ee
\end{itemize}
\end{definition}

A proper lower semicontinuous function $f$, which has the {\KL}  property at every point of $\dom\partial f$, is called a \emph{{\KL}  function}.  When $\varphi(s)=\sigma s^{1-\theta}$, $\sigma$ is a positive constant and $\theta\in[0,1)$,  $f$ is said to satisfy the \emph{\L{}ojasiewicz property} with  \emph{exponent} $\theta$. 	

\begin{remark}
It is known that the {\KL}  property is automatically satisfied at any noncritical point $x\in\X $ with a concave function $\varphi(s)=\sigma s$ (see \cite[Section 3.2]{ABRS2010}).
A very wide class of functions, such as nonsmooth semialgebraic functions, real subanalytic functions, and functions definable in an $o$-minimal structure, satisfy the {\KL}  property.  In particular, for problem (\ref{eq:p}),  $\sL$ is considered a {\KL}  function if $f$ and $h$ are semialgebraic (or, $f$ is semialgebraic and  $h^*$ satisfies a growth condition, see  \cite[Section 5]{BST2014}). We refer the readers to \cite{AB2009, ABRS2010, BST2014, DDKL2020} for more properties and examples of {\KL}  functions.

\end{remark}


We now establish the global convergence  of Algorithm \ref{alg:PPDG}.
\begin{theorem}\label{th:converge}
Suppose that  $\sL$ is a {\KL}  function. Let Assumption \ref{ass:setting} hold and the sequence $\{(x^k,y^k)\}$ generated by Algorithm \ref{alg:PPDG} be bounded, then $(x^k,y^k)$ converges to a critical point of $\cL$ and
\[
\sum_{k=1}^{\infty}\|x^{k+1}-x^k\|<\infty,\quad \sum_{k=1}^{\infty}\|y^{k+1}-y^k\|<\infty.
\]
\end{theorem}
\begin{proof}
In the proof of Theorem \ref{th:usual}, it has been shown that
\be\label{eq:cL-continuous}
\lim_{k\to\infty}\sL(z^{k})=\bar{\sL},
\ee
where $\bar{\sL}$ is the constant value of $\cL$ over $\cC$.

If there exists a number $l_0>0$ such that $\sL(z^{l_0})=\bar{\sL}$, which, together with Lemma \ref{lem:descent}, implies that  $\sL(z^k)=\bar{\sL}$  and $x^{k}=x^{k+1}$ for any $k\geq l_0$. By (\ref{eq:lambda-x}), we have $y^{k}=y^{k+1}$ for any $k\geq l_0$. Thus,  $(x^k,y^k)=(x^{k+1},y^{k+1})$ for any $k\geq l_0$, which proves the claim.

Otherwise, since the sequence $\{\sL(z^k)\}$ is nonincreasing by Lemma \ref{lem:descent}, it follows that
$\sL(z^k)>\bar{\sL}$  for any $k>0$.
Relation \eqref{eq:cL-continuous} ensures that for any $\eta>0$, there exists an integer $l_1>0$ such that
 \[
 \sL(z^k)<\bar{\sL}+\eta
 \]
 for any $k\geq l_1$. Let $\varOmega$ be the set of cluster points of $\{z^k\}$. By the same line as the proof of Theorem \ref{th:usual} (ii) and (iv), we have that the  function $\sL$ is constant on the nonempty compact set $\varOmega$ and  $\dist(z^k,\varOmega)\to 0$ as $k\to\infty$. Thus, for any $\varepsilon>0$, there exists $l_2>0$ such that for $k\geq l_2$,
\[
\dist(z^k,\varOmega)<\varepsilon.
\]
Let $K_0:=\max\{l_1,l_2\}$. By the above discussion, one has that $z^k\in \{z:\dist(z,\varOmega)<\varepsilon\}\cap [\bar{\sL}<\sL<\bar{\sL}+\eta]$ for all $k\geq K_0$.
Since $\sL$ is a {\KL}  function, from the uniformized {\KL}  property (\cite[Lemma 6]{BST2014}),
there exists a continuous concave function $\varphi$ such  that for all $k\geq K_0$,
\be\label{eq:kl-phi}
\varphi'(\sL(z^k)-\bar{\sL})\cdot \dist(0,\partial\sL(z^k))\geq 1.
\ee
Using the concavity of $\varphi$ yields that
\be\label{eq:concav-phi}
\varphi(\sL(z^{k+1})-\bar{\sL})\leq \varphi(\sL(z^k)-\bar{\sL})
+ \varphi'(\sL(z^k)-\bar{\sL})\cdot(\sL(z^{k+1})-\sL(z^k)).
\ee
 Lemma \ref{lem:gradient} implies that
\be\label{eq:bounded-gradient-variant}
\dist(0,\partial\sL(z^k))\leq\gamma(\|x^{k}-x^{k-1}\|+\|x^{k+1}-x^k\|),
\ee
where $\gamma:=\max\{\gamma_1, \gamma_2\}$.
Combining \eqref{eq:kl-phi}, \eqref{eq:concav-phi}, \eqref{eq:bounded-gradient-variant}
with Lemma \ref{lem:descent}, we obtain that  $\cM_{m,n}:=\varphi(\sL(z^{m})-\bar{\sL})-\varphi(\sL(z^{n})-\bar{\sL})$ satisfies
\bee
\cM_{k,k+1}\geq\varphi'(\sL(z^k)-\bar{\sL})\cdot(\sL(z^k)- \sL(z^{k+1}))\geq \frac{\sL(z^k)- \sL(z^{k+1})}{\dist(0,\partial\sL(z^k))}
\geq\ \frac{c(\|x^{k+1}-x^{k}\|^2+\|x^{k}-x^{k-1}\|^2)}{\gamma(\|x^{k+1}-x^k\|+\|x^{k}-x^{k-1}\|)},
\eee
which is rewritten as
\be\label{eq:a2}
\|x^{k+1}-x^{k}\|^2+\|x^{k}-x^{k-1}\|^2\leq \frac{\gamma}{c}\cM_{k,k+1}{(\|x^{k}-x^{k-1}\|+\|x^{k+1}-x^k\|)}.
\ee
This further indicates that
\bee
\|x^{k+1}-x^k\|\leq\sqrt{\frac{\gamma}{c}\cM_{k,k+1}{(\|x^{k}-x^{k-1}\|+\|x^{k+1}-x^k\|)}}\leq\frac{\gamma}{c}\cM_{k,k+1}+\frac{1}{4}(\|x^{k}-x^{k-1}\|+\|x^{k+1}-x^k\|),
\eee
which is equivalent to
\[
\|x^{k+1}-x^k\|\leq \frac{2\gamma}{c}\cM_{k,k+1}+\frac{1}{2}(\|x^{k}-x^{k-1}\|-\|x^{k+1}-x^k\|).
\]
Summing up   from $k=K_0$ to $n$ with $n>K_0$, it follows  that
\be\label{eq:sum-x}
\sum_{k=K_0}^n\|x^{k+1}-x^k\|\leq \frac{2\gamma}{c}\cM_{K_0,n+1}+\frac{1}{2}\|x^{K_0}-x^{K_0-1}\|.
\ee
Similarly, from (\ref{eq:a2}) we also have
\be\label{eq:sum-x-1}
\sum_{k=K_0}^n\|x^{k}-x^{k-1}\|\leq \frac{2\gamma}{c}\cM_{K_0,n+1}+\frac{1}{2}\|x^{n+1}-x^{n}\|.
\ee
Summing \eqref{eq:sum-x} and \eqref{eq:sum-x-1}, we obtain
\be\label{eq:sum-x-2}
\begin{aligned}
\sum_{k=K_0}^n(\|x^{k+1}-x^k\|+\|x^{k}-x^{k-1}\|)& \leq \frac{4\gamma}{c}\cM_{K_0,n+1}+\frac{1}{2}\|x^{K_0}-x^{K_0-1}\|+\frac{1}{2}\|x^{n+1}-x^{n}\|\\
& \leq \frac{4\gamma}{c}\varphi(\sL(z^{K_0})-\bar{\sL})+\frac{1}{2}\|x^{K_0}-x^{K_0-1}\|+\frac{1}{2}\|x^{n+1}-x^{n}\|,
\end{aligned}
\ee
where the second inequality follows from the fact that $\varphi>0$ over $(0,\eta)$. Let $n\to\infty$ in \eqref{eq:sum-x-2}, by the first term of \eqref{eq:x-x} we have
\[
\sum_{k=K_0}^{\infty}\|x^{k+1}-x^k\|<+\infty,
\]
which, together with  (\ref{eq:lambda-x}),  implies
\[
\sum_{k=K_0}^{\infty}\|y^{k+1}-y^k\|<+\infty.
\]
These two inequalities imply that $(x^k,y^k)$ is a Cauchy sequence by the same line of analysis as \cite[Theorem 1 (ii)]{BST2014}. Thus, the sequence $(x^k,y^k)$ converges to a limit $(\bar{x},\bar{y})$ that is a critical point of $\cL$ by Theorem \ref{th:usual} (iii).
\end{proof}


The convergence rates of the sequence $\{(x^k,y^k)\}$ in the context of \L{}ojasiewicz exponent are provided
in the following theorem which is proved in an analogous way as \cite{AB2009}.
\begin{theorem}\label{th:convergence-rate1}
Assume that the sequence $\{(x^k,y^k)\}$ is bounded 
and  $\sL$ is a {\KL}  function with the \L{}ojasiewicz exponent $\theta$. 
Let $(\bar{x},\bar{y})$ be the limit of $(x^k,y^k)$.
Then, under Assumption \ref{ass:setting} the following estimations hold:
\begin{itemize}
  \item[(i)] if $\theta=0$, the sequence $\{(x^k,y^k)\}$ converges in finite steps;
\item[(ii)] if $\theta\in(0,\frac{1}{2}]$, then there exist constants $\nu>0$, $0<\tau<1$ and a positive  integer $K$ such that  for $k\geq K$,
  \[\|x^{k}-\bar{x}\|\leq \nu\tau^{k-K}, \quad  \|y^{k}-\bar{y}\|\leq \nu'\tau^{k-K}, \]
  where $\nu':=\nu(1/\alpha+L)/\hat{\lambda}$ and $\hat{\lambda}$ is given in (\ref{eq:hat-lambda});
  \item[(iii)] if $\theta\in(\frac{1}{2},1)$, then there exist a constant $\mu>0$  and a positive integer $\bar{K}$ such that  for $k\geq \bar{K}$,
  \[\|x^{k}-\bar{x}\|\leq \mu k^{-\frac{1-\theta}{2\theta-1}}, \quad \|y^{k}-\bar{y}\|\leq \mu' k^{-\frac{1-\theta}{2\theta-1}},\]
  where $\mu':=\mu(1/\alpha+L)/\hat{\lambda}$.
\end{itemize}
\end{theorem}

\begin{proof}
Consider $\theta=0$, let $K_1:=\max\{k\in\N:x^{k+1}\neq x^k\}$. We now show that $K_1$ is a finite number.  On the contrary, we assume that $K_1$ is sufficiently
large such that (\ref{eq:kl-phi}) holds for all $k\geq K_1$. Note that $\varphi(s)=\sigma s$ when $\theta=0$, then (\ref{eq:kl-phi}) and (\ref{eq:bounded-gradient-variant}) read
\[
\gamma(\|x^{k}-x^{k-1}\|+\|x^{k+1}-x^k\|)\geq\dist(0,\partial\sL(z^k))\geq\frac{1}{\sigma},\  k\geq K_1,
\]
which, together with Lemma \ref{lem:descent} and $a^2+b^2\geq (a+b)^2/2$, yields that
\[
\sL(z^{k+1})\leq \sL(z^{k})-c(\|x^{k+1}-x^k\|^2+\|x^{k}-x^{k-1}\|^2)\leq \sL(z^{k})-\frac{c}{2\gamma^2\sigma^2}.
\]
Let $k\to\infty$. In the proof of Theorem \ref{th:usual}, it has been shown that $\lim_{k\to\infty} \sL(z^k)= \bar{\sL}=\cL(\bar{x},\bar{y})$, consequently,
\[
\cL(\bar{x},\bar{y})\leq \cL(\bar{x},\bar{y})-\frac{c}{2\gamma^2\sigma^2},
\]
which is contradictory. Therefore, $K_1$ is a finite number and $\{x^k\}$ converges in finite steps. From \eqref{eq:lambda-x} and \eqref{eq:hat-lambda}, we attain
\[
\|y^{k+1}-y^k\|\leq \frac{1/\alpha+L}{\hat{\lambda}}\|x^{k+1}-x^k\|+\frac{1}{\alpha\hat{\lambda}}\|x^{k+2}-x^{k+1}\|\leq \frac{1/\alpha+L}{\hat{\lambda}}\left(\|x^{k+1}-x^k\|+\|x^{k+2}-x^{k+1}\|\right).
\]
Hence, $\{y^k\}$ also converges in finite steps and item (i) holds.

Let $\Delta_k:=\sum_{q=k}^{\infty}\|x^{q+1}-x^q\|+\|x^{q}-x^{q-1}\|$.
The results in Theorem \ref{th:converge} state that $\Delta_k<+\infty$ for any $k\geq 1$ and the sequence $\{(x^k,y^k)\}$ converges to $(\bar{x},\bar{y})$ that is a critical point of $\cL$.  The triangle inequality implies that  $ \|x^{k}-\bar{x}\|\leq \Delta_{k}$
and
\[
\|y^k-\bar{y}\|\leq\sum_{q=k}^{\infty}\|y^{q+1}-y^{q}\|\leq \frac{1/\alpha+L}{\hat{\lambda}}\Delta_{k}.
\]
Therefore, it is sufficient to establish the estimations in (ii) and (iii) for $\Delta_k$.
If $\Delta_{k}=0$ for some $k$, it follows that $\|x^{q+1}-x^q\|=0$ for $q\geq k$ and $\{(x^k,y^k)\}$ converges in finite steps. Thus, without loss of generality we  assume $\Delta_{k}>0$ for any $k\geq 1$. 

For $\theta\in(0,1)$, noting that $\varphi(s)=\sigma s^{1-\theta}$, letting $n\to\infty$ in \eqref{eq:sum-x-2} and using (\ref{eq:kl-phi}), we have
\[
  \begin{aligned}
\Delta_{k+1}\leq\Delta_k&\leq\frac{4\gamma\sigma}{c}(\sL(z^k)-\cL(\bar{x},\bar{y}))^{1-\theta}+\frac{1}{2}\|x^{k}-x^{k-1}\|\\
  &\leq \frac{4\gamma\sigma^{\frac{1}{\theta}}}{c}((1-\theta) \dist(0,\partial\sL(z^k)))^\frac{1-\theta}{\theta}+\frac{1}{2}\|x^{k}-x^{k-1}\|
  \end{aligned}
\]
  for any $k\geq K_0$.
The above inequality, together with the definition of $\Delta_k$ and \eqref{eq:bounded-gradient-variant}, yields that
 \be\label{eq:trans}
 \begin{aligned}
 \Delta_{k+1}\leq\ & \frac{4\gamma\sigma^{\frac{1}{\theta}}}{c}[\gamma(1-\theta)(\Delta_{k}-\Delta_{k+1})]^\frac{1-\theta}{\theta}+\frac{1}{2}(\Delta_{k}-\Delta_{k+1})\\
 =\ & \gamma'(\Delta_{k}-\Delta_{k+1})^\frac{1-\theta}{\theta}+\frac{1}{2}(\Delta_{k}-\Delta_{k+1}),
 \end{aligned}
 \ee
where $\gamma':=\frac{4}{c}(\gamma\sigma)^\frac{1}{\theta}(1-\theta)^\frac{1-\theta}{\theta}$.

Consider $\theta\in(0,\frac{1}{2}]$. Noting that $0<\Delta_{k}-\Delta_{k+1}< 1$ when $k\geq K$ and $K\geq K_0$ is large enough. Then, from \eqref{eq:trans} and $\frac{1-\theta}{\theta}\geq 1$ it follows
\[
 \Delta_{k+1}\leq (\gamma'+\frac{1}{2})(\Delta_{k}-\Delta_{k+1}).
\]
By rearranging  the above inequality and setting $\tau:=(\gamma'+\frac{1}{2})/(\gamma'+\frac{3}{2})<1$, one has
\[
\Delta_{k+1}\leq \tau \Delta_{k}.
\]
 Therefore, for any $k\geq K$, it holds
\[
\Delta_{k}\leq\nu \tau^{k-K},
\]
where $\nu:=\Delta_{K}$  is a finite number. Item (ii) is derived.

Consider $\theta\in(\frac{1}{2},1)$. Let $\bar{K}\geq K_0$ be large enough such that $0<\Delta_{k}-\Delta_{k+1}< 1$ for all $k\geq \bar{K}$.
 Noting that $0<\frac{1-\theta}{\theta}<1$, we have from
\eqref{eq:trans} that
\[
\Delta_{k+1}\leq (\gamma'+\frac{1}{2})(\Delta_{k}-\Delta_{k+1})^\frac{1-\theta}{\theta}
\]
for all $k\geq\bar{K}$.
Following the same line  of the proof of \cite[(14)]{AB2009}, there exists  a constant $\mu_1>0$ such that for all $k\geq \bar{K}$,
\[
(\Delta_{k+1})^{\nu_1}-(\Delta_{k})^{\nu_1}\geq\mu_1,
\]
where $\nu_1:=(1-2\theta)/(1-\theta)<0$.
Summing up from $k=\bar{K}$ to $n$ for any $n\geq \bar{K}$ yields that
\[
(\Delta_{n})^{\nu_1}\geq (n-\bar{K})\mu_1+(\Delta_{\bar{K}})^{\nu_1},
\]
which, together with $\nu_1<0$, implies that for  any $n\geq \bar{K}$,
\[
\Delta_{n}\leq [(n-\bar{K})\mu_1+(\Delta_{\bar{K}})^{\nu_1}]^{\frac{1}{\nu_1}}\leq \mu n^{\frac{1}{\nu_1}},
\]
where $\mu$ is a positive constant. Item (iii) is obtained.
\end{proof}


\section{SPPDG for nonconvex finite-sum optimization}\label{sec:finite-sum}
In this section, we consider to solve the nonconvex  finite-sum optimization problem \eqref{eq:p-finite-sum}. By combining Algorithm \ref{alg:PPDG} with certain stochastic gradient estimator, we present a stochastic variant of PPDG, named as SPPDG,
and  establish its  almost sure convergence as well as convergence rates.

\subsection{The SPPDG algorithm}

\begin{algorithm2e}[htp]
\caption{SPPDG }
\label{alg:SPPDG}
\lnlset{alg:S1}{1}{Initialization: ~Choose an initial point $(x^0,y^0) \in \X\times\Y^*$, a constant $\alpha>0$ and a positive definite matrix $M$. }
\\ \vspace{.5ex}
\lnlset{alg:S2}{2}\For{$k=0,1,2,\ldots$}{
\lnlset{alg:S3}{3}{Update $x^{k},y^{k}$ as follows,
\begin{subnumcases}{\label{eq:iter-sto}}
	x^{k+1}=x^k-\alpha(\Sto f_k+A^Ty^k), \label{eq:iter-sto-1}\\
	\notag\\
	y^{k+1}=\prox_{h^*}^{M}(y^k+M^{-1}A(2x^{k+1}-x^k)),\label{eq:iter-sto-2}
\end{subnumcases}
where $\Sto f_k$ is a stochastic gradient estimator of $\nabla f(x^k)$.
} \\ \vspace{.5ex}
\lnlset{alg:S4}{4}{Set $k\leftarrow k+1$.}
\\ }
\end{algorithm2e}

In Algorithm \ref{alg:SPPDG}, we summarize the detail of SPPDG for the nonconvex  finite-sum optimization problem
\[
\min_{x\in \X }\ f(x)+h(Ax),
\]
where
\[
f(x):=\frac{1}{N}\sum_{i=1}^N f_i(x).
\]
In many applications, the number of components $N$ can be very large, which  makes the computation of the full gradient  $\nabla f(x)=\frac{1}{N}\sum_{i=1}^N \nabla f_i(x)$  challenging. To circumvent this difficulty, we apply the stochastic gradient estimator $\Sto f_k$ to approximate $\nabla f(x^k)$ in \eqref{eq:iter-sto-1}.
Hence, Algorithm \ref{alg:SPPDG} can be viewed as a stochastic approximate variant of Algorithm \ref{alg:PPDG}.

Let $\cF_k$ be the $\sigma$-field generated by the random variables of the first $k$ iterations of Algorithm \ref{alg:SPPDG} and $\Exp_k$ be the expectation conditioned on $\cF_k$. Clearly, the iterate $(x^k,y^k)$  is $\cF_k$-measurable since $x^k$ and $y^k$ are both dependent on the random information $\{\Sto f_0, \Sto f_1,\ldots,\Sto f_{k-1}\}$ of the first $k$ iterations.

In this paper, we will mainly focus on the variance reduced stochastic gradient estimator $\Sto f_k$ which is formally defined in \cite{BLZ2021,DTLDS2021}.
\begin{definition}\label{def:variance-reduced}
The stochastic gradient estimator $\Sto f_k$ is said to be \emph{variance reduced}, if  there exist  constants $\sigma_1, \sigma_2, \sigma_{\Lambda}>0$, $\rho\in(0,1]$, and $\cF_k$-measurable nonnegative random variables $\Lambda_1^k, \Lambda_2^k$ of the form $\Lambda_1^k=\sum_{i=1}^t(v_k^i)^2, \Lambda_2^k=\sum_{i=1}^tv_k^i$ for some nonnegative random variables $v_k^i\in\R$, such that for any $k\geq 1$:
\begin{itemize}
	\item[(i)] 	The estimator $\Sto f_k$ satisfies 
	\be\label{eq:sto-mse}
	\Exp_k[\|\Sto f_k-\nabla f(x^k)\|^2]\leq \Lambda_1^k+\sigma_1(\Exp_k[\|x^{k+1}-x^k\|^2]+\|x^k-x^{k-1}\|^2)
	\ee
and
	\be\label{eq:sto-mse-2}
	\Exp_k[\|\Sto f_k-\nabla f(x^k)\|]\leq \Lambda_2^k+\sigma_2(\Exp_k[\|x^{k+1}-x^k\|]+\|x^k-x^{k-1}\|).
	\ee
	\item[(ii)] The sequence $\{\Lambda_1^k\}$ decays geometrically
	\be\label{eq:geometric-decay}
	\Exp_k[\Lambda_1^{k+1}]\leq (1-\rho)\Lambda_1^k+\sigma_{\Lambda}(\Exp_k[\|x^{k+1}-x^k\|^2]+\|x^k-x^{k-1}\|^2).
	\ee
	\item[(iii)] If $\{x^k\}$ satisfies $\lim_{k\to\infty}\Exp[\|x^k-x^{k-1}\|^2]=0$, then $\Exp[\Lambda_1^k]\to 0$ and $\Exp[\Lambda_2^k]\to 0$ as $k\to \infty$.
\end{itemize}	
\end{definition}
\begin{remark}
A variety of popular stochastic gradient estimators satisfy the conditions in Definition \ref{def:variance-reduced}, for example, SAGA, SARAH, SAG and SVRG. Combining \eqref{eq:sto-mse} and \eqref{eq:geometric-decay}, for any $k\geq 1$ we have the following bound,
\be\label{eq:cor-variance-reduced}
\Exp_k[\|\Sto f_k-\nabla f(x^k)\|^2]\leq \frac{1}{\rho}(\Lambda_1^k- \Exp_k[\Lambda_1^{k+1}])+\kappa(\Exp_k[\|x^{k+1}-x^k\|^2]+\|x^k-x^{k-1}\|^2),
\ee
where
\[
\kappa:=\sigma_1+\frac{\sigma_{\Lambda}}{\rho}.
\]
The readers are referred to \cite{BLZ2021,DTLDS2021} for a detailed description on examples and properties of the variance reduced stochastic gradient estimator.
\end{remark}

In the rest of this section, we assume that $\Sto f_k$ in Algorithm \ref{alg:SPPDG} is some variance reduced gradient estimator satisfying the conditions of Definition \ref{def:variance-reduced}, and let $M$  be the matrix associated with $\alpha AA^T$.
We shall analyze the convergence of $(x^k,y^k)$ generated by Algorithm \ref{alg:SPPDG} under the following assumption.
\begin{assumption} \label{ass:setting-sto}
The assumption is the same as Assumption \ref{ass:setting} except that  Assumption \ref{ass:setting}(i) is replaced by:
 the functions $f_i$, $i=1,\cdots,N$ are $L$-smooth; and $\cL$ in Assumption \ref{ass:setting}(ii) is replaced by $\cL_s$ with
 \[
 \cL_s(x,y)=\frac{1}{N}\sum_{i=1}^N f_i(x)+\langle y,Ax\rangle-h^*(y).
 \]
\end{assumption}

 Unsurprisingly, we will observe that part of the convergence analysis of Algorithm \ref{alg:SPPDG} will be performed similarly with that of Algorithm \ref{alg:PPDG} in Section \ref{sec:composite}. Without any confusion, some notations in Section \ref{sec:composite}  will be repeatedly used in this section.

\subsection{Auxiliary lemmas}
Let us first define the following Lyapunov function
\be\label{eq:def-sto-aux}
\sL_s(x,y,u,v,w):=\cL_s(x,y)-a\|x-u\|^2+b\|x-v\|^2+c\|v-w\|^2,\ \forall x,u,v,w\in\X,\ y\in\Y^*.
\ee
Here, $a,b,c$ are constants given by
\[a:=e_0+\frac{2\delta_2}{\alpha}+2\delta_2\alpha\kappa,\ b:=e_0+\frac{9\alpha \kappa}{2\delta_2}+2\delta_2\alpha\kappa+\frac{\kappa}{2\delta_1}+\frac{3\alpha L^2}{2\delta_2},\ c:=\frac{3\alpha \kappa}{2\delta_2},\]
where
\be\label{eq:e}
e_0:=\frac{1}{3\alpha}-\frac{\delta_1+L}{6}-\frac{\kappa}{3\delta_1}-\frac{4\delta_2L}{3}-\frac{4\delta_2}{3\alpha}-\frac{2\delta_2\alpha L^2}{3}-\frac{\alpha L^2}{2\delta_2}-\frac{2\alpha\kappa}{\delta_2}-\frac{8\delta_2\alpha\kappa}{3}.
\ee
In addition, $\delta_1, \delta_2>0$ are properly selected constants such that $e_0>0$. 

In this subsection, we mainly aim to develop the decent property  corresponding to the Lyapunov function $\sL_s$  in expectation. 
Following the same line as the proof of Lemma \ref{lem:relation-crit}, we firstly derive the relation between $\crit\cL_s$ and $\crit\sL_s$.
\begin{lemma}\label{lem:crit-relation-sto}
For any $x,u,v,w\in\X $, $y\in\Y^* $, $(x,y,u,v,w)\in\crit\sL_s$ if and only if
$u=v=w=x$ and $(x,y)\in\crit\cL_s$.
\end{lemma}

The following lemma gives a connection between the two sequences $\{y^k\}$ and $\{x^k\}$.
\begin{lemma}\label{lem:lambda-x-sto}
Suppose Assumption \ref{ass:setting-sto} holds. Then, for $k\geq 1$,
\bee
\begin{aligned}
	\Exp_{k}[\|A^T(y^{k+1}-y^k)\|^2]& \leq 4\left(\frac{1}{\alpha^2}+\kappa\right)\Exp_{k}[\|x^{k+2}-x^{k+1}\|^2]+4\left(\left(\frac{1}{\alpha}+L\right)^2+2\kappa\right)\Exp_{k}[\|x^{k+1}-x^{k}\|^2]\\ &\quad\quad+4\kappa\|x^k-x^{k-1}\|^2+\frac{4}{\rho}\Exp_{k}[\Lambda_1^{k+1}-\Lambda_1^{k+2}]+\frac{4}{\rho}(\Lambda_1^{k}-\Exp_{k}[\Lambda_1^{k+1}]).
\end{aligned}
\eee
\end{lemma}
\begin{proof}
Using \eqref{eq:iter-sto-1} twice yields that
\be\label{eq:lambda-x-cond1-sto}
\begin{aligned}
	\|A^T(y^{k+1}-y^k)\|&=\left\|\left(\frac{x^{k+1}-x^{k+2}}{\alpha}-\Sto f_{k+1}\right)-\left(\frac{x^k-x^{k+1}}{\alpha}-\Sto f_k\right)\right\|\\
	&\leq\frac{1}{\alpha}\|x^{k+2}-x^{k+1}\|+\frac{1}{\alpha}\|x^{k+1}-x^{k}\|+\|\Sto f_{k+1}-\Sto f_k\|.	
\end{aligned}
\ee
Since $f$ is $L$-smooth, we have
\be\label{eq:L-smooth-apply}
\begin{aligned}
\|\Sto f_{k+1}-\Sto f_k\|&\leq\|\Sto f_{k+1}-\nabla f(x^{k+1})\|+\|\nabla f(x^{k+1})-\nabla f(x^k)\|+\|\Sto f_k-\nabla f(x^k)\|\\
&\leq \|\Sto f_{k+1}-\nabla f(x^{k+1})\|+\|\Sto f_k-\nabla f(x^k)\|+L\|x^{k+1}-x^{k}\|.
\end{aligned}
\ee
Substituting (\ref{eq:L-smooth-apply}) into \eqref{eq:lambda-x-cond1-sto}, one has
\bee
\begin{aligned}
	\|A^T(y^{k+1}-y^k)\|^2&\leq \frac{4}{\alpha^2}\|x^{k+2}-x^{k+1}\|^2+4\left(\frac{1}{\alpha}+L\right)^2\|x^{k+1}-x^{k}\|^2\\
	&\quad\quad+4\|\Sto f_{k+1}-\nabla f(x^{k+1})\|^2+4\|\Sto f_k-\nabla f(x^k)\|^2.	
\end{aligned}
\eee
Finally, taking conditional expectation on both sides  and using \eqref{eq:cor-variance-reduced}, we derive the claim.
\end{proof}	

For the sake of simplicity, define
\[
z^k:=(x^k,y^k,x^{k+1},x^{k-1},x^{k-2}).
\]
Similar to the process of convergence analysis in Section \ref{sec:composite}, we establish the following critical lemma on the descent property of the Lyapunov function $\sL_s$.
\begin{lemma}\label{lem:sto-descent}
Let Assumption \ref{ass:setting-sto} hold. Then, for any $k\geq 1$ and  $\delta_1, \delta_2>0$,
\be\label{eq:sto-descent}
\Exp[\sL^{\Lambda}_{s,k+1}]+\Exp[e_0(\|x^{k+2}-x^{k+1}\|^2+\|x^{k+1}-x^{k}\|^2+\|x^{k}-x^{k-1}\|^2)]\leq\Exp[\sL^{\Lambda}_{s,k}],
\ee
where $e_0$ is defined in (\ref{eq:e}),
\[
\sL^{\Lambda}_{s,k}:=\sL_{s}(z^k)+e_1\Lambda_1^{k+1}+e_2\Lambda_1^{k}+e_3\Lambda_1^{k-1},
\]
and
\[
e_1:=\frac{2\delta_2\alpha}{\rho},\  e_2:=\frac{2\delta_2\alpha}{\rho}+\frac{1}{2\delta_1\rho}+\frac{3\alpha}{2\delta_2\rho},\ e_3:=\frac{3\alpha}{2\delta_2\rho}.
\]
\end{lemma}
\begin{proof}
Since the function $f$ is $L$-smooth,  we have
\bee
\begin{aligned}
f(x^{k+1})&\leq f(x^k)+\langle\nabla f(x^k),x^{k+1}-x^k\rangle+\frac{L}{2}\|x^{k+1}-x^k\|^2\\
&= f(x^k)+\langle\Sto f_k,x^{k+1}-x^k\rangle+\langle\nabla f(x^k)-\Sto f_k,x^{k+1}-x^k\rangle+\frac{L}{2}\|x^{k+1}-x^k\|^2\\
&\leq f(x^k)+\langle\frac{x^k-x^{k+1}}{\alpha}-A^Ty^k,x^{k+1}-x^k\rangle+\frac{1}{2\delta_1}\|\nabla f(x^k)-\Sto f_k\|^2+\frac{\delta_1+L}{2}\|x^{k+1}-x^k\|^2\\
&= f(x^k)-\left(\frac{1}{\alpha}-\frac{\delta_1+L}{2}\right)\|x^{k+1}-x^k\|^2+\frac{1}{2\delta_1}\|\nabla f(x^k)-\Sto f_k\|^2-\langle y^k,A(x^{k+1}-x^k)\rangle,
\end{aligned}
\eee
where the second inequality is deduced from \eqref{eq:iter-sto-1} and $\langle x,y\rangle\leq \frac{\delta_1}{2}\|x\|^2+\frac{1}{2\delta_1}\|y\|^2$ for any $\delta_1>0$. Together with the convexity of $h^*$ and $g^k\in\partial h^*(y^k)$, it further indicates  	
\bee
\begin{aligned}
&f(x^{k+1})-h^*(y^{k+1})+\langle y^{k+1},Ax^{k+1}\rangle\\
&\leq f(x^k)-h^*(y^k)+\langle y^{k},Ax^{k}\rangle-\langle y^{k},Ax^{k}\rangle+\langle y^{k+1},Ax^{k+1}\rangle-\langle y^k,A(x^{k+1}-x^k)\rangle+\langle y^{k}-y^{k+1},g^k\rangle\\
&\quad\quad-\left(\frac{1}{\alpha}-\frac{\delta_1+L}{2}\right)\|x^{k+1}-x^k\|^2+\frac{1}{2\delta_1}\|\nabla f(x^k)-\Sto f_k\|^2\\
&= f(x^k)-h^*(y^k)+\langle y^{k},Ax^{k}\rangle+\langle y^{k+1}-y^k,Ax^{k+1}-g^k\rangle\\
&\quad\quad-\left(\frac{1}{\alpha}-\frac{\delta_1+L}{2}\right)\|x^{k+1}-x^k\|^2+\frac{1}{2\delta_1}\|\nabla f(x^k)-\Sto f_k\|^2.
\end{aligned}
\eee
Recalling the definition of $\cL_s$ and substituting \eqref{eq:optimal-h} (let $k+1=k$) into the above inequality, one has
\bee
\begin{aligned}
\cL_s(x^{k+1},y^{k+1})&\leq \cL_s(x^{k},y^{k})+\langle y^{k+1}-y^k,A(x^{k+1}-x^k+x^{k-1}-x^{k})+M(y^k-y^{k-1})\rangle\\
&\quad\quad-\left(\frac{1}{\alpha}-\frac{\delta_1+L}{2}\right)\|x^{k+1}-x^k\|^2+\frac{1}{2\delta_1}\|\nabla f(x^k)-\Sto f_k\|^2.
\end{aligned}
\eee
Using \eqref{eq:iter-sto-1} and the fact that $\langle x,y\rangle\leq\frac{\delta_2}{2}\|x\|^2+\frac{1}{2\delta_2}\|y\|^2$ for the second term of the right-hand side, and  letting $M=\alpha AA^T$, we have
\[
	\begin{aligned}
	\cL_s(x^{k+1},y^{k+1})&\leq \cL_s(x^{k},y^{k})-\left(\frac{1}{\alpha}-\frac{\delta_1+L}{2}\right)\|x^{k+1}-x^k\|^2+
	\frac{\delta_2\alpha}{2}\|A^T(y^{k+1}-y^k)\|^2\\
	&\quad\quad+\frac{1}{2\delta_1}\|\nabla f(x^k)-\Sto f_k\|^2+\frac{\alpha}{2\delta_2}\|\Sto f_{k-1}-\Sto f_k\|^2,
    \end{aligned}	
\]
which, together with \eqref{eq:L-smooth-apply} (take $k=k-1$), yields that
\be\label{eq:c1}
	\begin{aligned}
&	\cL_s(x^{k+1},y^{k+1})\\
&\leq \cL_s(x^{k},y^{k})-\left(\frac{1}{\alpha}-\frac{\delta_1+L}{2}\right)\|x^{k+1}-x^k\|^2+
	\frac{\delta_2\alpha}{2}\|A^T(y^{k+1}-y^k)\|^2+\frac{3\alpha L^2}{2\delta_2}\|x^k-x^{k-1}\|^2\\
	&\quad\quad+\left(\frac{1}{2\delta_1}+\frac{3\alpha}{2\delta_2}\right)\|\nabla f(x^k)-\Sto f_k\|^2+\frac{3\alpha}{2\delta_2}\|\nabla f(x^{k-1})-\Sto f_{k-1}\|^2.
    \end{aligned}
\ee
Taking conditional expectation on both sides of (\ref{eq:c1}), and applying Lemma \ref{lem:lambda-x-sto} as well as \eqref{eq:cor-variance-reduced}, we have
	\bee
	\begin{aligned}
		&\Exp_{k-1}[\cL_s(x^{k+1},y^{k+1})]\\		&\leq\Exp_{k-1}[\cL_s(x^{k},y^{k})]-\left(\frac{1}{\alpha}-\frac{\delta_1+L}{2}-\frac{\kappa}{2\delta_1}-\frac{3\alpha\kappa}{2\delta_2}-2\delta_2\alpha\left(\left(\frac{1}{\alpha}+L\right)^2+2\kappa\right)\right)\Exp_{k-1}[\|x^{k+1}-x^k\|^2]\\
		&\quad+2\delta_2\alpha \left(\frac{1}{\alpha^2}+\kappa\right)\Exp_{k-1}[\|x^{k+2}-x^{k+1}\|^2]+\left(2\delta_2\alpha\kappa+\frac{\kappa}{2\delta_1}+\frac{3\alpha (L^2+2\kappa)}{2\delta_2}\right)\Exp_{k-1}[\|x^k-x^{k-1}\|^2]\\	&\quad+\frac{3\alpha\kappa}{2\delta_2}\|x^{k-1}-x^{k-2}\|^2+\frac{2\delta_2\alpha}{\rho}\Exp_{k-1}[\Lambda_1^{k+1}-\Lambda_1^{k+2}]+\left(\frac{2\delta_2\alpha}{\rho}+\frac{1}{2\delta_1\rho}+\frac{3\alpha}{2\delta_2\rho}\right)\Exp_{k-1}[\Lambda_1^{k}-\Lambda_1^{k+1}]\\
		&\quad+\frac{3\alpha}{2\delta_2\rho}(\Lambda_1^{k-1}-\Exp_{k-1}[\Lambda_1^{k}]).
	\end{aligned}
	\eee
Therefore, taking expectation on both sides implies that
\bee
\begin{aligned}
\Exp[\cL_s(x^{k+1},y^{k+1})]
&\leq\Exp[\cL_s(x^{k},y^{k})]-e_4\Exp[\|x^{k+1}-x^k\|^2]+e_5\Exp[\|x^{k+2}-x^{k+1}\|^2]+e_6\Exp[\|x^k-x^{k-1}\|^2]\\
&\quad\quad+e_7\Exp[\|x^{k-1}-x^{k-2}\|^2]+e_1\Exp[\Lambda_1^{k+1}-\Lambda_1^{k+2}]+e_2\Exp[\Lambda_1^{k}-\Lambda_1^{k+1}]+e_{3}\Exp[\Lambda_1^{k-1}-\Lambda_1^{k}],
\end{aligned}
\eee
where
	\[
	\begin{array}{lll}
e_1=\frac{2\delta_2\alpha}{\rho},\   e_2=\frac{2\delta_2\alpha}{\rho}+\frac{1}{2\delta_1\rho}+\frac{3\alpha}{2\delta_2\rho},\ e_{3}=\frac{3\alpha}{2\delta_2\rho},\ e_4=\frac{1}{\alpha}-\frac{\delta_1+L}{2}-\frac{\kappa}{2\delta_1}-\frac{3\alpha\kappa}{2\delta_2}-2\delta_2\alpha((\frac{1}{\alpha}+L)^2+2\kappa),\\[5pt]
	 e_5=2\delta_2\alpha(\frac{1}{\alpha^2}+\kappa),\
	 e_6=2\delta_2\alpha\kappa+\frac{\kappa}{2\delta_1}+\frac{3\alpha (L^2+2\kappa)}{2\delta_2}, \ e_7=\frac{3\alpha\kappa}{2\delta_2}.
	\end{array}
	 \]
Recalling the definitions of $a,b,c,e_0$, we have $e_0=\frac{1}{3}(e_4-e_5-e_6-e_7)$ and $a=e_0+e_5, b=e_0+e_6+e_7, c=e_7$, and thus
\bee
\begin{aligned}
	\Exp[\sL^{\Lambda}_{s,k+1}+e_0(\|x^{k+2}-x^{k+1}\|^2+\|x^{k+1}-x^{k}\|^2+\|x^{k}-x^{k-1}\|^2)]\leq\Exp[\sL^{\Lambda}_{s,k}].
\end{aligned}
\eee
This proof is completed.
\end{proof}

\begin{remark}
As stated previously, the constant $e_0$ is guaranteed to be positive  through a careful selection of  $\delta_1,\delta_2$ and the stepsize $\alpha$. For example, let $\delta_1=1$, $\delta_2=\frac{1}{6}$ and  $\alpha\in(0,1/2(3+7L+6\kappa))$, we have $e_0>0$ by an straightforward  calculation. Thus, Lemma \ref{lem:sto-descent} indicates that the sequence $\{\Exp[\sL^{\Lambda}_{s,k}]\}$ is nonincreasing.
\end{remark}

Define
\be\label{eq:dk}
d^k:=(d_1^k,d_2^k,d_3^k,d_4^k,d_5^k)
\ee
 with
\[
\begin{array}{ll}
d_1^k:=\frac{1}{N}\sum_{i=1}^N\nabla f_i(x^k)+A^Ty^k-2a(x^k-x^{k+1})+2b(x^k-x^{k-1}),\\[8pt]
d_2^k:=Ax^k+M(y^k-y^{k-1})-A(2x^k-x^{k-1}),\\[8pt]
d_3^k:=-2a(x^{k+1}-x^k),\ d_4^k:=2b(x^{k-1}-x^k)+2c(x^{k-1}-x^{k-2}),\ d_5^k:=2c(x^{k-2}-x^{k-1}).
\end{array}
\]
Noting that $g^k=-M(y^k-y^{k-1})+A(2x^k-x^{k-1})\in\partial h^*(y^k)$ from (\ref{eq:optimal-h}), we can easily check that $d^k\in\partial \sL_{s}(z^k)$. In the following lemma, we derive a bound of $d^k$.
\begin{lemma}\label{lem:gradient-bounded-sto}
Let Assumption \ref{ass:setting-sto} be satisfied.
It holds that
\be\label{eq:gradient-bounded-sto}
\Exp_k[\|d^{k}\|^2]\leq \gamma_3\Exp_k[\|x^{k+1}-x^{k}\|^2]+\gamma_4(\|x^k-x^{k-1}\|^2+\|y^k-y^{k-1}\|^2) +\gamma_5\|x^{k-1}-x^{k-2}\|^2+3\Lambda_1^k,
\ee
where
 $\gamma_3:=\frac{3}{\alpha^2}+\frac{12a}{\alpha}+16a^2+3\sigma_1$, $\gamma_4:=20b^2+3\sigma_1+2\|A\|^2+2\|M\|^2$ and $\gamma_5:=12c^2$.
\end{lemma}
\begin{proof}
It is sufficient to bound the five components of $d^k$.
Firstly, from \eqref{eq:iter-sto-1} we have
\bee
\begin{aligned}
\|d_1^k\|^2&=\|\nabla f(x^{k})+A^Ty^{k}-2a(x^{k}-x^{k+1})+2b(x^{k}-x^{k-1})\|^2\\
&\leq\left(\|\nabla f(x^{k})-\Sto f_k\|+\|A^Ty^{k}+\Sto f_k\|+2a\|x^{k+1}-x^{k}\|+2b\|x^{k}-x^{k-1}\|\right)^2\\
&= \left(\|\nabla f(x^{k})-\Sto f_k\|+\left(\frac{1}{\alpha}+2a\right)\|x^{k+1}-x^k\|+2b\|x^{k}-x^{k-1}\|\right)^2\\
&= 3\|\nabla f(x^{k})-\Sto f_k\|^2+3\left(\frac{1}{\alpha}+2a\right)^2\|x^{k+1}-x^k\|^2+12b^2\|x^{k}-x^{k-1}\|^2.
\end{aligned}
\eee
Taking conditional expectation on both sides and using \eqref{eq:sto-mse} yield that
\[
\Exp_k[\|d_1^k\|^2]\leq 3\left(\frac{1}{\alpha^2}+\frac{4a}{\alpha}+4a^2+\sigma_1\right)\Exp_k[\|x^{k+1}-x^k\|^2]+3(4b^2+\sigma_1)\|x^k-x^{k-1}\|^2+3\Lambda_1^k.
\]
For the other four components of $d^k$, it follows that
\[
\begin{array}{ll}
\Exp_k[\|d_2^k\|^2]= \|M(y^{k}-y^{k-1})-A(x^{k}-x^{k-1})\|^2
\leq 2\|M\|^2\cdot\|y^{k}-y^{k-1}\|^2+2\|A\|^2\cdot\|x^{k}-x^{k-1}\|^2,
\\[8pt]
\Exp_k[\|d_3^k\|^2]= 4a^2\Exp_k[\|x^{k+1}-x^{k}\|^2],\\[8pt]
\Exp_k[\|d_4^k\|^2]\leq 8b^2\|x^{k}-x^{k-1}\|^2+8c^2\|x^{k-1}-x^{k-2}\|^2,\\[8pt]
\Exp_k[\|d_5^k\|^2]=4c^2\|x^{k-1}-x^{k-2}\|^2.
\end{array}
\]
Combining these results, we derive the conclusion.
\end{proof}

\subsection{Convergence analysis}
Now, with the help of the auxiliary lemmas established in the previous subsection, we demonstrate that  the iterates $\{(x^k,y^k)\}$ of Algorithm \ref{alg:SPPDG} exhibit  the following elementary convergence property  under the assumption that $\{(x^k,y^k)\}$ is  bounded almost surely (for short, a.s.). This assumption is also used in \cite{Davis2016,LMQ2021} for studying stochastic optimization algorithms.
\begin{proposition}\label{prop:finite-length}
Suppose that $\{(x^k,y^k)\}$ is bounded almost surely. Then, under Assumption \ref{ass:setting-sto},
\[
  \sum_{k=0}^{\infty}\|x^{k+1}-x^k\|^2<\infty\ \mbox{a.s.}\quad \text{and}\quad  \sum_{k=0}^{\infty}\|y^{k+1}-y^k\|^2<\infty\ \mbox{a.s.}
  \]
\end{proposition}
\begin{proof}
Summing (\ref{eq:sto-descent}) over $k=1,\ldots,n$ yields that
\be\label{eq:i-1sto}
e_0\sum_{k=1}^n\Exp[\|x^{k+2}-x^{k+1}\|^2+\|x^{k+1}-x^{k}\|^2+\|x^{k}-x^{k-1}\|^2]\leq \Exp[\sL^{\Lambda}_{s,1}]-\Exp[\sL^{\Lambda}_{s,n+1}].
\ee	
From Assumption \ref{ass:setting-sto}, $\cL_s$ is bounded from below, which, together with the almost sure boundedness of $\{x^k\}$, ensures that $\Exp[\sL^{\Lambda}_{s,k}]$ is bounded from below. Since $\Exp[\sL^{\Lambda}_{s,k}]$ is nonincreasing (Lemma \ref{lem:sto-descent}),  thus $\Exp[\sL^{\Lambda}_{s,k}]$ converges to a finite value. Then, from (\ref{eq:i-1sto}) it follows that
\be\label{eq:i-2sto}
\sum_{k=1}^{\infty}\Exp[\|x^{k}-x^{k-1}\|^2]=\sum_{k=0}^{\infty}\Exp[\|x^{k+1}-x^k\|^2]<\infty.
\ee
This also implies that
\be\label{eq:i-coro}
\lim_{k\to\infty}\Exp[\|x^{k+1}-x^k\|^2] = 0
\ee
and
\be\label{eq:i-coro1}
 \sum_{k=0}^{\infty}\|x^{k+1}-x^k\|^2<\infty\quad \mbox{a.s.}
\ee
Furthermore, from item (iii) in Definition \ref{def:variance-reduced}, it follows that
\be\label{eq:gamma-conver}
\lim_{k\to\infty}\Exp[\Lambda_1^k] = 0 \text{ and } \lim_{k\to\infty}\Exp[\Lambda_2^k] = 0.
\ee
By  Lemma \ref{lem:lambda-x-sto} and (\ref{eq:hat-lambda}), we have
\be\label{eq:i-4sto}
\begin{aligned}
\Exp_{k}[\|y^{k+1}-y^k\|^2]&\leq\frac{4}{{\hat{\lambda}}^2}\left(\left(\frac{1}{\alpha}+L\right)^2+2\kappa\right) \left(\Exp_{k}[\|x^{k+2}-x^{k+1}\|^2+\|x^{k+1}-x^{k}\|^2]+\|x^k-x^{k-1}\|^2\right)\\
&\quad\quad+\frac{4}{{\hat{\lambda}}^2\rho}\left(\Exp_{k}[\Lambda_1^{k+1}-\Lambda_1^{k+2}]+(\Lambda_1^{k}-\Exp_{k}[\Lambda_1^{k+1}])\right).
\end{aligned}
\ee
Taking expectation on both sides and summing it over $k=1,\ldots,n$, we have
\be\label{eq:a4}
\begin{aligned}
\sum_{k=1}^{n}\Exp[\|y^{k+1}-y^k\|^2]&\leq\frac{4}{{\hat{\lambda}}^2}\left(\left(\frac{1}{\alpha}+L\right)^2+2\kappa\right)\sum_{k=1}^{n}\Exp[\|x^{k+2}-x^{k+1}\|^2+\|x^{k+1}-x^{k}\|^2+\|x^k-x^{k-1}\|^2]\\
	&\quad\quad+\frac{4}{{\hat{\lambda}}^2\rho}\Exp[\Lambda_1^{1}+\Lambda_1^{2}-\Lambda_1^{n+1}-\Lambda_1^{n+2}].
\end{aligned}
\ee
Let $n\to\infty$, then using \eqref{eq:i-2sto} and \eqref{eq:gamma-conver}, one has
 \be\label{eq:i-3sto}
 \sum_{k=0}^{\infty}\Exp[\|y^{k+1}-y^k\|^2]<\infty,
 \ee
 which implies that
 \be\label{eq:i-coro2}
 \lim_{k\to\infty}\Exp[\|y^{k+1}-y^k\|^2] = 0
 \ee
 and
 \be\label{eq:i-coro4}
\sum_{k=0}^{\infty}\|y^{k+1}-y^k\|^2<\infty \quad\mbox{a.s.}
 \ee
The proof is completed.
\end{proof}


\begin{remark}\label{rem:event}
Due to the random nature of $\Sto f_k$, we can define a suitable sample space $\Omega$ based on the structure of Algorithm \ref{alg:SPPDG}. Then, the sequence $\{(x^k(\omega),y^k(\omega))\}$ with each sample $\omega\in\Omega$ corresponds to the generated iterates by a single run of Algorithm \ref{alg:SPPDG}. The sample space $\Omega$ can be equipped with a $\sigma$-algebra $\cF$ and a probability measure $\mathbb{P}$ to form a probability space $(\Omega,\cF,\mathbb{P})$.
 Consequently, the assumption that  $\{(x^k,y^k)\}$ is  bounded almost surely implies that there exists an event $\cA$ with $\mathbb{P}(\cA)=1$ such that the sequence  $\{(x^k(\omega),y^k(\omega))\}$ is bounded for every $\omega\in\cA$.
 \end{remark}

The following theorem establishes subsequence convergence by showing that any cluster point of the sequence $\{(x^k,y^k)\}$  is a critical point of $\cL_s$ with probability $1$.
\begin{theorem}\label{th:usual-sto}
	Let Assumption \ref{ass:setting-sto} be satisfied and $\{(x^k,y^k)\}$ be  bounded almost surely. Then, there exists an event $\cA$ with measure $1$ such that, for all $\omega\in\cA$, the following statements hold:
	\begin{itemize}
	\item[(i)]  the set $\cC_{\omega}$ containing all cluster points of $\{(x^k(\omega),y^k(\omega))\}$ is nonempty and compact, and \[\dist((x^k(\omega),y^k(\omega)),\cC_{\omega})\to 0;\]
	\item[(ii)] $\cC_{\omega}\subseteq \crit\cL_s$;
	\item[(iii)]  $\cL_s$ is finite and constant on  $\cC_{\omega}$.	
	\end{itemize}
\end{theorem}
\begin{proof}
Because  $\{(x^k,y^k)\}$ is  bounded almost surely, from Remark \ref{rem:event} there exists an event $\cA$ with measure 1 such that the sequence  $\{(x^k(\omega),y^k(\omega))\}$ is bounded for any fixed $\omega\in\cA$. Hence, the set $\cC_{\omega}$ is nonempty.
For any $(\bar{x}(\omega),\bar{y}(\omega))\in\cC_{\omega}$,  there exists a subsequence $\{(x^{k_q}(\omega),y^{k_q}(\omega))\}$ of $\{(x^k(\omega),y^k(\omega))\}$ such that
 \be\label{eq:conver-subsequence}
 x^{k_q}(\omega)\to \bar{x}(\omega) \text{ and } y^{k_q}(\omega)\to \bar{y}(\omega).
 \ee
 From \eqref{eq:i-coro1} and \eqref{eq:i-coro4}, we have
\be\label{eq:i-coro3}
\lim_{k\to\infty}\|x^{k+1}(\omega)-x^k(\omega)\| = 0 \text{ and } \lim_{k\to\infty}\|y^{k+1}(\omega)-y^k(\omega)\| = 0.
\ee
Thus, we obtain that $\cC_{\omega}$  is   compact  and $\dist((x^k(\omega),y^k(\omega)),\cC_{\omega})\to 0$, by following the same line of the proof of  Theorem \ref{th:usual} (ii). Item (i) is derived.

We next prove that, for any $(\bar{x}(\omega),\bar{y}(\omega))\in\cC_{\omega}$, $\bar{z}(\omega):=(\bar{x}(\omega),\bar{y}(\omega),\bar{x}(\omega),\bar{x}(\omega),\bar{x}(\omega))\in\crit\sL_s$, i.e., $0\in\partial \sL_s(\bar{z}(\omega))$, by using the outer semicontinuity of $\partial \sL_s$. Let $$z^{k_q}(\omega):=(x^{k_q}(\omega),y^{k_q}(\omega),x^{k_q+1}(\omega),x^{{k_q}-1}(\omega),x^{{k_q}-2}(\omega)).$$
It immediately follows from \eqref{eq:conver-subsequence} that $z^{k_q}(\omega)\to \bar{z}(\omega)$. Let $d^{k_q}(\omega)$ be defined similarly to (\ref{eq:dk}) with respect to $\omega$. Then, we have $d^{k_q}(\omega)\in\partial \sL_s(z^{k_q}(\omega))$. Therefore, due to the outer semicontinuity of $\partial \sL_s$, in order to obtain $0\in\partial \sL_s(\bar{z}(\omega))$, it is sufficient to show  $d^{k_q}(\omega)\to 0$.
From Lemma \ref{lem:gradient-bounded-sto}, there exists a constant $r>0$ such that
\be\label{eq:gradient square}
\Exp[\|d^{k_q}\|^2]
\leq r(\Exp[\|x^{k_q+1}-x^{k_q}\|^2+\|x^{k_q}-x^{k_q-1}\|^2+\|x^{k_q-1}-x^{k_q-2}\|^2+\|y^{k_q}-y^{k_q-1}\|^2+\Lambda_1^{k_q}]).
\ee
By rearranging \eqref{eq:geometric-decay}, we attain
\[
\Exp[\Lambda_1^{k_q}]\leq\frac{1}{\rho}\Exp[\Lambda_1^{k_q}-\Lambda_1^{k_q+1}]+\frac{\sigma_{\Lambda}}{\rho}(\Exp[\|x^{k_q+1}-x^{k_q}\|^2]+\Exp[\|x^{k_q}-x^{k_q-1}\|^2]),
\]
 which, together with  \eqref{eq:gradient square}, yields that
\[
\begin{aligned}
\Exp[\|d^{k_q}\|^2]&\leq (r+\frac{r\sigma_{\Lambda}}{\rho})\Exp[\|x^{k_q+1}-x^{k_q}\|^2+\|x^{k_q}-x^{k_q-1}\|^2+\|x^{k_q-1}-x^{k_q-2}\|^2]\\
&\quad\quad+r\Exp[\|y^{k_q}-y^{k_q-1}\|^2]+\frac{r}{\rho} \Exp[\Lambda_1^{k_q}-\Lambda_1^{k_q+1}].
\end{aligned}
\]
Summing up  from $k_q=2$ to $\infty$ and using \eqref{eq:i-2sto}, \eqref{eq:i-3sto} and \eqref{eq:gamma-conver}, one has
\[
\sum_{k_q=2}^{\infty}\Exp[\|d^{k_q}\|^2]<\infty.
\]
Hence $d^{k_q}\to 0$ almost surely, which further implies that $d^{k_q}(\omega)\to 0$.
Thus, we finish the proof of $\bar{z}(\omega)\in\crit\sL_s$. Furthermore, we derive item (ii) by Lemma \ref{lem:crit-relation-sto}.

To prove item (iii), let us first show that $\sum_{k=1}^{\infty}W_k<\infty$ almost surely, where
\[
\begin{array}{ll}
\displaystyle W_k:=\frac{\delta_1+L}{2}\|x^{k+1}-x^k\|^2+
\frac{\delta_2\alpha}{2}\|A^T(y^{k+1}-y^k)\|^2+\frac{3\alpha L^2}{2\delta_2}\|x^k-x^{k-1}\|^2\\[10pt]
\quad\quad\quad\displaystyle+\left(\frac{1}{2\delta_1}+\frac{3\alpha}{2\delta_2}\right)\|\nabla f(x^k)-\Sto f_k\|^2+\frac{3\alpha}{2\delta_2}\|\nabla f(x^{k-1})-\Sto f_{k-1}\|^2.
\end{array}
\]
It follows from (\ref{eq:cor-variance-reduced}) that
\[
\Exp[\|\Sto f_k-\nabla f(x^k)\|^2]\leq \frac{1}{\rho}(\Exp[\Lambda_1^k]- \Exp[\Lambda_1^{k+1}])+\kappa(\Exp[\|x^{k+1}-x^k\|^2]+\Exp[\|x^k-x^{k-1}\|^2]),
\]
which, together with the facts that $\Exp[\Lambda_1^k]\to 0$  and $\sum_{k=0}^{\infty}\Exp[\|x^{k+1}-x^k\|^2]<\infty$ from (\ref{eq:gamma-conver}) and (\ref{eq:i-2sto}), indicates that
\[
\sum_{k=1}^{\infty}\Exp[\|\Sto f_k-\nabla f(x^k)\|^2]<\infty,
\]
and hence $\sum_{k=1}^{\infty}\|\Sto f_k-\nabla f(x^k)\|^2<\infty$ almost surely. Therefore, using Proposition \ref{prop:finite-length} we obtain that $\sum_{k=1}^{\infty}W_k<\infty$ almost surely.

In a completely analogous way to (\ref{eq:c1}), we can prove that for any fixed $\omega\in\cA$,
\[
	\begin{aligned}
&	\cL_s(x^{k+1}(\omega),y^{k+1}(\omega))\\
&\leq \cL_s(x^{k}(\omega),y^{k}(\omega))+\frac{\delta_1+L}{2}\|x^{k+1}(\omega)-x^k(\omega)\|^2+
	\frac{\delta_2\alpha}{2}\|A^T(y^{k+1}(\omega)-y^k(\omega))\|^2\\
	&\quad\quad+\frac{3\alpha L^2}{2\delta_2}\|x^k(\omega)-x^{k-1}(\omega)\|^2+\left(\frac{1}{2\delta_1}+\frac{3\alpha}{2\delta_2}\right)\|\nabla f(x^k(\omega))-\Sto f_k(\omega)\|^2\\
&\quad\quad+\frac{3\alpha}{2\delta_2}\|\nabla f(x^{k-1}(\omega))-\Sto f_{k-1}(\omega)\|^2\\
&=\cL_s(x^{k}(\omega),y^{k}(\omega))+W_k(\omega).
    \end{aligned}
\]
Because $\sum_{k=1}^{\infty}W_k<\infty$ almost surely, we have $\sum_{k=1}^{\infty}W_k(\omega)<\infty$. Therefore, from \cite[Proposition A.4.4]{Bertsekas2015} it follows that $\{\cL_s(x^{k}(\omega),y^{k}(\omega))\}$ converges to a finite value.
 Since $\cL_s$ is continuous over $\X \times\dom h^*$, one has from (\ref{eq:conver-subsequence}) that
 \[
 \lim_{q\to\infty} \cL_s(x^{k_q}(\omega),y^{k_q}(\omega))=\cL_s(\bar{x}(\omega),\bar{y}(\omega)).
 \]
Combining these results with the definition of $\cC_{\omega}$, we have that $\cL_s$ is finite and constant on  $\cC_{\omega}$. The proof is completed.
\end{proof}

\begin{remark}\label{rem:usual-sto}
Under the assumptions in Theorem \ref{th:usual-sto}, from item (i) and item (iii),
there exists an event $\cA$ with $\mathbb{P}(\cA)=1$ such that, for all $\omega\in\cA$, $\dist((x^k(\omega),y^k(\omega)),\cC_{\omega})\to 0$ and $\cL_s$ equals to a constant value $\bar{\cL}_s(\omega)$ over $\cC_{\omega}$. Hence, it follows that $\Exp[\cL_s(x^k,y^k)]\to \bar{\cL_s}$ with $\bar{\cL_s}:=\Exp[\bar{\cL}_s(\omega)]$. Further, it follows from (\ref{eq:def-sto-aux}) and $z^k=(x^k,y^k,x^{k+1},x^{k-1},x^{k-2})$ that
\[
\sL_s(z^k):=\cL_s(x^k,y^k)-a\|x^k-x^{k+1}\|^2+b\|x^k-x^{k-1}\|^2+c\|x^{k-1}-x^{k-2}\|^2,
\]
which, together with (\ref{eq:i-coro}), implies that $\Exp[\sL_s(z^k)]\to \bar{\cL_s}$ as $k\to\infty$.
\end{remark}

%

We now present the main theorem of this section about the finite length property and the almost sure convergence of the whole sequence $\{(x^k,y^k)\}$ generated by Algorithm \ref{alg:SPPDG} depending on the {\KL}  property of the Lyapunov function $\sL_s$.
\begin{theorem}\label{th:convergence-sto}
Suppose that Assumption \ref{ass:setting-sto} holds and  $\sL_s$ is  a {\KL}  function with \L{}ojasiewicz exponent $\theta\in[0,1)$. Let the sequence $\{(x^k,y^k)\}$ be  bounded almost surely. Then,
\begin{itemize}
\item[(i)] it holds that
\[
\sum_{k=0}^{\infty}\Exp[\|x^{k+1}-x^k\|]<\infty,\quad \sum_{k=0}^{\infty}\Exp[\|y^{k+1}-y^k\|]<\infty;
\]
\item[(ii)] the sequence $\{(x^k,y^k)\}$ converges almost surely to a random vector $(\bar{x},\bar{y})$, and $(\bar{x},\bar{y})\in \crit\cL_s$ a.s.
\end{itemize}
\end{theorem}

\begin{proof}
Let us begin with the proof of a simple fact that $\sum_{k=0}^{\infty}\Exp[\|x^{k+1}-x^k\|]<\infty$ and $\sum_{k=0}^{\infty}\Exp[\|y^{k+1}-y^k\|]<\infty$ if $\sum_{k=0}^{\infty}\sqrt{\Exp[\|x^{k+1}-x^k\|^2]}<\infty$. In other words, if this fact is true, in order to derive item (i), it is sufficient to prove $\sum_{k=0}^{\infty}\sqrt{\Exp[\|x^{k+1}-x^k\|^2]}<\infty$. Indeed, by Jensen's inequality,  $\sum_{k=0}^{\infty}\Exp[\|x^{k+1}-x^k\|]<\infty$ is obvious. By \eqref{eq:i-4sto} (with $k=k-1$) and  $\sqrt{a+b}\leq\sqrt{a}+\sqrt{b}$, there exists a constant $\gamma_6>0$ 
such that
\be\label{eq:conver-y}
\sqrt{\Exp[\|y^k-y^{k-1}\|^2]}\leq \gamma_6(\sqrt{\Exp[\|x^{k+1}-x^{k}\|^2]}+\sqrt{\Exp[\|x^k-x^{k-1}\|^2]}+\sqrt{\Exp[\|x^{k-1}-x^{k-2}\|^2]}+\sqrt{\Exp[\Lambda_1^{k-1}]}).
\ee
Using inequalities \eqref{eq:geometric-decay}, $\sqrt{a+b}\leq\sqrt{a}+\sqrt{b}$ and $\sqrt{1-\rho}\leq1-\frac{\rho}{2}$, it follows that
\be\label{eq:Lambda1^k}
\begin{aligned}
\sqrt{\Exp[\Lambda_1^k]}\leq\ &\sqrt{(1-\rho)\Exp[\Lambda_1^{k-1}]+\sigma_{\Lambda}(\Exp[\|x^{k}-x^{k-1}\|^2]+\Exp[\|x^{k-1}-x^{k-2}\|^2])}\\
\leq\ & (1-\frac{\rho}{2})\sqrt{\Exp[\Lambda_1^{k-1}]}+\sqrt{\sigma_{\Lambda}\Exp[\|x^{k}-x^{k-1}\|^2]}+\sqrt{\sigma_{\Lambda}\Exp[\|x^{k-1}-x^{k-2}\|^2]}.
\end{aligned}
\ee
Rearranging this inequality, we obtain
\be\label{eq:cond-pre-3}
\sqrt{\Exp[\Lambda_1^{k-1}]}\leq\frac{2}{\rho}(\sqrt{\Exp[\Lambda_1^{k-1}]}-\sqrt{\Exp[\Lambda_1^{k}]})+\frac{2}{\rho}\sqrt{\sigma_{\Lambda}\Exp[\|x^{k}-x^{k-1}\|^2]}+\frac{2}{\rho}\sqrt{\sigma_{\Lambda}\Exp[\|x^{k-1}-x^{k-2}\|^2]}.
\ee
Therefore, by substituting \eqref{eq:cond-pre-3} into \eqref{eq:conver-y}, we have
\[
\sum_{k=0}^{\infty}\Exp[\|y^{k+1}-y^{k}\|]\leq\sum_{k=0}^{\infty}\sqrt{\Exp[\|y^{k+1}-y^{k}\|^2]}<\infty.
\]
Hence, the simple fact is proved.

We next prove that $\sum_{k=0}^{\infty}\sqrt{\Exp[\|x^{k+1}-x^k\|^2]}<\infty$.
By \cite[Lemma 4.5]{DTLDS2021}, if $\sL_s$ is  a {\KL}  function with exponent $\theta$, there exist an integer $K_0$ and a function $\varphi_0(s)=\sigma_0 s^{1-\theta}$ such that the following holds
\be\label{eq:kl-sto-0}
\varphi_0'(\Exp[\sL_{s}(z^k)]-\bar{\sL}_{s,k})\Exp[\dist(0,\partial \sL_{s}(z^k))]\geq 1, \ \forall k\geq K_0,
\ee
where $\{\bar{\sL}_{s,k}\}$ is a nondecreasing sequence satisfying $\Exp[\sL_{s}(z^k)]-\bar{\sL}_{s,k}>0$ and converging to a finite value $\bar{\cL}_s$ which is given in Remark \ref{rem:usual-sto}.

When $\theta=0$, we show that $\Exp[\sL_{s,k}^{\Lambda}]=\bar{\cL}_s$ holds after a finite number of iterations by contradiction. Otherwise, inequality (\ref{eq:kl-sto-0}) implies that
\be\label{eq:a5}
\Exp[\dist(0,\partial \sL_{s}(z^k))]\geq \frac{1}{\sigma_0}, \ \forall k\geq K_0.
\ee
From (\ref{eq:a5}), \eqref{eq:gradient square}  (letting $k_q=k$) and Jensen's inequality, we have
	\[
	\begin{array}{ll}
		\frac{1}{\sigma_0^2}&\leq(\Exp[\dist(0,\partial \sL_{s}(z^k))])^2\\[8pt]
		&\leq r(\Exp[\|x^{k+1}-x^k\|^2+\|x^k-x^{k-1}\|^2+\|x^{k-1}-x^{k-2}\|^2+\|y^k-y^{k-1}\|^2+\Lambda_1^k]).
	\end{array}
	\]
	Applying this inequality to (\ref{eq:sto-descent}), we have
	\[
	\begin{array}{ll}
		\Exp[\sL_{s,k}^{\Lambda}]&\leq\Exp[\sL_{s,k-1}^{\Lambda}]-e_0\Exp[\|x^{k+1}-x^k\|^2+\|x^k-x^{k-1}\|^2+\|x^{k-1}-x^{k-2}\|^2]\\[8pt]
		&\leq \Exp[\sL_{s,k-1}^{\Lambda}]-\frac{e_0}{r\sigma_0^2}+e_0\Exp[\|y^k-y^{k-1}\|^2]+e_0\Exp[\Lambda_1^k],
	\end{array}
	\]
	which is impossible after a large enough number of iterations by noticing that $\Exp[\|y^k-y^{k-1}\|^2]\to 0$ (cf. \ref{eq:i-coro2}), $\Exp[\Lambda_1^k]\to 0$ (cf. \ref{eq:gamma-conver}) and $\Exp[\sL_{s,k}^{\Lambda}]\to \bar{\cL}_s$ (cf. Remark \ref{rem:usual-sto}). Therefore, there exists an integer $\bar{K}\geq 0$ such that $\Exp[\sL_{s,k}^{\Lambda}]=\bar{\cL}_s$ holds for $k\geq \bar{K}$. In view of (\ref{eq:sto-descent}), we have $\Exp[\|x^k-x^{k-1}\|^2]=0$ for $k\geq \bar{K}$, and hence $\sum_{k=0}^{\infty}\sqrt{\Exp[\|x^{k+1}-x^k\|^2]}<\infty$.

	We now consider $\theta\in[\frac{1}{2},1)$. By  \eqref{eq:gradient square}, Jensen's inequality and $\sqrt{a+b}\leq \sqrt{a}+\sqrt{b}$, it holds that
	\be\label{eq:cond-pre-1}
	\begin{aligned}
		\Exp[\dist(0,\partial\sL_{s}(z^k))]
		\leq\ &  \sqrt{r}(\sqrt{\Exp[\|x^{k+1}-x^{k}\|^2]}+\sqrt{\Exp[\|x^k-x^{k-1}\|^2]}+\sqrt{\Exp[\|y^k-y^{k-1}\|^2]}\\
		&+\sqrt{\Exp[\|x^{k-1}-x^{k-2}\|^2]}+\sqrt{\Exp[\Lambda_1^k]}).
	\end{aligned}
	\ee
	Substituting (\ref{eq:conver-y}) into \eqref{eq:cond-pre-1},
	then we have
	\be\label{eq:cond-pre-2}
	\begin{aligned}
		\Exp[\dist(0,\partial \sL_{s}(z^k))]
		\leq\ &(\sqrt{r}+\gamma_6\sqrt{r})\left(\sqrt{\Exp[\|x^{k+1}-x^{k}\|^2]}+\sqrt{\Exp[\|x^k-x^{k-1}\|^2]}+\sqrt{\Exp[\|x^{k-1}-x^{k-2}\|^2]}\right)\\
		&+\gamma_6\sqrt{r}\sqrt{\Exp[\Lambda_1^{k-1}]}+\sqrt{r\Exp[\Lambda_1^k]}.
	\end{aligned}
	\ee
	Applying  \eqref{eq:cond-pre-3} to the last two terms in \eqref{eq:cond-pre-2}, respectively, and letting $\gamma:= \sqrt{r}+\gamma_6\sqrt{r}+\frac{2\sqrt{r\sigma_{\Lambda}}}{\rho}+\frac{2\gamma_6\sqrt{r\sigma_{\Lambda}}}{\rho}$, one has
	\bee
	\begin{aligned}
		\Exp[\dist(0,\partial \sL_{s}(z^k))]
		\leq\ & \gamma\left(\sqrt{\Exp[\|x^{k+1}-x^{k}\|^2]}+\sqrt{\Exp[\|x^{k}-x^{k-1}\|^2]}+\sqrt{\Exp[\|x^{k-1}-x^{k-2}\|^2]}\right)\\
		&+\frac{2\gamma_6\sqrt{r}}{\rho}\left(\sqrt{\Exp[\Lambda_1^{k-1}]}-\sqrt{\Exp[\Lambda_1^{k}]}\right)+\frac{2\sqrt{r}}{\rho}\left(\sqrt{\Exp[\Lambda_1^{k}]}-\sqrt{\Exp[\Lambda_1^{k+1}]}\right).
	\end{aligned}
	\eee
Denote by $\Sigma_k$ the right-hand side of the above inequality. Obviously, $\Sigma_k>0$.  Then, combining the inequality $\Exp[\dist(0,\partial \sL_{s}(z^k))]\leq \Sigma_k$ with \eqref{eq:kl-sto-0} and $\varphi_0(s)=\sigma_0 s^{1-\theta}$ gives that
\be\label{eq:cond-pre-4}
\frac{\sigma_0(1-\theta)\Sigma_k}{(\Exp[\sL_{s}(z^k)]-\bar{\sL}_{s,k})^{\theta}}\geq 1, \ \forall k\geq K_0.
\ee
Note that, for $\theta\in[\frac{1}{2}, 1)$, there exist positive constants $\beta_0$, $\kappa_2,\kappa_3$  and a sufficiently large integer $K_1>0$ such that for $k\geq K_1$,
\bee
\begin{aligned}
&(\Exp[e_1\Lambda_1^{k+1}+e_2\Lambda_1^{k}+e_3\Lambda_1^{k-1}])^{\theta}
\leq \kappa_2(\Exp[\Lambda_1^{k+1}+\Lambda_1^{k}+\Lambda_1^{k-1}])^{\theta}
\leq \kappa_2\sqrt{\Exp[\Lambda_1^{k+1}+\Lambda_1^{k}+\Lambda_1^{k-1}]}\\
&\leq  \kappa_3(\sqrt{\Exp[\Lambda_1^{k}]}+\sqrt{\Exp[\Lambda_1^{k-1}]}+\sqrt{\Exp[\|x^{k+1}-x^{k}\|^2]}+\sqrt{\Exp[\|x^{k}-x^{k-1}\|^2]})  \leq \beta_0\Sigma_k,
\end{aligned}
\eee
where the second inequality is deduced from $\Exp[\Lambda_1^k]\to 0$ for $k\to\infty$ (cf. Theorem \ref{th:usual-sto}), the third inequality is obtained by $\sqrt{a+b}\leq\sqrt{a}+\sqrt{b}$ and \eqref{eq:Lambda1^k}, and the last inequality is from \eqref{eq:cond-pre-3} and the definition of $\Sigma_k$. Take a constant $\beta>0$ such that $\beta\sigma_0(1-\theta)\geq\sigma_0(1-\theta)+\beta_0$. Then, from \eqref{eq:cond-pre-4} and the fact that $(a+b)^{\theta}\leq a^\theta+b^\theta$ for $\theta\in[\frac{1}{2},1]$, it holds for $k\geq K:=\max\{K_0,K_1\}$,
\be\label{eq:cond-pre-5}
\begin{aligned}
\displaystyle\frac{\beta\sigma_0(1-\theta)\Sigma_k}{(\Exp[\sL_{s,k}^{\Lambda}]-\bar{\sL}_{s,k})^{\theta}}
&\displaystyle\geq\frac{\beta\sigma_0(1-\theta)\Sigma_k}{(\Exp[\sL_{s}(z^k)]-\bar{\sL}_{s,k})^{\theta}+(\Exp[e_1\Lambda_1^{k+1}+e_2\Lambda_1^{k}+e_{3}\Lambda_1^{k-1}])^{\theta}}\\
&\geq \displaystyle\frac{\beta\sigma_0(1-\theta)\Sigma_k}{\sigma_0(1-\theta)\Sigma_k+\beta_0\Sigma_k}\geq 1.
\end{aligned}
\ee
 Thus, let $\varphi_1(s):=\beta\sigma_0 s^{1-\theta}$, for any $k\geq K$, (\ref{eq:cond-pre-5}) is rewritten as
\be\label{eq:cond-2}
\varphi_1'(\Exp[\sL_{s,k}^{\Lambda}]-\bar{\sL}_{s,k})\Sigma_k\geq 1.
\ee
 Since $\varphi_1$ is concave, we have
\be\label{eq:b1}
\begin{aligned}
\varphi_1(\Exp[\sL_{s,k+1}^{\Lambda}]-\bar{\sL}_{s,k+1})\leq\ & \varphi_1(\Exp[\sL_{s,k}^{\Lambda}]-\bar{\sL}_{s,k})+\varphi_1'(\Exp[\sL_{s,k}^{\Lambda}]-\bar{\sL}_{s,k})\Exp[\sL_{s,k+1}^{\Lambda}-\bar{\sL}_{s,k+1}-\sL_{s,k}^{\Lambda}+\bar{\sL}_{s,k}]\\
\leq\ & \varphi_1(\Exp[\sL_{s,k}^{\Lambda}]-\bar{\sL}_{s,k})+\varphi_1'(\Exp[\sL_{s,k}^{\Lambda}]-\bar{\sL}_{s,k})\Exp[\sL_{s,k+1}^{\Lambda}-\sL_{s,k}^{\Lambda}]\\
\leq\ &
\varphi_1(\Exp[\sL_{s,k}^{\Lambda}]-\bar{\sL}_{s,k})-\frac{e_0}{\Sigma_k}\Exp[\|x^{k+2}-x^{k+1}\|^2+\|x^{k+1}-x^{k}\|^2+\|x^{k}-x^{k-1}\|^2],
\end{aligned}
\ee
where the second inequality is obtained by $\bar{\sL}_{s,k}\leq \bar{\sL}_{s,k+1}$, the third inequality is from Lemma \ref{lem:sto-descent} and (\ref{eq:cond-2}). Let $\cM_{m,n}:=\varphi_1(\Exp[\sL_{s,m}^{\Lambda}]-\bar{\sL}_{s,m})-\varphi_1(\Exp[\sL_{s,n}^{\Lambda}]-\bar{\sL}_{s,n})$, then (\ref{eq:b1}) implies
\[
\cM_{k,k+1}\geq
 \frac{e_0}{\Sigma_k}\Exp[\|x^{k+2}-x^{k+1}\|^2].
\]
Rewriting this inequality and using $4\sqrt{ab}\leq a/\gamma+4\gamma b$ for any $\gamma>0$ yield that
\bee
\begin{aligned}
	4\sqrt{\Exp[\|x^{k+2}-x^{k+1}\|^2]}\leq\ & 4\sqrt{\frac{\cM_{k,k+1}\Sigma_k}{e_0}}\leq \frac{\Sigma_k}{\gamma}+\frac{4\gamma\cM_{k,k+1}}{e_0},
\end{aligned}
\eee
 which, together with the definition of $\Sigma_k$, gives
\bee
\begin{aligned}
	4\sqrt{\Exp[\|x^{k+2}-x^{k+1}\|^2]}\leq\ &\sqrt{\Exp[\|x^{k+1}-x^{k}\|^2]}+\sqrt{\Exp[\|x^{k}-x^{k-1}\|^2]}+\sqrt{\Exp[\|x^{k-1}-x^{k-2}\|^2]}+\frac{4\gamma\cM_{k,k+1}}{e_0}\\
	&+\frac{2\gamma_6\sqrt{r}}{\rho\gamma}\left(\sqrt{\Exp[\Lambda_1^{k-1}]}-\sqrt{\Exp[\Lambda_1^{k}]}\right)+\frac{2\sqrt{r}}{\rho\gamma}\left(\sqrt{\Exp[\Lambda_1^{k}]}-\sqrt{\Exp[\Lambda_1^{k+1}]}\right).
\end{aligned}
\eee
Summing up from $k=K$ to $n$, we have
\be\label{eq:cond-3}
\begin{aligned}
	\sum_{k=K}^{n}\sqrt{\Exp[\|x^{k+2}-x^{k+1}\|^2]}\leq\ & 3\sqrt{\Exp[\|x^{K+1}-x^{K}\|^2]}+2\sqrt{\Exp[\|x^{K}-x^{K-1}\|^2]}+\sqrt{\Exp[\|x^{K-1}-x^{K-2}\|^2]}\\
	&+\sum_{k=K}^{n}\frac{4\gamma\cM_{k,k+1}}{e_0}	+\frac{2\gamma_6\sqrt{r}}{\rho\gamma}\sqrt{\Exp[\Lambda_1^{K-1}]}+\frac{2\sqrt{r}}{\rho\gamma}\sqrt{\Exp[\Lambda_1^{K}]}.
\end{aligned}
\ee
 By the definition of $\cM_{k,k+1}$, it holds that $\sum_{k=K}^{n}\cM_{k,k+1}=\cM_{K,n+1}\leq \varphi_1(\Exp[\sL_{s,K}^{\Lambda}]-\bar{\sL}_{s,K})$.
Let $n\to\infty$ in (\ref{eq:cond-3}), then  it follows that
\[
\sum_{k=0}^{\infty}\sqrt{\Exp[\|x^{k+1}-x^{k}\|^2]}<\infty.
\]

For $\theta\in(0,\frac{1}{2})$, we show that it can be reduced to the case that $\theta=\frac{1}{2}$. Indeed, from Remark \ref{rem:usual-sto}, we can
let $K_0$ be large enough such that $\Exp[\sL_{s}(z^k)]-\bar{\sL}_{s,k}<1$. Then,
since (\ref{eq:kl-sto-0}) holds with  $\theta\in(0,\frac{1}{2})$, we have that (\ref{eq:kl-sto-0}) also holds with $\theta=\frac{1}{2}$. Thus, we can get the claim immediately by following the analysis for the case that $\theta\in[\frac{1}{2},1)$.

Combining these results together, we have that  $\sum_{k=0}^{\infty}\sqrt{\Exp[\|x^{k+1}-x^{k}\|^2]}<\infty$ holds for all $\theta\in[0,1)$, and hence item (i) is derived by the previously mentioned simple fact.

In the proof of Theorem \ref{th:usual-sto}, we have shown that there exists an event $\cA$ with measure 1 such that, for any $\omega\in\cA$, every convergent subsequence of $\{(x^k(\omega),y^k(\omega))\}$ converges to a point $(\bar{x}(\omega),\bar{y}(\omega))$ belonging to $\crit\cL_s$. If follows from item (i) that
\[
\sum_{k=0}^{\infty}\|x^{k+1}-x^k\|<\infty \ \mbox{a.s.},\quad \sum_{k=0}^{\infty}\|y^{k+1}-y^k\|<\infty\ \mbox{a.s.},
\]
and consequently,
\[
\sum_{k=0}^{\infty}\|x^{k+1}(\omega)-x^k(\omega)\|<\infty ,\quad \sum_{k=0}^{\infty}\|y^{k+1}(\omega)-y^k(\omega)\|<\infty.
\]
In other words, $\{(x^k(\omega),y^k(\omega))\}$ is a Cauchy sequence. Thus, the whole sequence $\{(x^k(\omega),y^k(\omega))\}$ converges to $(\bar{x}(\omega),\bar{y}(\omega))$. Therefore, there exists a random vector $(\bar{x},\bar{y})$ such that $\{(\bar{x},\bar{y})\}\in \crit\cL_s$ a.s. and $\{(x^k,y^k)\}$ converges almost surely to $(\bar{x},\bar{y})$.  Item (ii) is proved.
\end{proof}

Finally, we establish
the convergence rates of the sequence $\{(x^k,y^k)\}$ in the context of \L{}ojasiewicz exponent
in the following theorem.
\begin{theorem}\label{th:convergence-rate2}
Suppose that Assumption \ref{ass:setting-sto} holds and  $\sL_s$ is a {\KL}  function with \L{}ojasiewicz exponent $\theta\in[0,1)$. Let the sequence $\{(x^k,y^k)\}$ be  bounded almost surely and $\{(x^k,y^k)\}$ converges almost surely to some random vector $(\bar{x},\bar{y})$.  Then, the following statements hold:
	\begin{itemize}
		\item[(i)] if $\theta=0$, the sequence $\{(x^k,y^k)\}$ converges in expectation after finite steps;
\item[(ii)] if $\theta\in(0,\frac{1}{2}]$, then there exist constants $\nu, \bar{\nu}>0$, $\tau, \bar{\tau}\in(0,1)$ and a sufficiently large integer $K$ such that for $k\geq K$,
		\[\Exp[\|x^{k}-\bar{x}\|]\leq \nu\tau^{k-K}, \quad \Exp[\|y^{k}-\bar{y}\|]\leq \bar{\nu}\bar{\tau}^{k-K};\]
		\item[(iii)] if $\theta\in(\frac{1}{2},1)$, then there exist  constants $\mu,\bar{\mu}>0$ and a sufficiently large integer $\bar{K}$ such that for $k\geq \bar{K}$,
		\[\Exp[\|x^{k}-\bar{x}\|]\leq \mu k^{-{\frac{1-\theta}{2\theta-1}}}, \quad \Exp[\|y^{k}-\bar{y}\|]\leq \bar{\mu} k^{-{\frac{1-\theta}{2\theta-1}}}.\]
	\end{itemize}
\end{theorem}

\begin{proof}
	Item (i) has been presented in the proof of Theorem \ref{th:convergence-sto}.

Let us point out that, for the case that $\theta\in(0,1/2)$, by the same reason in the proof of Theorem \ref{th:convergence-sto}, the analysis in the following can reduce to the case that $\theta=1/2$.
Therefore, it is sufficient to consider the case that $\theta\in[1/2,1)$.
Let \[\Delta_k:=\sum_{q=k}^{\infty}\sqrt{\Exp[\|x^{q+1}-x^{q}\|^2]}+\sum_{q=k}^{\infty}\sqrt{\Exp[\|x^{q}-x^{q-1}\|^2]}+\sum_{q=k}^{\infty}\sqrt{\Exp[\|x^{q-1}-x^{q-2}\|^2]}.\]
Noticing that \eqref{eq:cond-3} holds for all $\theta\in[1/2,1)$.	
Similarly, we can also have
	\be\label{eq:cond-5}
	\begin{aligned}
		\sum_{k=K}^{n}\sqrt{\Exp[\|x^{k}-x^{k-1}\|^2]}\leq\ & \sqrt{\Exp[\|x^{n+1}-x^{n}\|^2]}+ \sqrt{\Exp[\|x^{K-1}-x^{K-2}\|^2]}+\sum_{k=K}^{n}\frac{4\gamma\cM_{k,k+1}}{e_0}\\	&+\frac{2\gamma_6\sqrt{r}}{\rho\gamma}\sqrt{\Exp[\Lambda_1^{K-1}]}+\frac{2\sqrt{r}}{\rho\gamma}\sqrt{\Exp[\Lambda_1^{K}]}
	\end{aligned}
	\ee
and 	
	\be\label{eq:cond-4}
	\begin{aligned}
	\sum_{k=K}^{n}\sqrt{\Exp[\|x^{k+1}-x^{k}\|^2]}\leq\ & 2\sqrt{\Exp[\|x^{K}-x^{K-1}\|^2]}+\sqrt{\Exp[\|x^{K-1}-x^{K-2}\|^2]}+\sum_{k=K}^{n}\frac{4\gamma\cM_{k,k+1}}{e_0}\\ &+\frac{2\gamma_6\sqrt{r}}{\rho\gamma}\sqrt{\Exp[\Lambda_1^{K-1}]}+\frac{2\sqrt{r}}{\rho\gamma}\sqrt{\Exp[\Lambda_1^{K}]}.
	\end{aligned}
	\ee	
	Then, combining  \eqref{eq:cond-3}, \eqref{eq:cond-5} and \eqref{eq:cond-4} (let $n\to\infty$), for any $k\geq K$, it follows from the definition of $\cM_{m,n}$ and $\varphi_1(s)=\beta\sigma_0 s^{1-\theta}$ that
	\be\label{eq:iii-1-1}
	\begin{aligned}
		\Delta_{k+1}
\leq& \  4\left(\sqrt{\Exp[\|x^{k+1}-x^{k}\|^2]}+\sqrt{\Exp[\|x^{k}-x^{k-1}\|^2]}+\sqrt{\Exp[\|x^{k-1}-x^{k-2}\|^2]}\right)\\
		&\ +\frac{6\gamma_6\sqrt{r}}{\rho\gamma}\sqrt{\Exp[\Lambda_1^{k-1}]}+\frac{6\sqrt{r}}{\rho\gamma}\sqrt{\Exp[\Lambda_1^{k}]}+\frac{12\gamma}{e_0}\varphi_1(\Exp[\sL_{s,k}^{\Lambda}]-\bar{\sL}_{s,k})\\
		=\ &
	4\left(\sqrt{\Exp[\|x^{k+1}-x^{k}\|^2]}+\sqrt{\Exp[\|x^{k}-x^{k-1}\|^2]}+\sqrt{\Exp[\|x^{k-1}-x^{k-2}\|^2]}\right)\\
	&+\frac{6\gamma_6\sqrt{r}}{\rho\gamma}\sqrt{\Exp[\Lambda_1^{k-1}]}+\frac{6\sqrt{r}}{\rho\gamma}\sqrt{\Exp[\Lambda_1^{k}]}+\beta_1(\Exp[\sL_{s,k}^{\Lambda}]-\bar{\sL}_{s,k})^{1-\theta},
	\end{aligned}	
	\ee	
where $\beta_1:=12\gamma\beta\sigma_0/e_0$.
By the definition of $\sL_{s,k}^{\Lambda}$ 
	and $(a+b)^{1-\theta}\leq a^{1-\theta}+b^{1-\theta}$ for $\theta\in[1/2,1)$, it follows that
	\be\label{eq:iii-2}
	(\Exp[\sL_{s,k}^{\Lambda}]-\bar{\sL}_{s,k})^{1-\theta}\leq (\Exp[\sL_{s}(z^k)]-\bar{\sL}_{s,k})^{1-\theta}+\beta_2(\Exp[\Lambda_1^{k+1}]+\Exp[\Lambda_1^{k}]+\Exp[\Lambda_1^{k-1}])^{1-\theta},
	\ee
	where $\beta_2=\max\{e_1,e_2,e_3\}^{1-\theta}$. Let $\bar{\Sigma}_k$ be the right-hand side of \eqref{eq:cond-pre-2}, then there exists a constant $\beta_4>0$ such that
	\be\label{eq:iii-3-1}
4(\sqrt{\Exp[\|x^{k+1}-x^{k}\|^2]}+\sqrt{\Exp[\|x^{k}-x^{k-1}\|^2]}+\sqrt{\Exp[\|x^{k-1}-x^{k-2}\|^2]})+\frac{6\gamma_6\sqrt{r}}{\rho\gamma}\sqrt{\Exp[\Lambda_1^{k-1}]}+\frac{6\sqrt{r}}{\rho\gamma}\sqrt{\Exp[\Lambda_1^{k}]}\leq \beta_4\bar{\Sigma}_k.
	\ee
Plugging (\ref{eq:iii-2}) and (\ref{eq:iii-3-1}) into (\ref{eq:iii-1-1}) yields that
\be\label{eq:iii-1}
	\Delta_{k+1}\leq\beta_4\bar{\Sigma}_k +	\beta_1 (\Exp[\sL_{s}(z^k)]-\bar{\sL}_{s,k})^{1-\theta}+\beta_1\beta_2(\Exp[\Lambda_1^{k+1}]+\Exp[\Lambda_1^{k}]+\Exp[\Lambda_1^{k-1}])^{1-\theta}.	
\ee		
From (\ref{eq:kl-sto-0}), it follows that
	\be\label{eq:iii-3-kl-1}
	(\Exp[\sL_{s}(z^k)]-\bar{\sL}_{s,k})^{1-\theta}\leq (\sigma_0 (1-\theta)\Exp[\dist(0,\partial \sL_{s}(z^k))])^{\frac{1-\theta}{\theta}}.
	\ee
 Since $2\theta\geq 1$, we have that
	\bee
	\begin{aligned}
		(\Exp[\Lambda_1^{k+1}]+\Exp[\Lambda_1^{k}]+\Exp[\Lambda_1^{k-1}])^{1-\theta}\leq\ & (\Exp[\Lambda_1^{k+1}]+\Exp[\Lambda_1^{k}]+\Exp[\Lambda_1^{k-1}])^{\frac{1-\theta}{2\theta}}\\
		\leq\ & (\sqrt{\Exp[\Lambda_1^{k+1}]}+\sqrt{\Exp[\Lambda_1^{k}]}+\sqrt{\Exp[\Lambda_1^{k-1}]})^{\frac{1-\theta}{\theta}},
	\end{aligned}
	\eee
which further implies that, there exists a constant $\beta_3>0$ such that
	\be\label{eq:iii-4}
	(\Exp[\Lambda_1^{k+1}]+\Exp[\Lambda_1^{k}]+\Exp[\Lambda_1^{k-1}])^{1-\theta}\leq \beta_3\bar{\Sigma}_k^{\frac{1-\theta}{\theta}}.
	\ee
It follows from (\ref{eq:iii-3-kl-1}) and \eqref{eq:cond-pre-2} that
	\be\label{eq:iii-3-kl}
	(\Exp[\sL_{s}(z^k)]-\bar{\sL}_{s,k})^{1-\theta}
	\leq (\sigma_0 (1-\theta)\bar{\Sigma}_k)^{\frac{1-\theta}{\theta}}.
	\ee
Substituting  \eqref{eq:iii-3-kl} and \eqref{eq:iii-4} into (\ref{eq:iii-1})
gives that
	\be\label{eq:iii-5}
			\Delta_{k+1}\leq ((\sigma_0(1-\theta))^{\frac{1-\theta}{\theta}}\beta_1+\beta_1\beta_2\beta_3)\bar{\Sigma}_k^{\frac{1-\theta}{\theta}}+\beta_4\bar{\Sigma}_k\leq\varrho\bar{\Sigma}_k^{\frac{1-\theta}{\theta}},
	\ee
	where the second inequality is derived from $\frac{1-\theta}{\theta}\leq 1$, $\bar{\Sigma}_k\to 0$ and  $\varrho:=(\sigma_0(1-\theta))^{\frac{1-\theta}{\theta}}\beta_1+\beta_1\beta_2\beta_3+\beta_4$. To bound $\bar{\Sigma}_k$, we have from \eqref{eq:Lambda1^k} and \eqref{eq:cond-pre-3} that
	\be\label{eq:iii-8}
	\begin{aligned}
		\bar{\Sigma}_k\leq\  &(\sqrt{r}+\gamma_6\sqrt{r}+\sqrt{r\sigma_{\Lambda}})\left(\sqrt{\Exp[\|x^{k+1}-x^{k}\|^2]}+\sqrt{\Exp[\|x^k-x^{k-1}\|^2]}+\sqrt{\Exp[\|x^{k-1}-x^{k-2}\|^2]}\right)\\
		&+(2\gamma_6\sqrt{r}+(2-\rho)\sqrt{r})\sqrt{\Exp[\Lambda_1^{k-1}]}-(\gamma_6\sqrt{r}+(1-\frac{\rho}{2})\sqrt{r})\sqrt{\Exp[\Lambda_1^{k-1}]}\\
		\leq\  & \beta_5\left(\sqrt{\Exp[\|x^{k+1}-x^{k}\|^2]}+\sqrt{\Exp[\|x^{k}-x^{k-1}\|^2]}+\sqrt{\Exp[\|x^{k-1}-x^{k-2}\|^2]}\right)\\		&+2\beta_6\left(\sqrt{\Exp[\Lambda_1^{k-1}]}-\sqrt{\Exp[\Lambda_1^{k}]}\right)-(\gamma_6\sqrt{r}+(1-\frac{\rho}{2})\sqrt{r})\sqrt{\Exp[\Lambda_1^{k-1}]},
	\end{aligned}
	\ee
	where  $\beta_5:=\sqrt{r}+\gamma_6\sqrt{r}+\frac{4+4\gamma_6-\rho}{\rho}\sqrt{r\sigma_{\Lambda}}$, $\beta_6:=\frac{2+2\gamma_6-\rho}{\rho}\sqrt{r}$.
 Let   $\Delta^{\Lambda}_k:=\Delta_k+\frac{2\beta_6(1-\frac{\rho}{4})}{\beta_5}\sqrt{\Exp[\Lambda_1^{k-1}]}$.
	Then,
	\be\label{eq:iii-6}	(\Delta^{\Lambda}_{k+1})^{\frac{\theta}{1-\theta}}\leq\frac{2^{\frac{\theta}{1-\theta}}}{2}\Delta_{k+1}^{\frac{\theta}{1-\theta}}
+\frac{2^{\frac{\theta}{1-\theta}}}{2}\left(\frac{2\beta_6(1-\frac{\rho}{4})}{\beta_5}\sqrt{\Exp[\Lambda_1^{k}]}\right)^{\frac{\theta}{1-\theta}}\leq \frac{(2\varrho)^{\frac{\theta}{1-\theta}}}{2} \bar{\Sigma}_k+\frac{(\frac{4\beta_6(1-\frac{\rho}{4})}{\beta_5})^{\frac{\theta}{1-\theta}}}{2}\sqrt{\Exp[\Lambda_1^{k}]},
	\ee
	where the first inequality is obtained by using inequalities $\frac{\theta}{1-\theta}\geq 1$ and  $(a+b)^v\leq2^{v-1}a^v+2^{v-1}b^v$ for any $v\geq 1$, the second inequality is from \eqref{eq:iii-5}.
	Substituting \eqref{eq:iii-8} into (\ref{eq:iii-6}) and rearranging terms, we have
\[ (\Delta^{\Lambda}_{k+1})^{\frac{\theta}{1-\theta}}\leq\frac{\beta_5(2\varrho)^{\frac{\theta}{1-\theta}}}{2}(\Delta^{\Lambda}_{k}-\Delta^{\Lambda}_{k+1}),
\]
which further gives
	\be\label{eq:iii-9}
   \Delta^{\Lambda}_{k+1}\leq2\varrho(\frac{\beta_5}{2})^{\frac{1-\theta}{\theta}}(\Delta^{\Lambda}_{k}-\Delta^{\Lambda}_{k+1})^{\frac{1-\theta}{\theta}}.
    \ee

Note that, the    result (\ref{eq:iii-9}) is very similar to \eqref{eq:trans}. Hence, the rest of the proof can be conducted similarly as that of Theorem \ref{th:convergence-rate1}.
In specific,
if $\theta\in(\frac{1}{2},1)$,  there exist an integer $\bar{K}$, constants $\mu>0$  and $\nu_1=\frac{1-2\theta}{1-\theta}<0$ such that for $n>\bar{K}$,
\[
  \Delta^{\Lambda}_{n}\leq \mu n^{\frac{1}{\nu_1}};
\]
if $\theta\in(0,\frac{1}{2}]$, there exists constants $\nu>0$ and $\tau=\frac{\varrho\beta_5}{1+\varrho\beta_5}<1$ such that for $k\geq K$,
\[
\Delta^{\Lambda}_{k+1}\leq \nu\tau^{k-K}.
\]
Since $\Exp[\|x^k-\bar{x}\|]\leq \Delta^{\Lambda}_{k+1}$, the estimations for $\Exp[\|x^k-\bar{x}\|]$ in (ii) and (iii) are derived.

Finally, we consider the estimations for $\Exp[\|y^k-\bar{y}\|]$.  Combining  \eqref{eq:conver-y} and \eqref{eq:cond-pre-3} yields that
\bee
\begin{aligned}
\sqrt{\Exp[\|y^q-y^{q-1}\|^2]}\leq \ &\gamma_6(1+\frac{2\sqrt{\sigma_{\Lambda}}}{\rho})\left(\sqrt{\Exp[\|x^{q+1}-x^{q}\|^2]}+\sqrt{\Exp[\|x^q-x^{q-1}\|^2]}+\sqrt{\Exp[\|x^{q-1}-x^{q-2}\|^2]}\right)\\
&+\frac{2\gamma_6}{\rho}\left(\sqrt{\Exp[\Lambda_1^{q-1}]} -\sqrt{\Exp[\Lambda_1^{q}]}\right).
\end{aligned}
\eee
Summing up from $q=k$ to $\infty$, we have
\be\label{eq:conv-y-1}
\begin{aligned}
\sum_{q=k}^{\infty}\sqrt{\Exp[\|y^{q+1}-y^{q}\|^2]}\leq \ &\gamma_6(1+\frac{2\sqrt{\sigma_{\Lambda}}}{\rho}) \Delta_{k+1}+\frac{2\gamma_6}{\rho}\sum\limits_{q=k}^{\infty}\left(\sqrt{\Exp[\Lambda_1^{q}]} -\sqrt{\Exp[\Lambda_1^{q+1}]}\right)\\
\leq \ & \gamma_6(1+\frac{2\sqrt{\sigma_{\Lambda}}}{\rho}) \Delta_{k+1}+\frac{2\gamma_6}{\rho}\sqrt{\Exp[\Lambda_1^{k}]}.
\end{aligned}
\ee
Let $V_k:=\gamma_6(1+\frac{2\sqrt{\sigma_{\Lambda}}}{\rho}) \Delta_{k}+\frac{2\gamma_6}{\rho}\sqrt{\Exp[\Lambda_1^{k-1}]}$.
Then by \eqref{eq:iii-1-1}, it holds that
\bee
\begin{aligned}
V_{k+1}\leq \ & 4\gamma_6(1+\frac{2\sqrt{\sigma_{\Lambda}}}{\rho}) \left(\sqrt{\Exp[\|x^{k+1}-x^{k}\|^2]}+\sqrt{\Exp[\|x^{k}-x^{k-1}\|^2]}+\sqrt{\Exp[\|x^{k-1}-x^{k-2}\|^2]}\right)\\
&+\frac{6\gamma_6^2\sqrt{r}}{\rho\gamma}(1+\frac{2\sqrt{\sigma_{\Lambda}}}{\rho})\sqrt{\Exp[\Lambda_1^{k-1}]}+\frac{6\gamma_6\sqrt{r}}{\rho\gamma}(1+\frac{2\sqrt{\sigma_{\Lambda}}}{\rho})\sqrt{\Exp[\Lambda_1^{k}]}\\
&+\gamma_6\beta_1(1+\frac{2\sqrt{\sigma_{\Lambda}}}{\rho})(\Exp[\sL_{s,k}^{\Lambda}]-\bar{\sL}_{s,k})^{1-\theta}+\frac{2\gamma_6}{\rho}\sqrt{\Exp[\Lambda_1^{k}]}.
 \end{aligned}
\eee
By the same line to obtain \eqref{eq:iii-5}, there exists a constant $\varrho'>0$ such that
\bee
V_{k+1}\leq \varrho' \bar{\Sigma}_k^{\frac{1-\theta}{\theta}}.
\eee
Similar to obtain the estimations of $\Delta^{\Lambda}_{k+1}$, for $V^{\Lambda}_{k}:=V_k+\frac{2\beta_6(1-\frac{\rho}{4})}{\beta_5}\sqrt{\Exp[\Lambda_1^{k-1}]}$, we can obtain that
\be\label{eq:V-1}
  V^{\Lambda}_{k}\leq \bar{\mu} k^{\frac{1}{\nu_1}} \text{ for } \theta\in\left(1/2,1\right)
  \ee
 and
 \be\label{eq:V-2}
V^{\Lambda}_{k+1}\leq\bar{\nu}\tau^{k-K} \text{ for } \theta\in\left(0,1/2\right],
 \ee
where $\bar{\mu}$ and $\bar{\nu}$ are some positive constants.  The triangle inequality gives
\[
\Exp[\|y^k-\bar{y}\|]\leq V_{k+1}\leq V^{\Lambda}_{k+1},
\]
which, together with \eqref{eq:V-1} and \eqref{eq:V-2}, implies the estimations for $\Exp[\|y^k-\bar{y}\|]$ in (ii) and (iii).
The proof is completed.
\end{proof}

\section{Preliminary numerical experiments}\label{sec:Numerical Experiments}
In this section, we show the efficiency of our proposed algorithms, and compare them with several  state-of-the-art algorithms on a variety of test problems.
All numerical experiments are carried out using MATLAB R2023a on  a desktop
computer with Intel Core i5 2.5GHz and 32GB memory.

\subsection{Image denoising via $\ell_0$ gradient minimization}
Let $b\in\R^{n\times m}$ represent the noisy input image and $x$ be the result after denoising. The $2D$ discrete gradient operator of $x$ is denoted by $\nabla x$ which is linear.
In this subsection, we focus on the following  $\ell_0$ gradient minimization problem \cite{XLXJ2011,HS2013,TYPC2020}:
\be\label{eq:image-denoising}
\min_{x\in\hat{\cD}}\quad \frac{1}{2}\|x-b\|_F^2+\lambda\|\nabla x\|_0.
\ee
Here,  $\lambda>0$ is a regularization parameter,   and the set $\hat{\cD}=\{x\in\R^{n\times m}: c_1\leq (\nabla x)_{i,j}\leq c_2\}$ for two given constants  $c_1,c_2$.
Problem \eqref{eq:image-denoising} can be expressed  in the form of \eqref{eq:p} with $f(x)=\frac{1}{2}\|x-b\|_F^2$, $h(u)=\lambda\|u\|_0+\cI_{\cD}(u)$ with $\cD=\{u:c_1\leq u_{i,j}\leq c_2\}$ and $\cI_{\cD}(\cdot)$ being the indicator function, and $A$ is the  linear operator associated with $\nabla x$ such that $Ax=\nabla x$..

In this experiment, we aim to compare the performance of
PPDG with ADMM \cite{Li2015} and PDHG \cite{Thomas2015} for solving problem \eqref{eq:image-denoising}.
Recall that, to apply our algorithm, PPDG, we should calculate $\prox_{h^*}^{M}(y^k+M^{-1}A(2x^{k+1}-x^k))$ with $M=\alpha AA^T$ at each iteration. To avoid computing the inverse of $M$, in practice we calculate the following term as an approximation,
\be\label{eq:proximal-l0a}
\prox_{\beta h^*}\left( y^k+\beta A(2x^{k+1}-x^k)\right),
\ee
where  $ \beta=1/(\alpha\|A\|^2)$. Here, the proximal mapping $\prox_{\beta h^*}(\cdot)$ is computed according to Example \ref{ex:l0}.

The following peak signal-to-noise ratio (PSNR) is used as a measure of the  quality of the denoised image,
\[
\text{PSNR}=10\times \log_{10}\frac{mn(\max{x^k})^2}{\|x^k-x_{org}\|^2},
\]
where  $x_{org}$ is the original image without any noisy, $x^k$ is the output image.
Take $c_1=-1$, $c_2=1$ and  $\lambda=0.1$.
The numerical results are illustrated in Figure \ref{Fig-2} and Figure \ref{Fig-3}. The first row  of Figure \ref{Fig-2} displays the original images\footnote{The images are available in \url{https://www.robots.ox.ac.uk/~vgg/data/}  and  \url{http://www.eecs.qmul.ac.uk/~phao/IP/Images/}}, the second row exhibits the noisy input images with varying levels of noise,  while the third, fourth and fifth rows  show the denoised images by PPDG, ADMM and PDHG, respectively. The values of PSNR for the denoised images and the corresponding running time   are presented in Table \ref{table:psnr}. Figure \ref{Fig-3} provides enlarged views  of specific images from Figure \ref{Fig-2}, allowing readers to see the details clearly.

 We can observe that,   from Figure \ref{Fig-2}, PPDG outperforms ADMM and PDHG in terms of denoising capability, and from Table \ref{table:psnr}, PPDG is superior to  ADMM and PDHG in view of running time. Moreover, PPDG  is comparable to ADMM in terms of PSNR from Table \ref{table:psnr}. Finally, Figure \ref{Fig-3} shows that  PPDG is superior to ADMM with respect to image detail processing.

\begin{figure}[!htp]
	\centering
	\setlength{\belowcaptionskip}{-6pt}
	\begin{tabular}{cccc}
		\subfloat[original image]{
			\includegraphics[width=3.8cm,height=3.2cm]{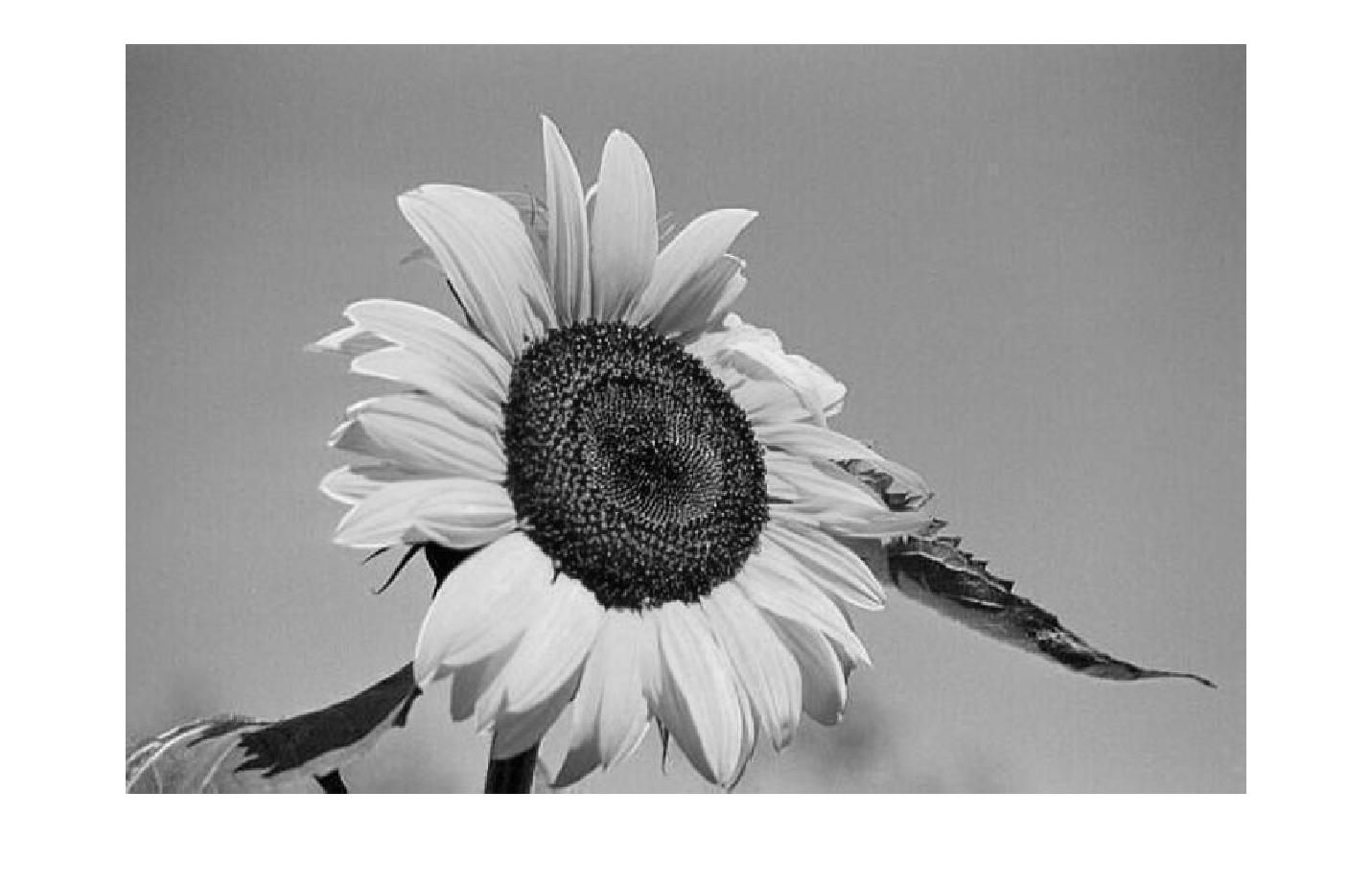}}
    	\subfloat[original image]{
			\includegraphics[width=3.8cm,height=3.2cm]{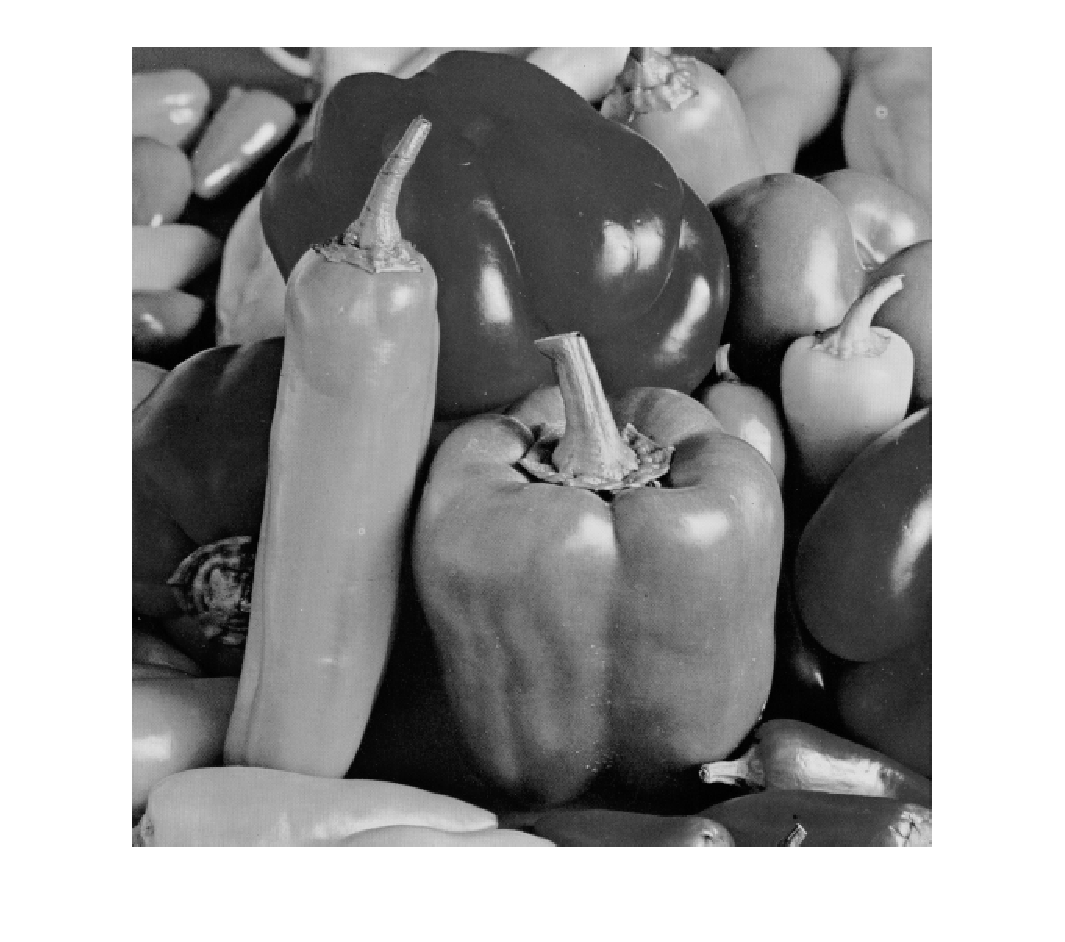}}
		\subfloat[original image]{
			\includegraphics[width=3.8cm,height=3.2cm]{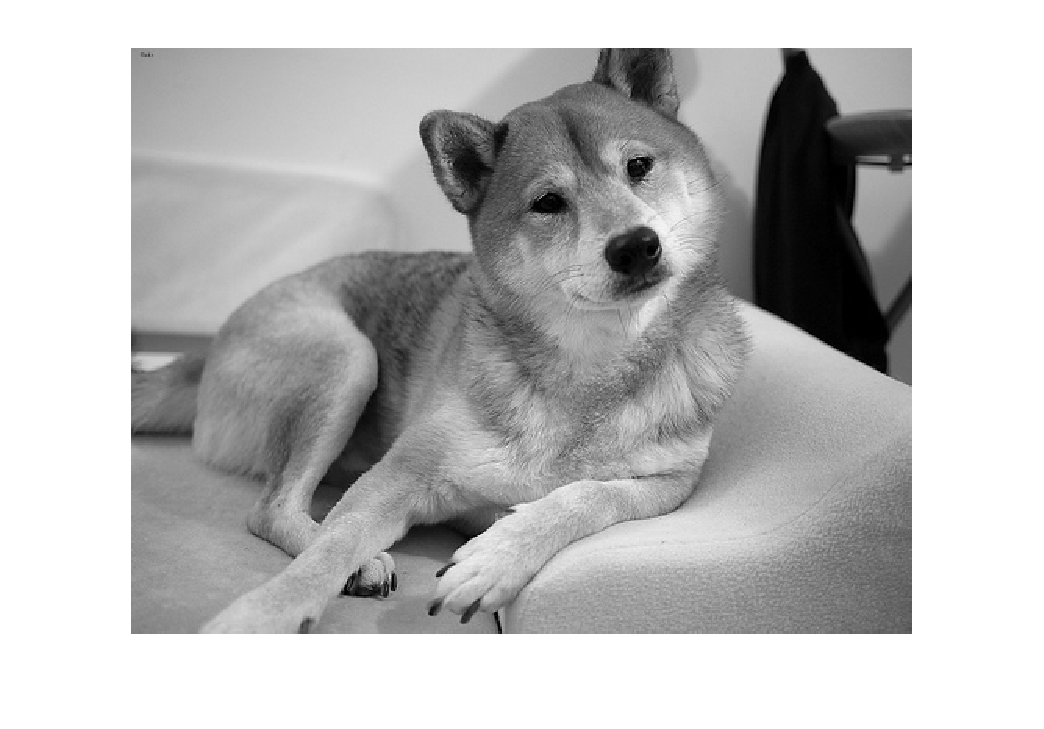}}
		\subfloat[original image]{
			\includegraphics[width=3.8cm,height=3.2cm]{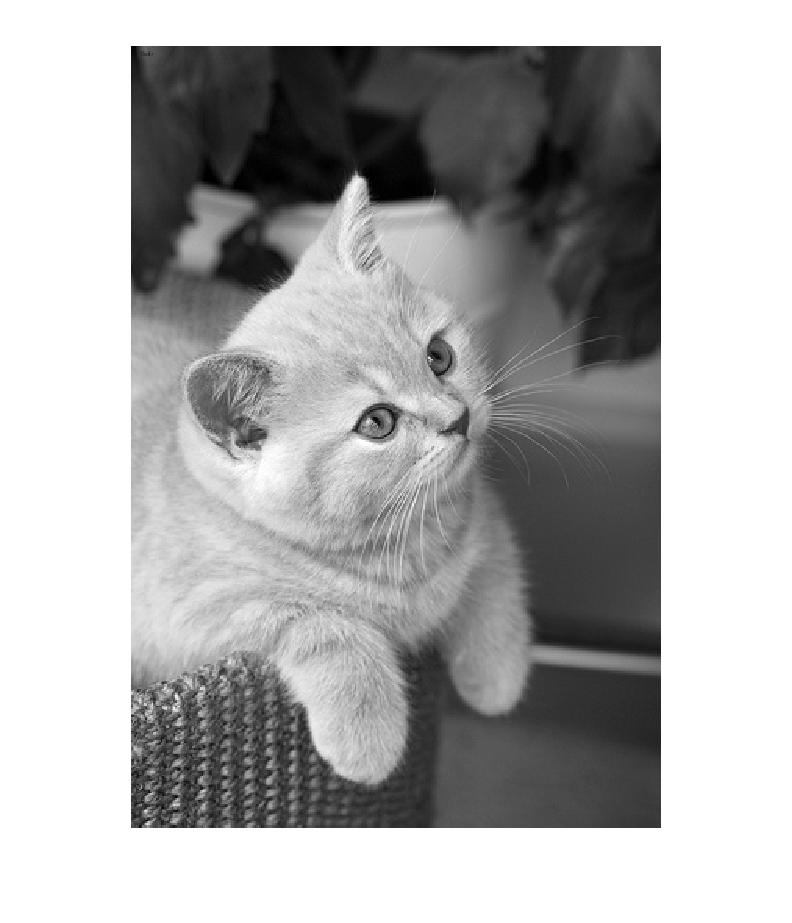}}\\
		\subfloat[noisy image]{
			\includegraphics[width=3.8cm,height=3.2cm]{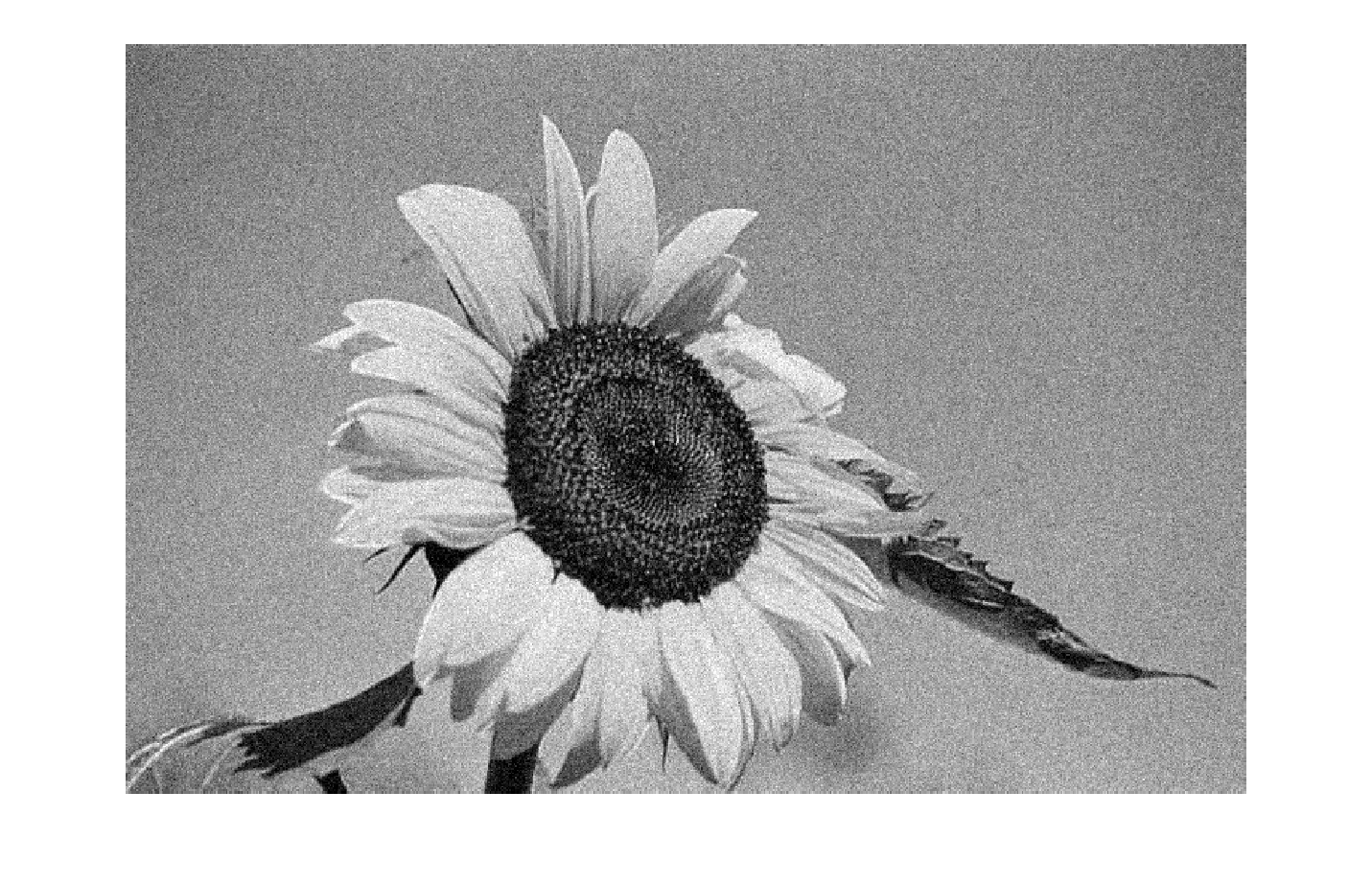}}
		\subfloat[noisy image]{
			\includegraphics[width=3.8cm,height=3.2cm]{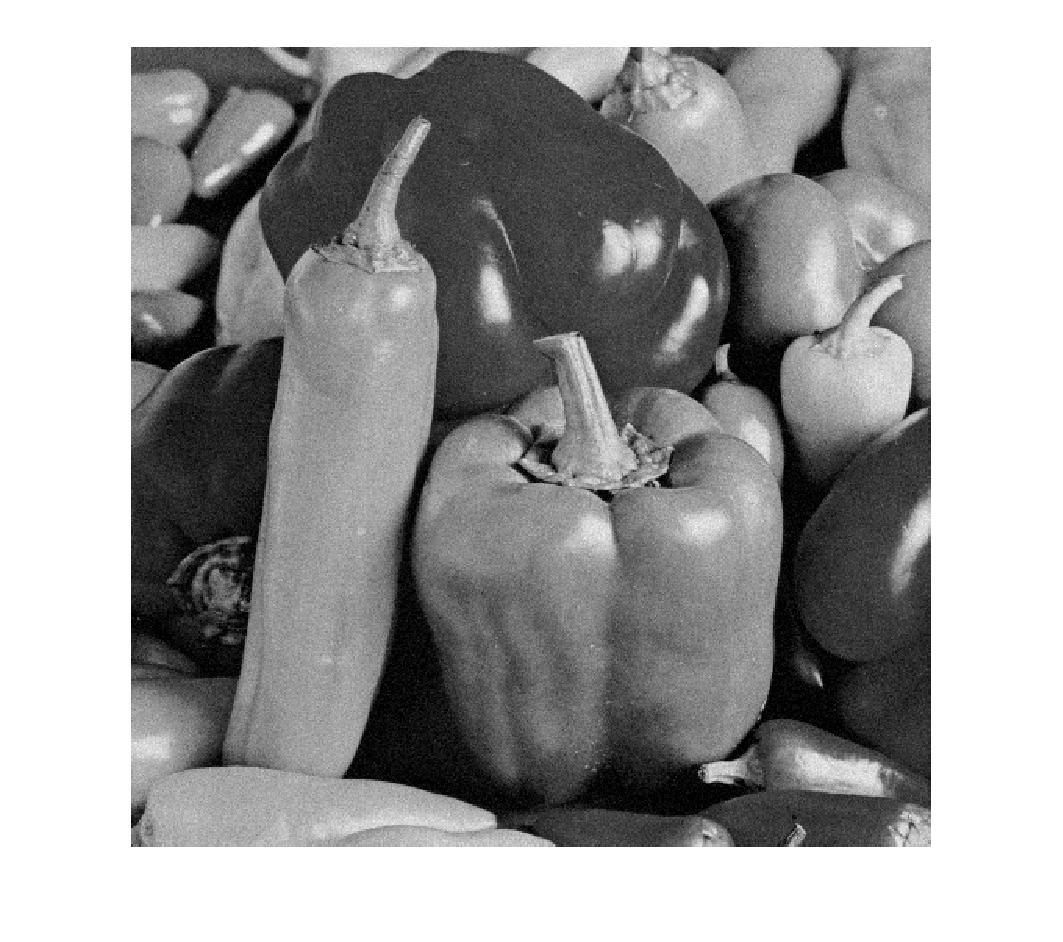}}
		\subfloat[noisy image]{
			\includegraphics[width=3.8cm,height=3.2cm]{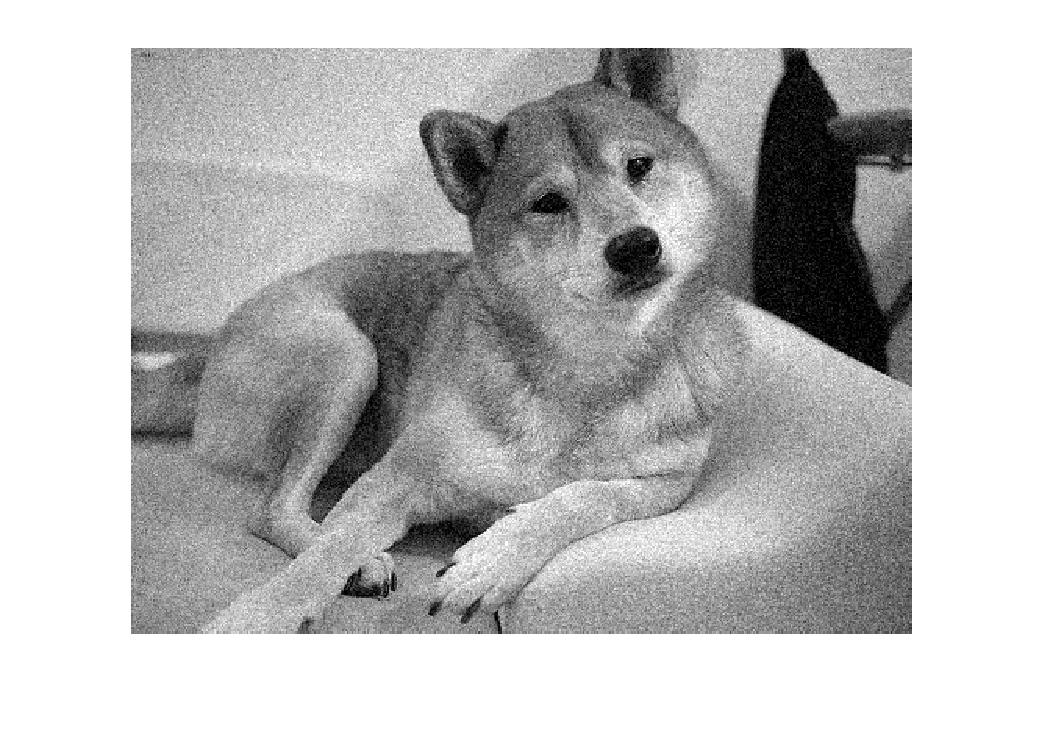}}
		\subfloat[noisy image]{
			\includegraphics[width=3.8cm,height=3.2cm]{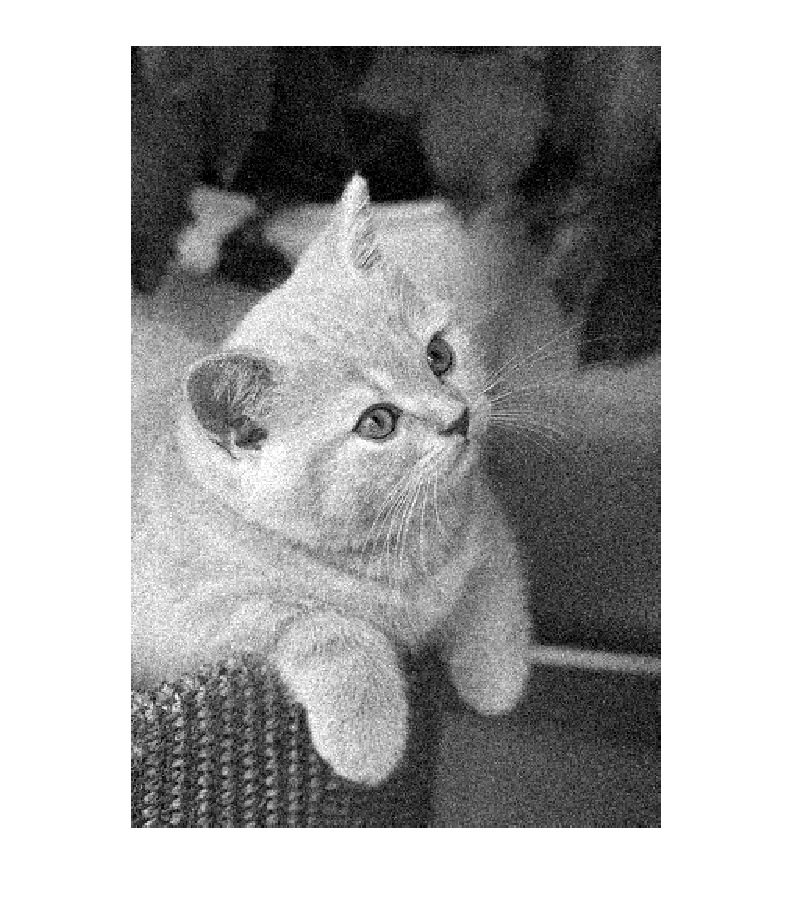}}\\
		\subfloat[PPDG]{
			\includegraphics[width=3.8cm,height=3.2cm]{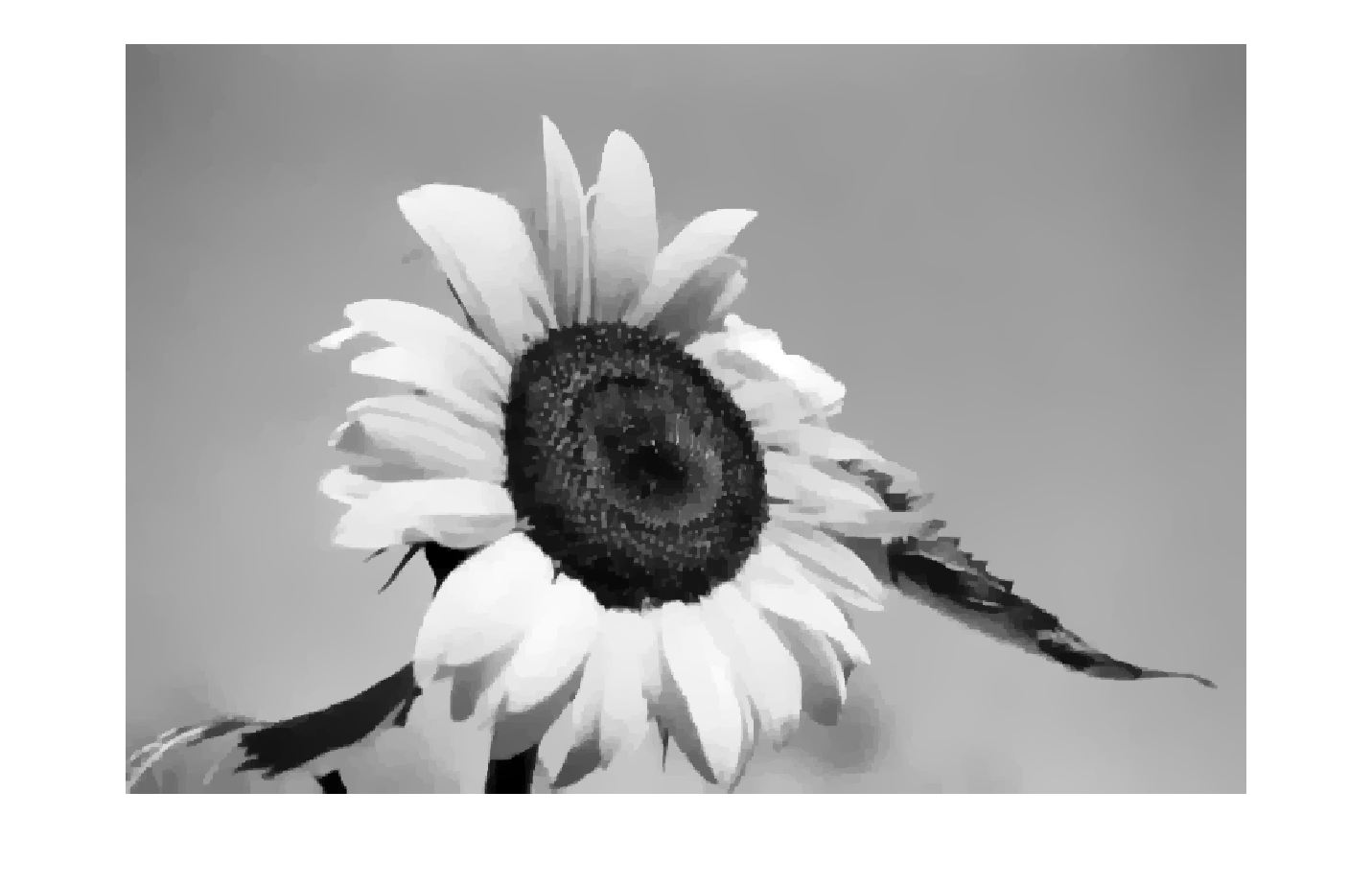}}	
		\subfloat[PPDG]{
			\includegraphics[width=3.8cm,height=3.2cm]{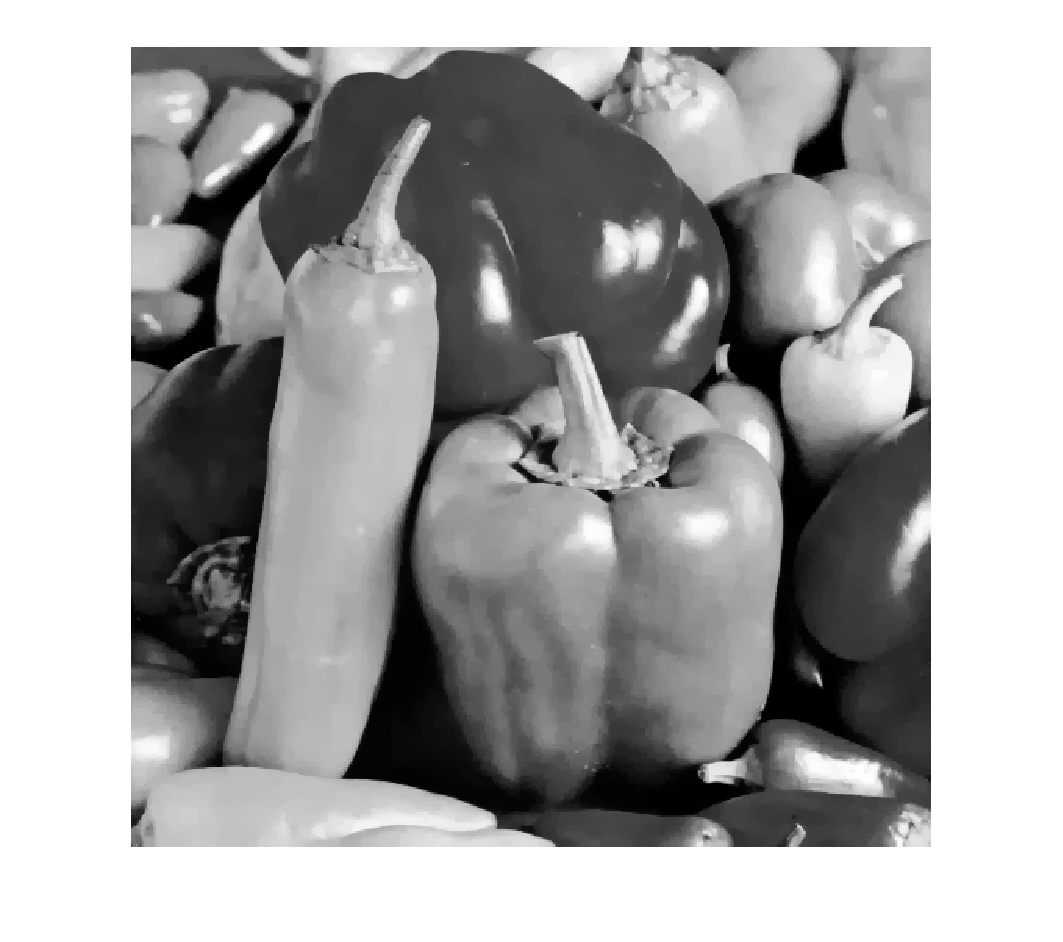}}		
		\subfloat[PPDG]{
			\includegraphics[width=3.8cm,height=3.2cm]{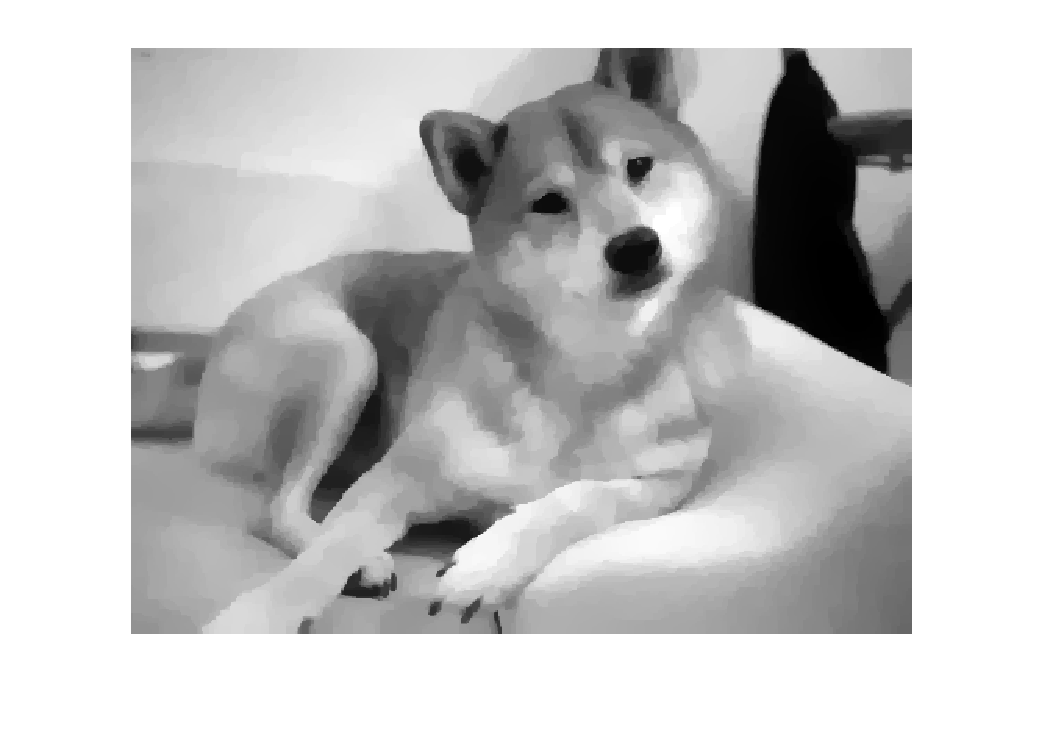}}		
		\subfloat[PPDG]{
			\includegraphics[width=3.8cm,height=3.2cm]{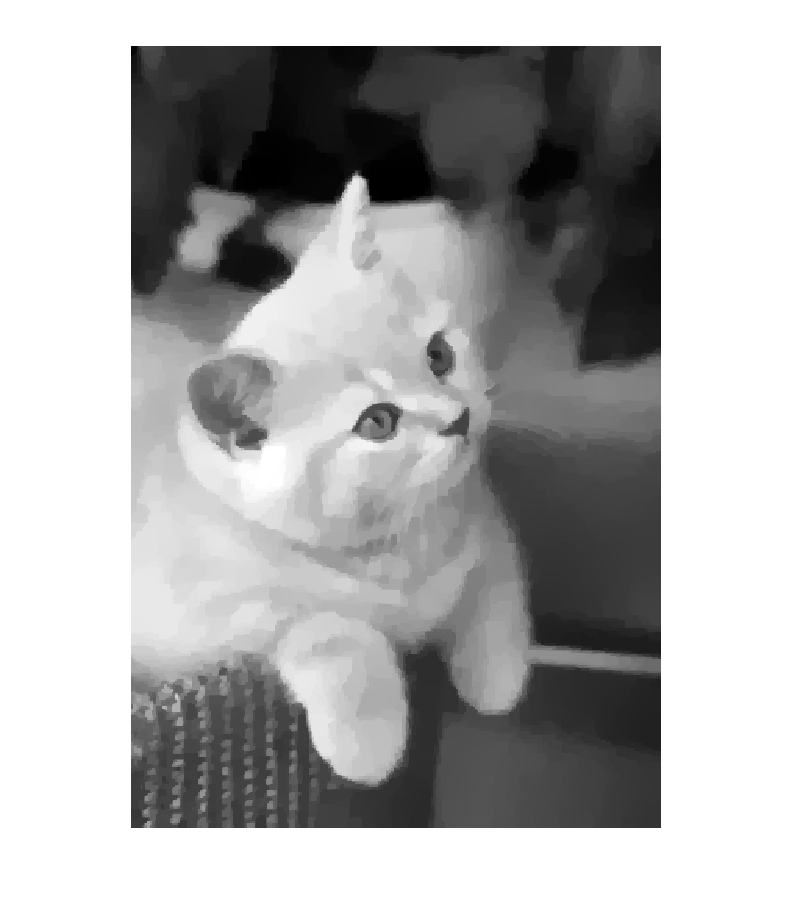}}\\
		\subfloat[ADMM]{
			\includegraphics[width=3.8cm,height=3.2cm]{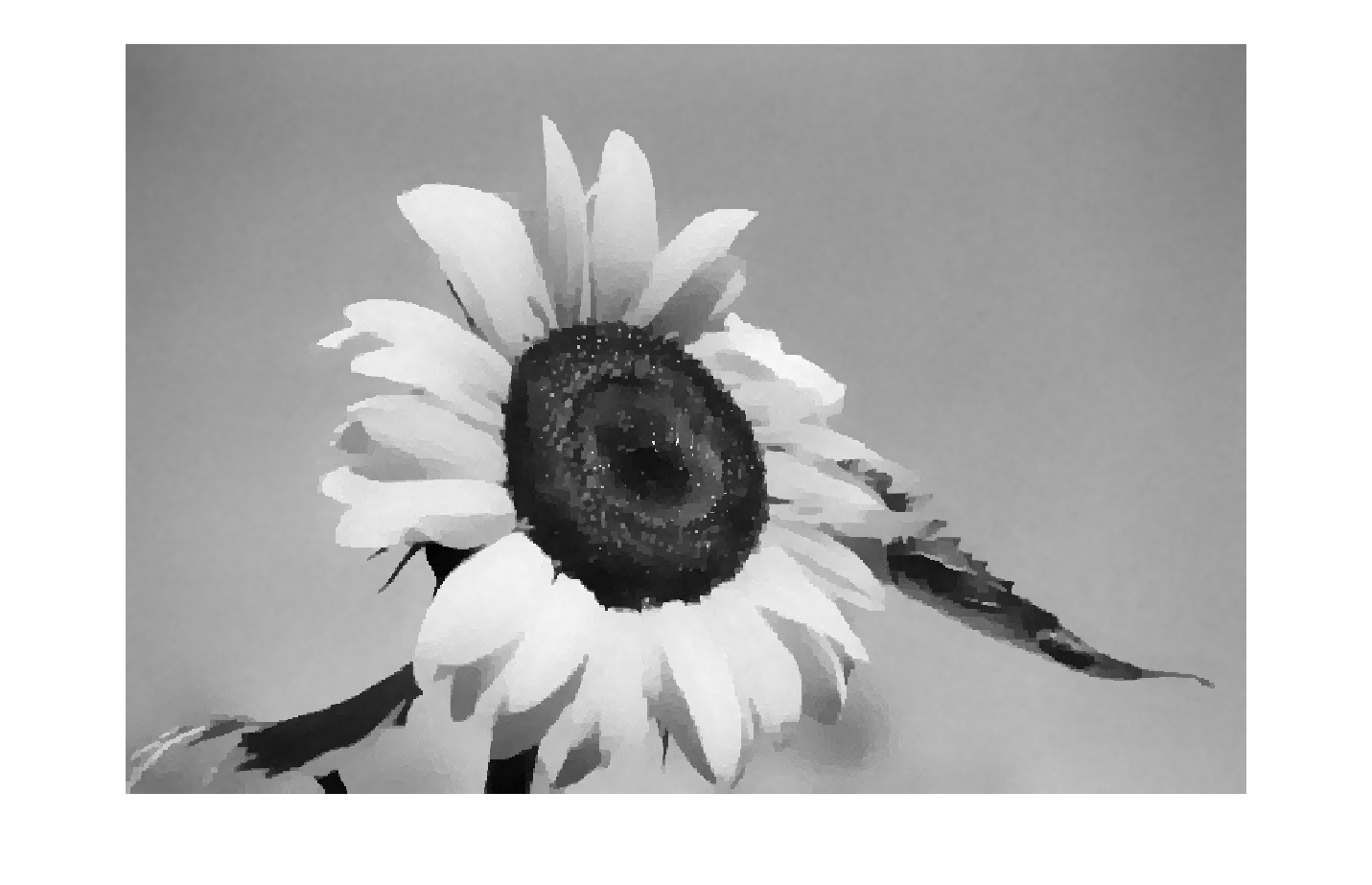}}
		\subfloat[ADMM]{
			\includegraphics[width=3.8cm,height=3.2cm]{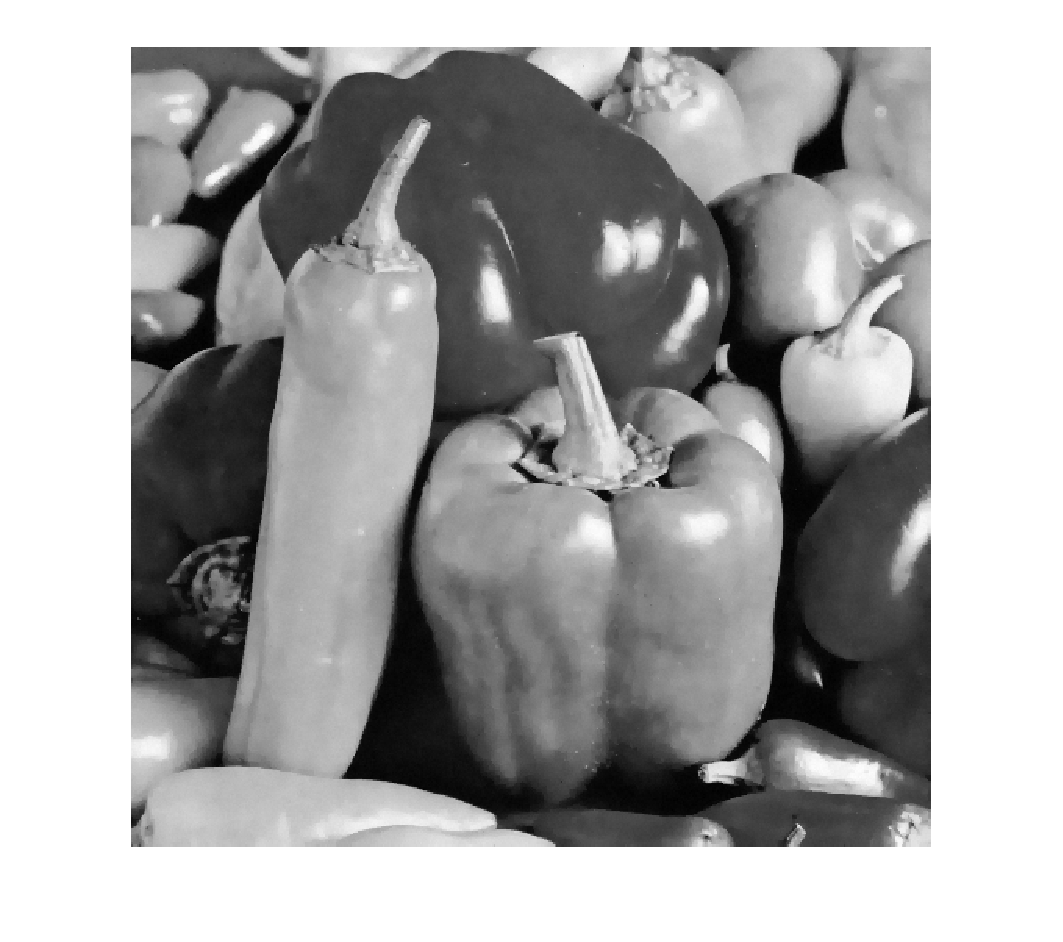}}
		\subfloat[ADMM]{
			\includegraphics[width=3.8cm,height=3.2cm]{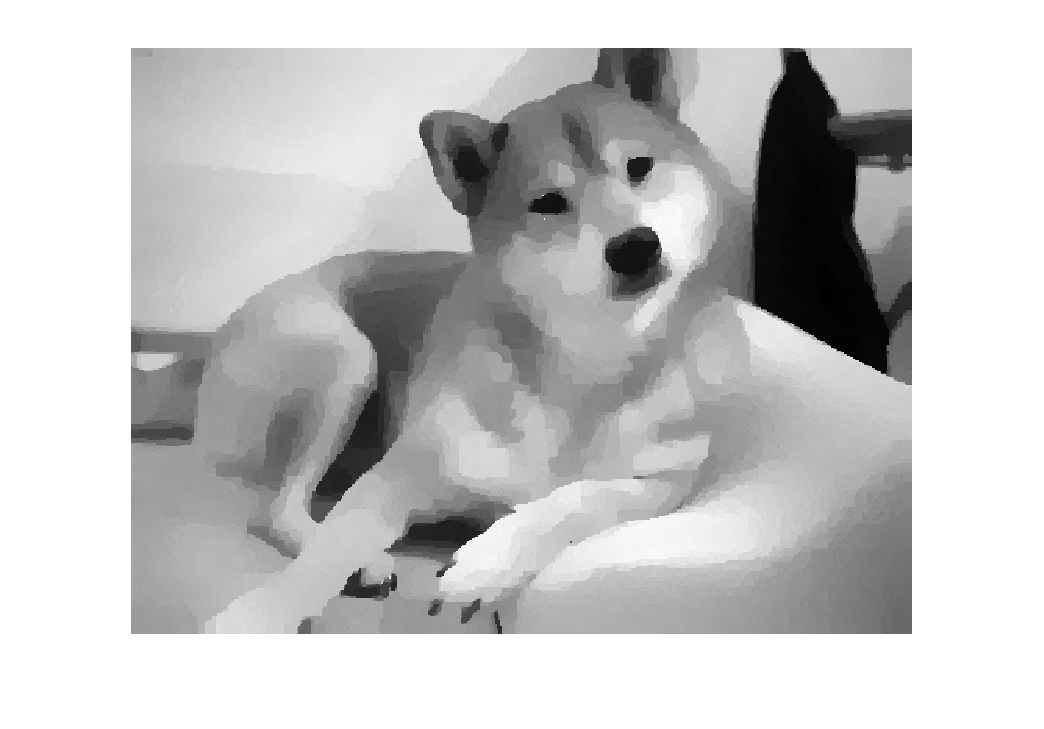}}
		\subfloat[ADMM]{
			\includegraphics[width=3.8cm,height=3.2cm]{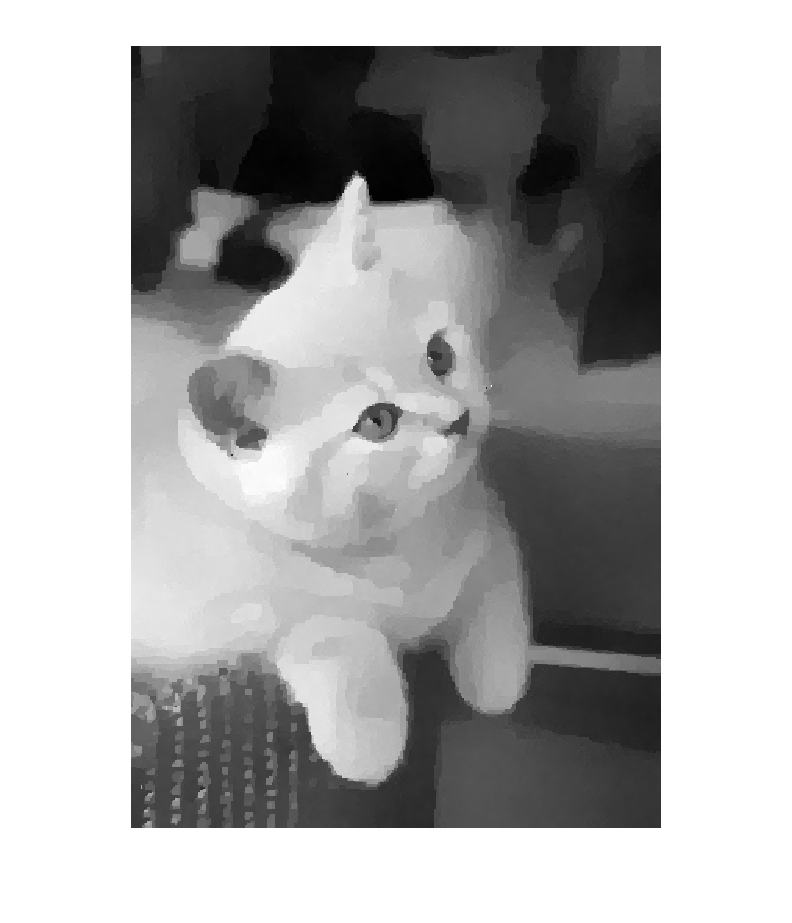}}\\
		\subfloat[PDHG]{
			\includegraphics[width=3.8cm,height=3.2cm]{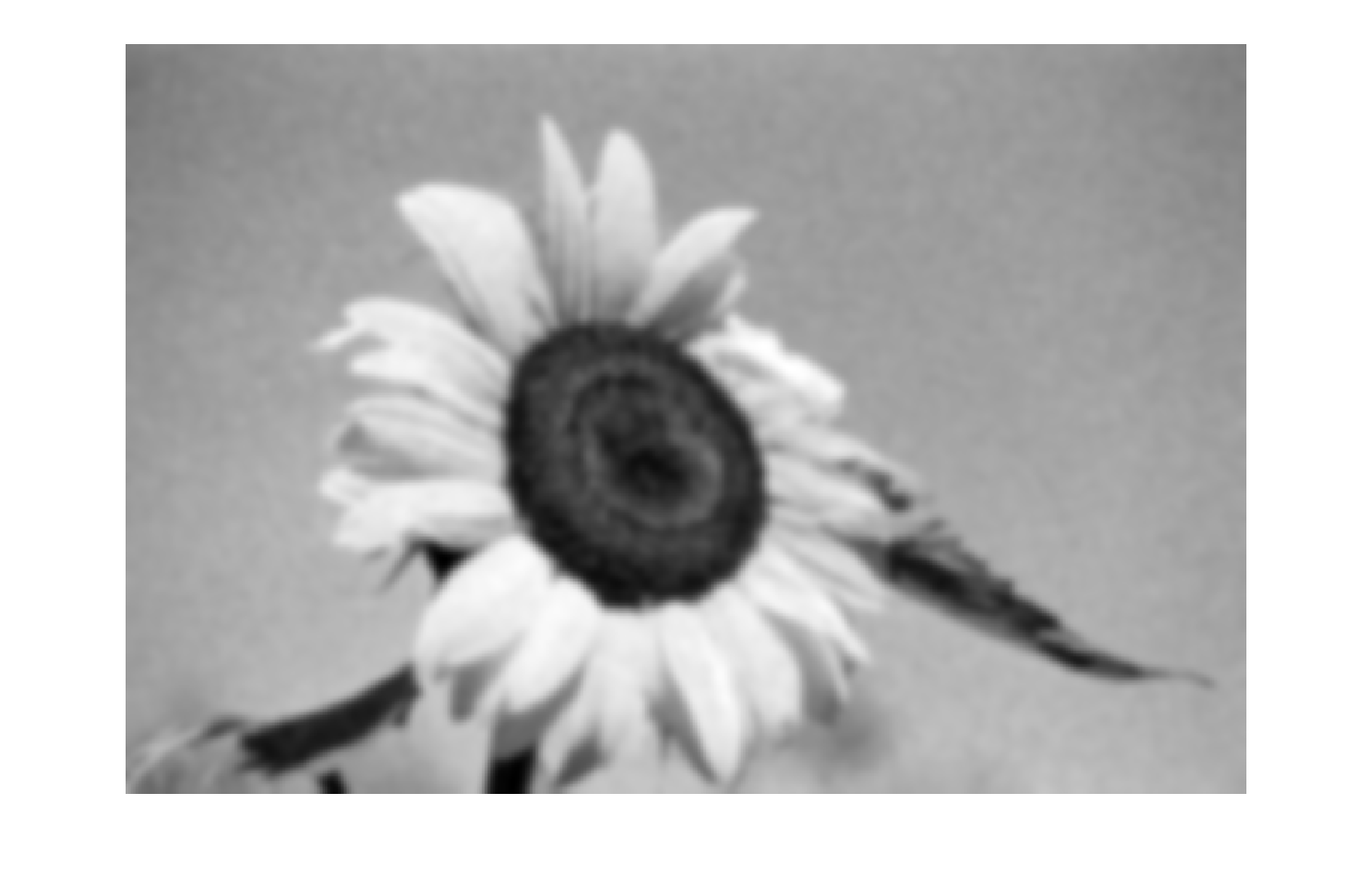}}
		\subfloat[PDHG]{
			\includegraphics[width=3.8cm,height=3.2cm]{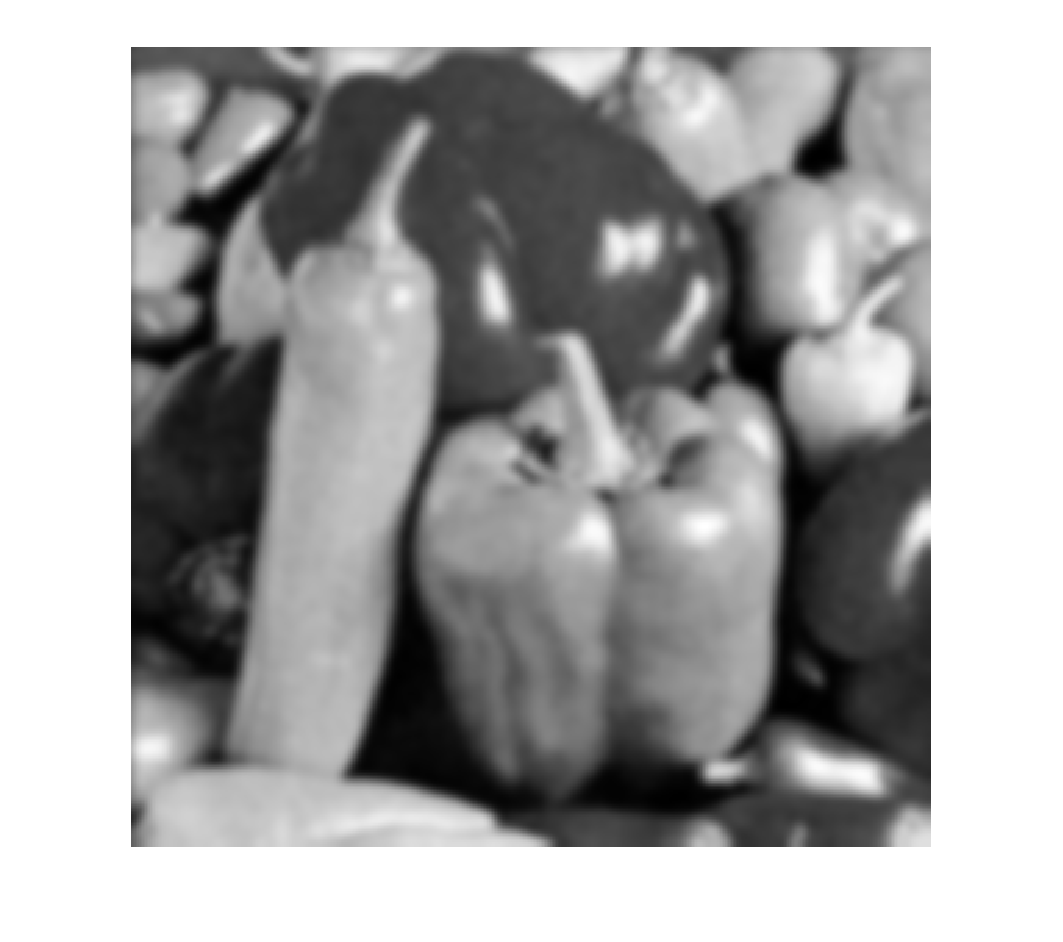}}
		\subfloat[PDHG]{
			\includegraphics[width=3.8cm,height=3.2cm]{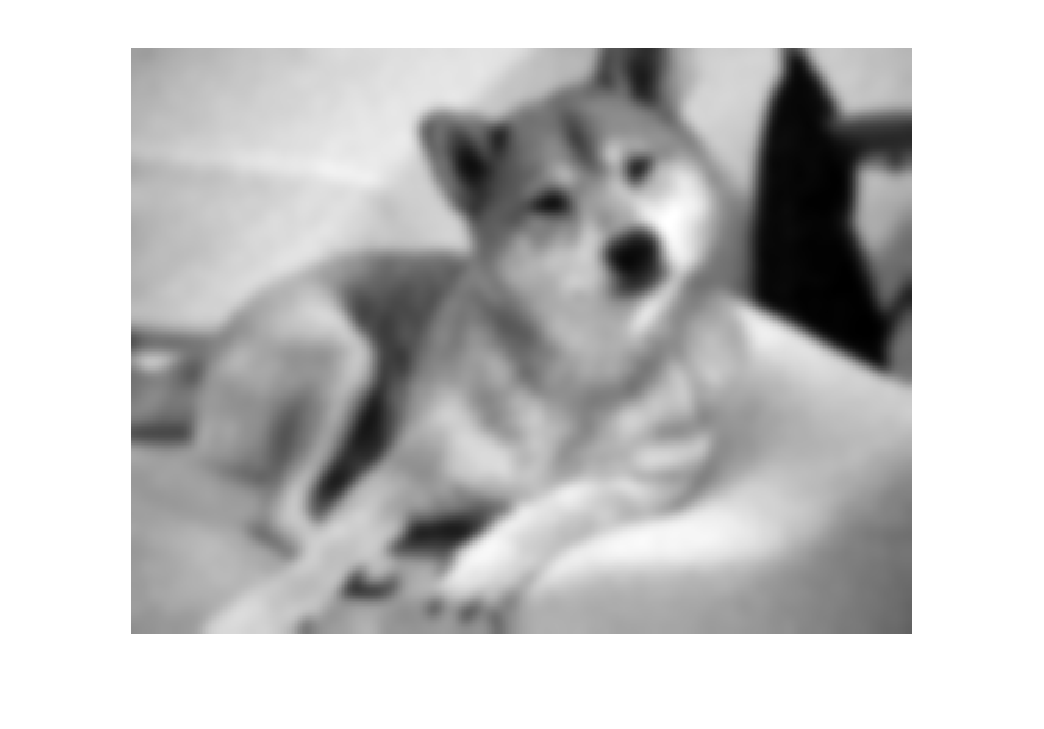}}
		\subfloat[PDHG]{
			\includegraphics[width=3.8cm,height=3.2cm]{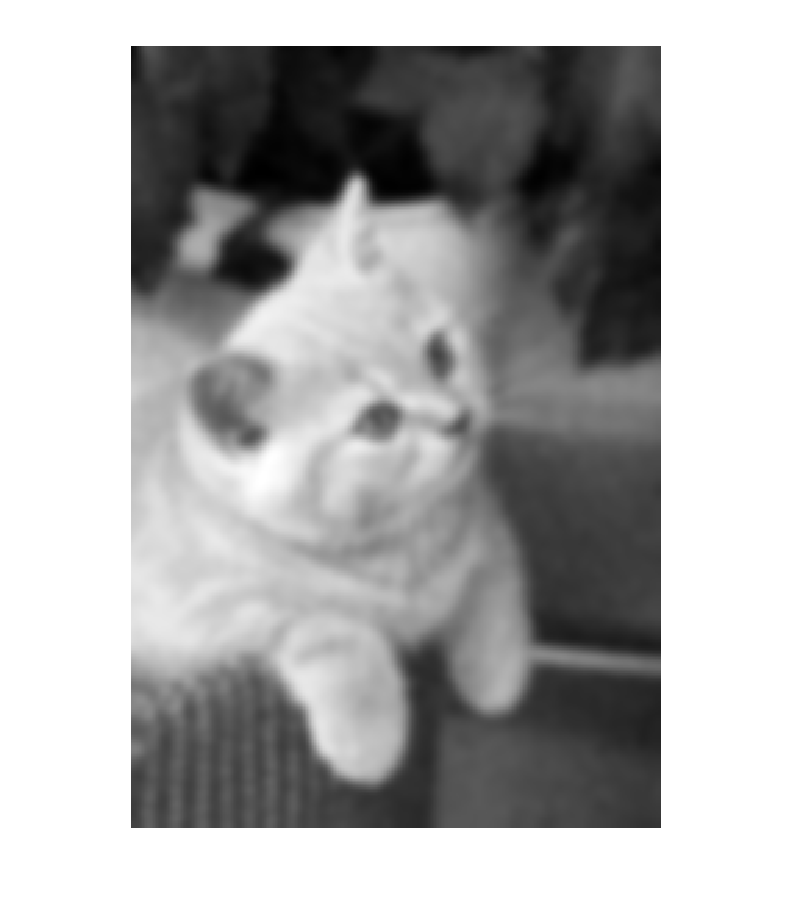}}	
	\end{tabular}
	\caption{The first row contains the original images; the second row represents the Gaussian noised images;
 the third, fourth and fifth rows  are the denoised images by PPDG, ADMM and PDHG, respectively.}
	\label{Fig-2}
\end{figure}

\begin{figure}[!htp]
	\centering
	\setlength{\belowcaptionskip}{-6pt}
	\begin{tabular}{cc}
		\subfloat[PPDG]{
			\includegraphics[width=5.8cm,height=5.2cm]{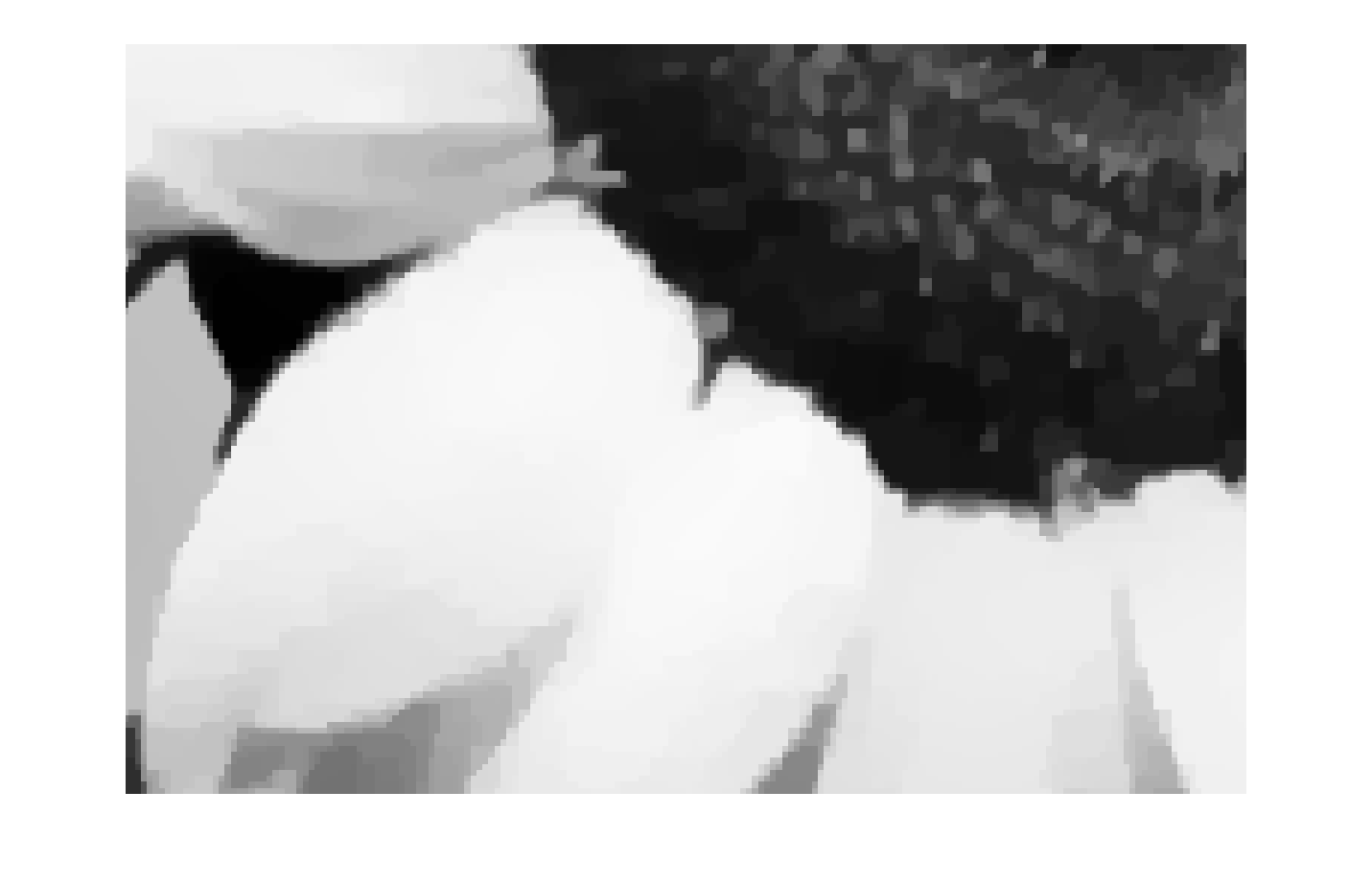}}&\hspace{-10mm}		
		\subfloat[ADMM]{
			\includegraphics[width=5.8cm,height=5.2cm]{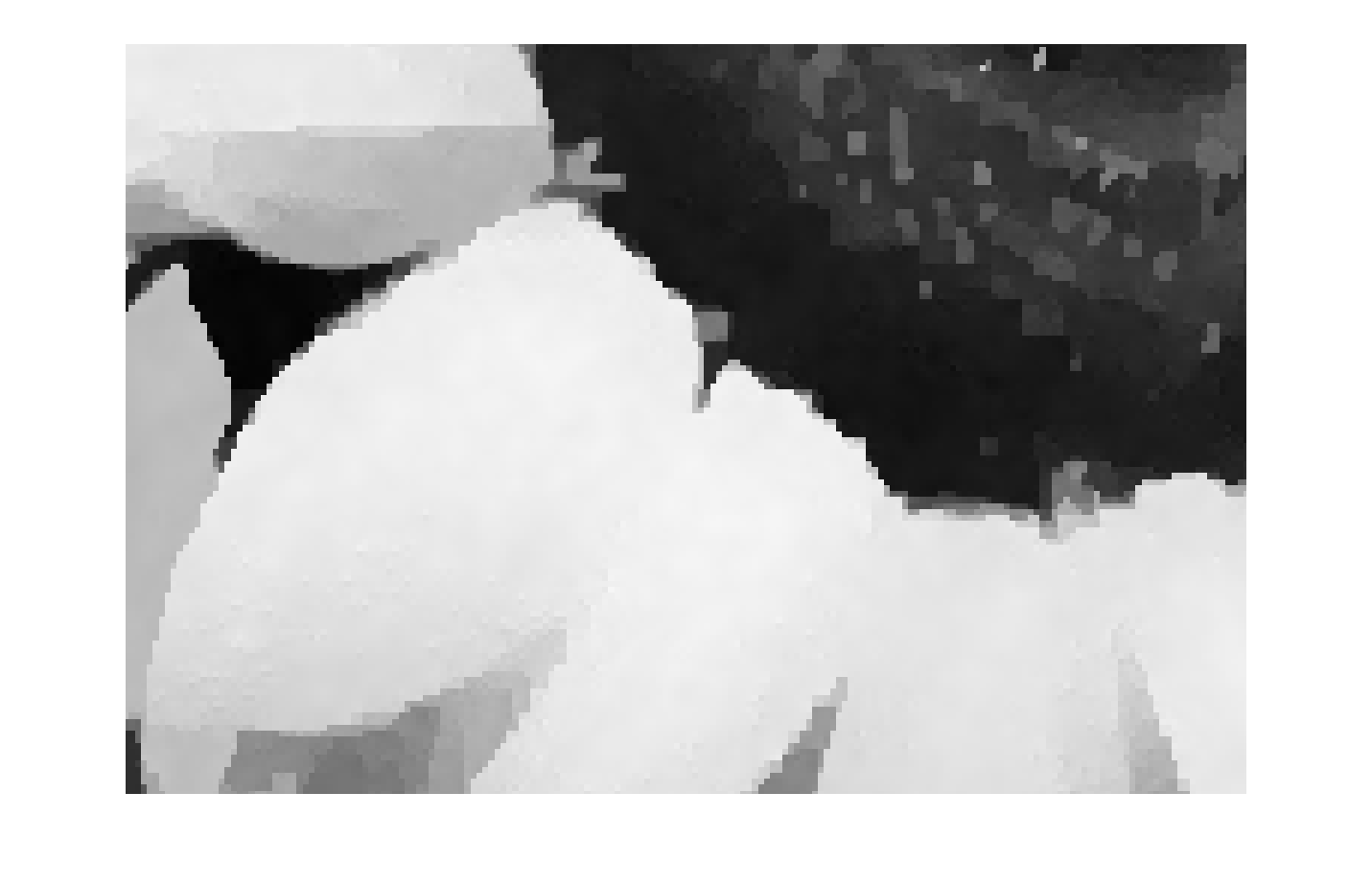}}\\
		\subfloat[PPDG]{
			\includegraphics[width=6.8cm,height=5.2cm]{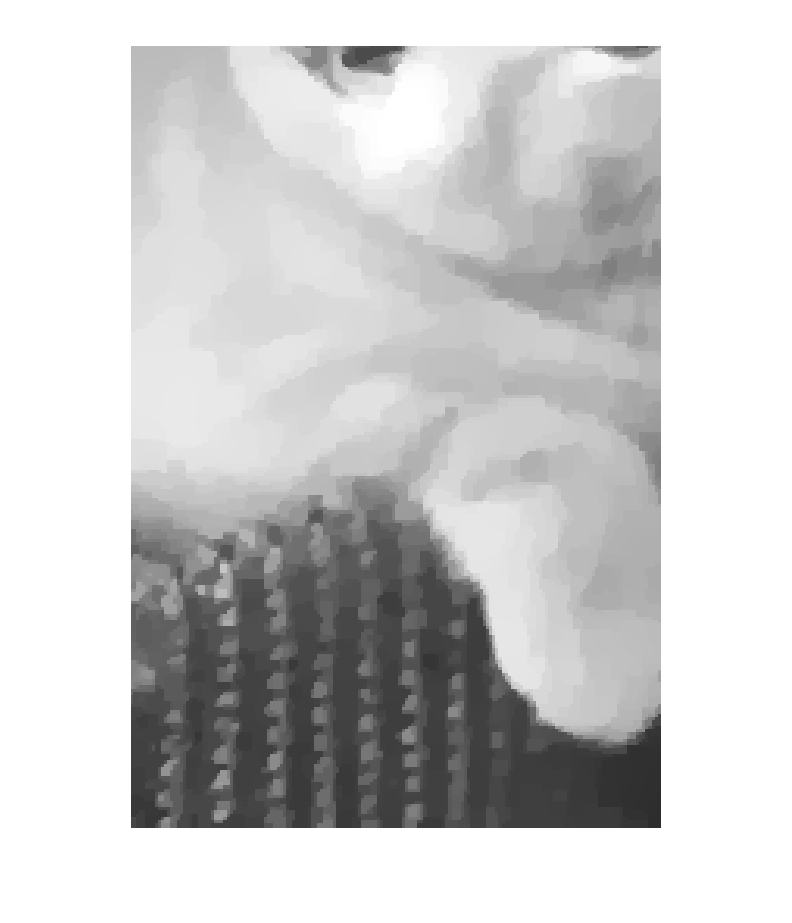}}&\hspace{-10mm}
		\subfloat[ADMM]{
			\includegraphics[width=6.8cm,height=5.2cm]{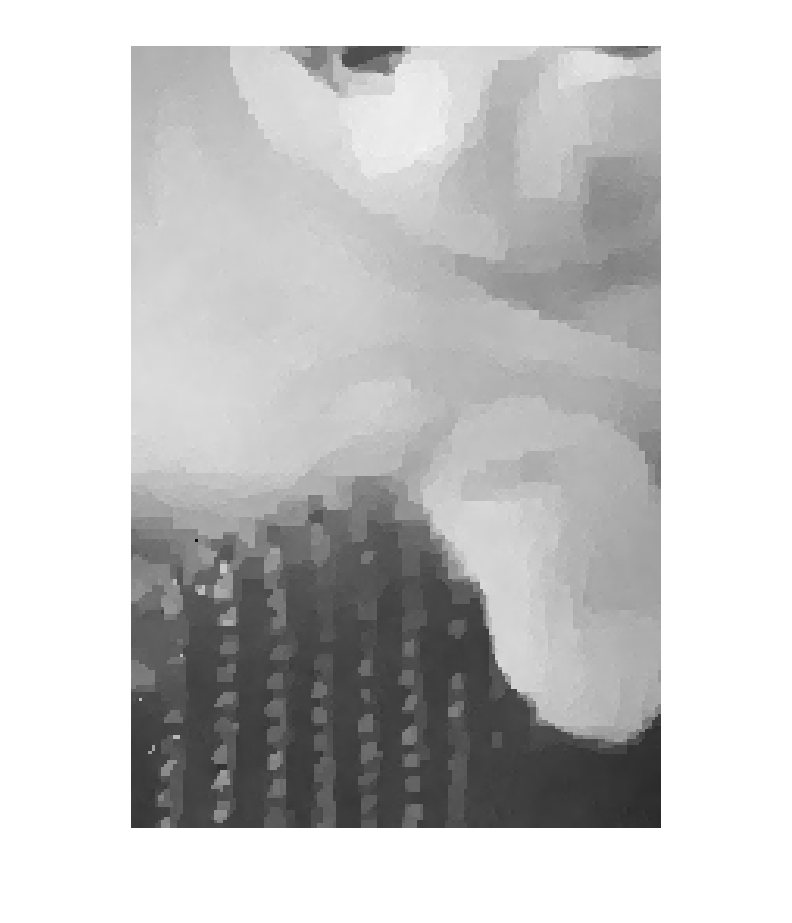}}	
	\end{tabular}
	\caption{Images (a), (b), (c) and (d) are the partial enlargements  of  (i), (m), (l) and (p) in Figure \ref{Fig-2}, respectively.}
	\label{Fig-3}
\end{figure}

 \begin{table}[!htp]
	\centering
	\caption{PSNR and running time (seconds) of four images (in Figure \ref{Fig-2})  denoised by PPDG, ADMM and PDHG. Bold numbers are the best results.}
	\begin{tabular}{|c||c|c|c|c|}
			\hline
		Algorithm &   Sunflower    & Peppers       &  Dog  &    Cat    \\[2pt]
		\hline
		PPDG & $\mathbf{27.5878}|\mathbf{3.8}$ & $26.7780|\mathbf{2.3}$ & $\mathbf{26.2217}|\mathbf{1.6}$ & $\mathbf{28.0326}|\mathbf{1.4}$ \\[2pt]
		ADMM & $27.2407|5.3$ & $\mathbf{27.0784}|3.3$ & $25.4212|2.6$ & $27.1934|2.4$ \\[2pt]
		PDHG & $24.0718|4.6$ & $20.0308|3.3$ & $23.4053|2.3$ & $25.8938|2.2$ \\
		\hline
	\end{tabular}
	\label{table:psnr}
\end{table}

\subsection{Deep learning for image classification}
In this subsection, we employ a one-hidden layer deep neural network for image classification using the dataset \texttt{CIFAR-$10$}\footnote{The dataset can be found in \url{https://www.cs.toronto.edu/~kriz/cifar.html}}, which consists of  $60000$ color images   of size $32\times 32$ divided into $10$ classes. Within this dataset,  $50000$ images have already been designated for training, while the remaining $10000$ are reserved for testing.
The one-hidden layer neural network consists of an input layer, a hidden layer and an output layer.  The hidden layer size  is $175$.

We adopt the following notation:
\begin{itemize}
\item $N$: the number of input samples, $m$: the number of neurons in the hidden layer,  $d$: the dimension of each input sample;
\item  $x_{i}\in\R^d$: the $i$-th input sample, $i=1,\cdots,N$, $z_{ij}\in\R$: the $j$-th element of  the actual output, $y_{ij}\in\R$: the $j$-th element of the desired output, $j=1,\cdots,10$;  	
\item $w_{kl}$: the weights of the connections between the input nodes and the hidden layer, $v_{jk}$: the weights of the connections between the hidden layer neurons and the
output neuron, $b_k,b_j$: the bias parameters of hidden layer and output layer, $k=1,\cdots,m$, $j=1,\cdots,10$;
\end{itemize}

Training the neural network amounts to obtaining the value of the model parameter $(w,b)$ such that, for each input data $x$, the output $z$ of the model predicts the real value $y$ with satisfying accuracy.
To achieve this, it is required to solve the following finite-sum optimization problem,
\be\label{eq:deep}
\min_{w,b} \frac{1}{N}\sum_{i=1}^{N}f_i(w,b)+\lambda(\|w\|_1+\|b\|_1),
\ee
where $$f_i(w,b)=-\sum_{j=1}^{10}y_{ij}\log(z_{ij})$$ is the cross-entropy loss function,
$$z_{ij}=g_2\left(\sum_{k=1}^mv_{jk}g_1\left(\sum_{l=1}^dw_{kl}x_{il}+b_k\right)+b_j\right),$$
 and $\lambda>0$ represents a regularization parameter. Here,
$g_1$ is the sigmoid activation function and
$g_2$ is the softmax activation function.

Choose a normalized vector drawn from standard normal distribution  as the initial point $x^0$, and set $\lambda=1e-4$.  Combining with SAGA, SVRG and SARAH estimators,  we apply SPPDG (Algorithm \ref{alg:SPPDG}), the stochastic linearized ADMM (SADMM) proposed recently in \cite{BLZ2021} and the stochastic proximal gradient method (SPG), to train the neural network on the training set. After training, we evaluate the classification performance of this  neural network  on test set. The numerical results are presented in
Figure  \ref{Fig-6},  which displays  the training loss, training error and test error as functions of the    total number of propagations for all methods. By observing this figure, all three SPPDG methods obviously outperform the methods associated with  SADMM and SPG, and the gradient estimator SVRG seems more competitive than SAGA and SARAH.

\begin{figure}[htp]
	\centering
	\setlength{\belowcaptionskip}{-6pt}
	\begin{tabular}{ccc}
		\subfloat[Training loss vs Props.]{
			\includegraphics[width=5cm, height=4cm]{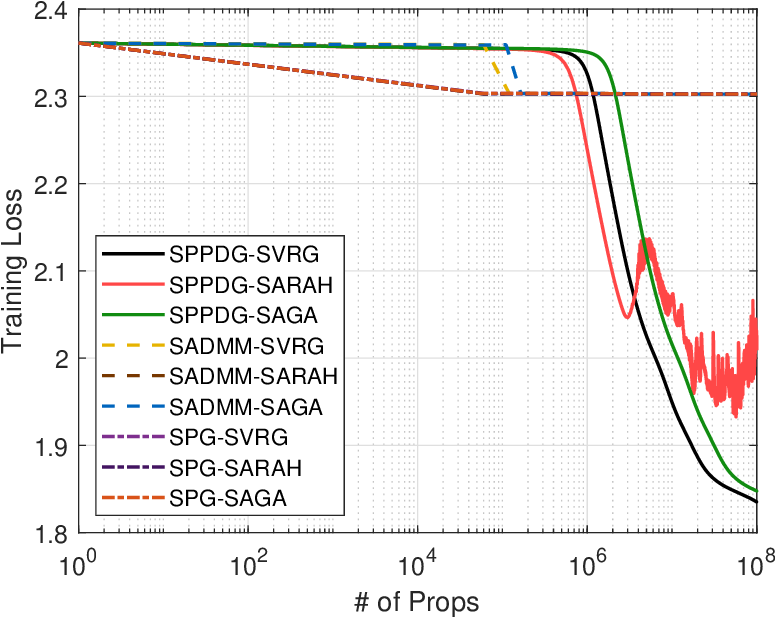}} &\hspace{-4mm}
		\subfloat[Training error vs Props.]{
			\includegraphics[width=5cm, height=4cm]{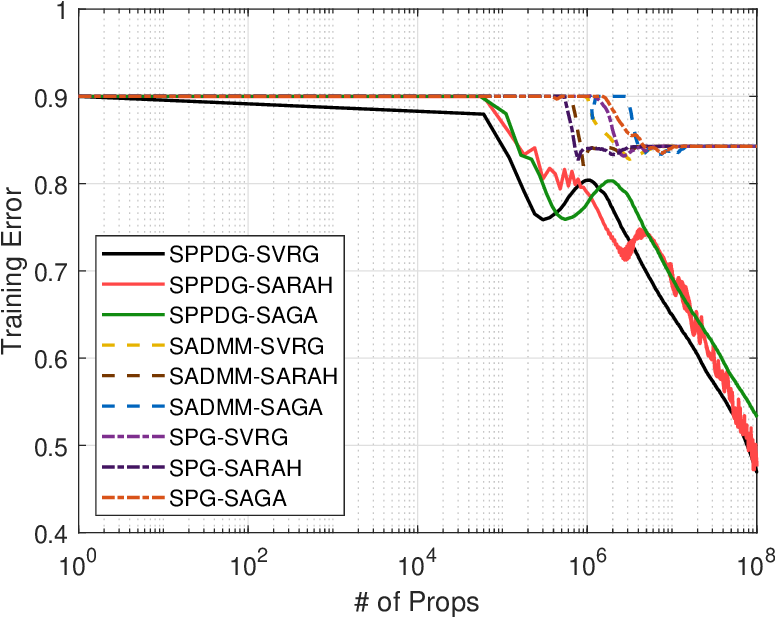}} &\hspace{-4mm}
		\subfloat[Test error vs Props.]{
			\includegraphics[width=5cm,height=4cm]{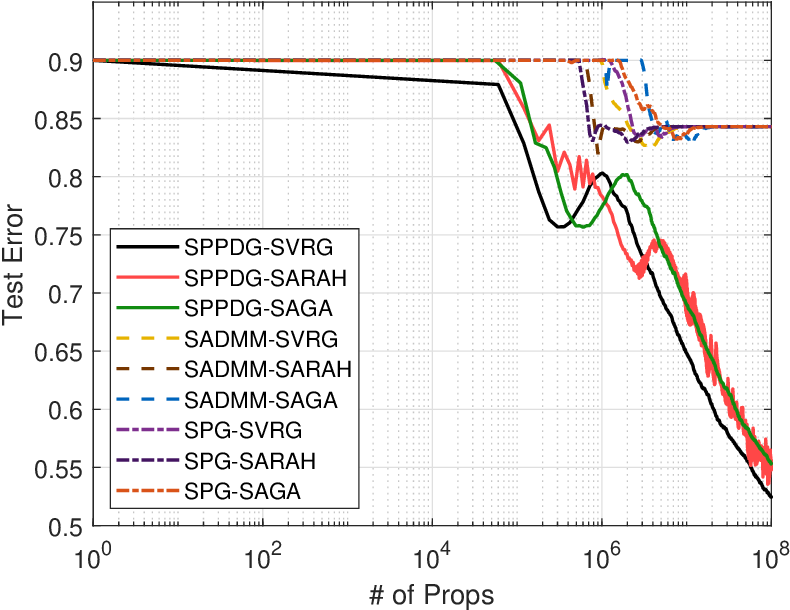}}
	\end{tabular}
	\caption{Comparison of SPPDG, SADMM and SGD for image classification using one-hidden layer neural network.}
	\label{Fig-6}
\end{figure}

\subsection{Nonconvex graph-guided fused  lasso}
In this subsection, we consider the following problem
\be\label{eq:sigmoid-re}
\min_{x\in\hat{\cD}} {\frac{1}{N}\sum_{i=1}^{N}f_i(x)+\lambda\|Ax\|_p^p}.
\ee
Here, $A=[V;I]\in \R^{m\times n}$ with  $V\in\R^{n\times n}$
is the sparsity pattern of the graph obtained by
sparse inverse covariance estimation (see \cite{Frie2008}), the set $\hat{\cD}$ is defined as $\hat{\cD}=\{x\in\R^{n}: \|Ax\|_{\infty}\leq r\}$,  $f_i(x)=1-\tanh(b_i\cdot\langle a_i,x\rangle)$ is the sigmoid loss function which is nonconvex, and  $\|u\|_p$, $p\in(0,1)$ is the $\ell_p$-norm.  Evidently,  problem \eqref{eq:sigmoid-re} can be categorized as an instance  of the fully nonconvex finite-sum optimization (\ref{eq:p-finite-sum}) with $h(u)=\lambda\|u\|_p^p+\cI_{\cD}(u)$, $\cD=\{u:\|u\|_{\infty}\leq r\}$.
In what follows, we choose $\lambda=1e-4$ and $r=1$.

\begin{table}[!htp]
	\centering
	\caption{Datasets used in graph-guided fused lasso.}
	\begin{tabular}{|c|c|c|c|c|}
		\hline
		Dataset & Data  $N$ & Variable $n$ & {Density}  \\
		\hline
		$\mathtt{CINA}$ & 16033 & 132 & 29.56\% \\[2pt]
		$\mathtt{MNIST}$ & 60000 & 784 & 19.12\% \\[2pt]
		$\mathtt{gisette}$ & 6000 & 5000 & 12.97\% \\
		$\mathtt{covtype}$ & 581012 & 54 & 22.12\% \\
		\hline
	\end{tabular}
	\label{table:datasets}
\end{table}

In this experiment, we test problem \eqref{eq:sigmoid-re} by datasets \texttt{CINA}\footnote{The dataset is available in \url{http://www.causality.inf.ethz.ch/data/CINA.html}}, \texttt{MNIST}\footnote{The dataset is available in \url{http://yann.lecun.com/exdb/mnist}}, \texttt{gisette} \cite{GG2004} and \texttt{covtype}\footnote{The dataset is available in \url{https://datahub.io/machine-learning/covertype}},  which are summarized  in Table \ref{table:datasets}. At the $k$-th iteration of SPPDG, we also apply \eqref{eq:proximal-l0a} to compute an approximate point of the extended proximal mapping  $\prox_{h^*}^{M}(y^k+M^{-1}A(2x^{k+1}-x^k))$ where $\prox_{\beta h^*}(\cdot)$ is calculated according to Example \ref{ex:lp}.

The numerical result is  displayed in Figure \ref{Fig-4} with initial point $x^0=0$, $q=0.5$ and  a fixed mini-batch sample size $\lfloor0.01N\rfloor$.
Due to the fully nonconvex structue and the existence of the linear operator $A$, problem \eqref{eq:proximal-l0a} cannot be solved by SADMM and SPG directly as in the previous subsection. However, we can observe from Figure \ref{Fig-4} that SPPDG (Algorithm \ref{alg:SPPDG} with gradient estimators SAGA, SVRG and SARAH)
is able to solve problem \eqref{eq:sigmoid-re} efficiently.  It is also shown that by running same number of iterations, SPPDG-SVRG outperforms both  SPPDG-SAGA and SPPDG-SARAH significantly. However,  with the same CPU time the performance of SPPDG-SAGA is  better than  that of SPPDG-SVRG and SPPDG-SARAH.
\begin{figure}[!htp]
	\centering
		\subfloat[Obj. vs Iter. (CINA)]{
			\resizebox*{5.8cm}{!}{\includegraphics{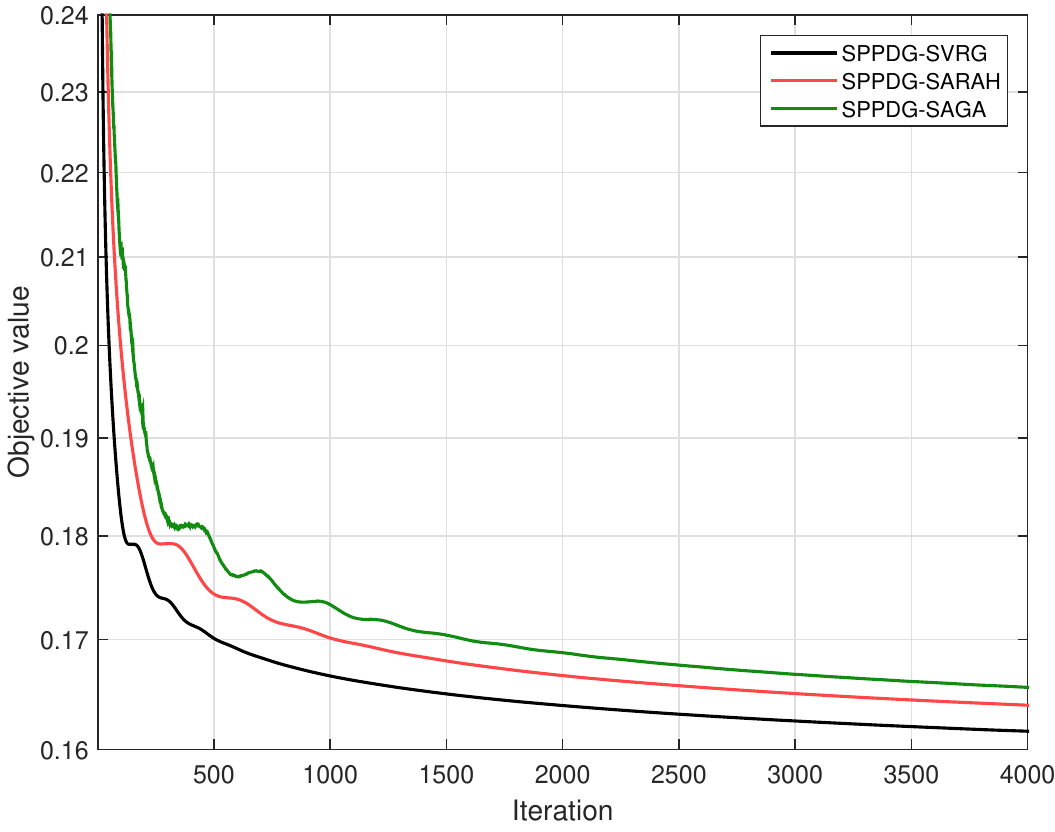}}} \hspace{5pt}
		\subfloat[Obj. vs Time (CINA)]{
			\resizebox*{5.8cm}{!}{\includegraphics{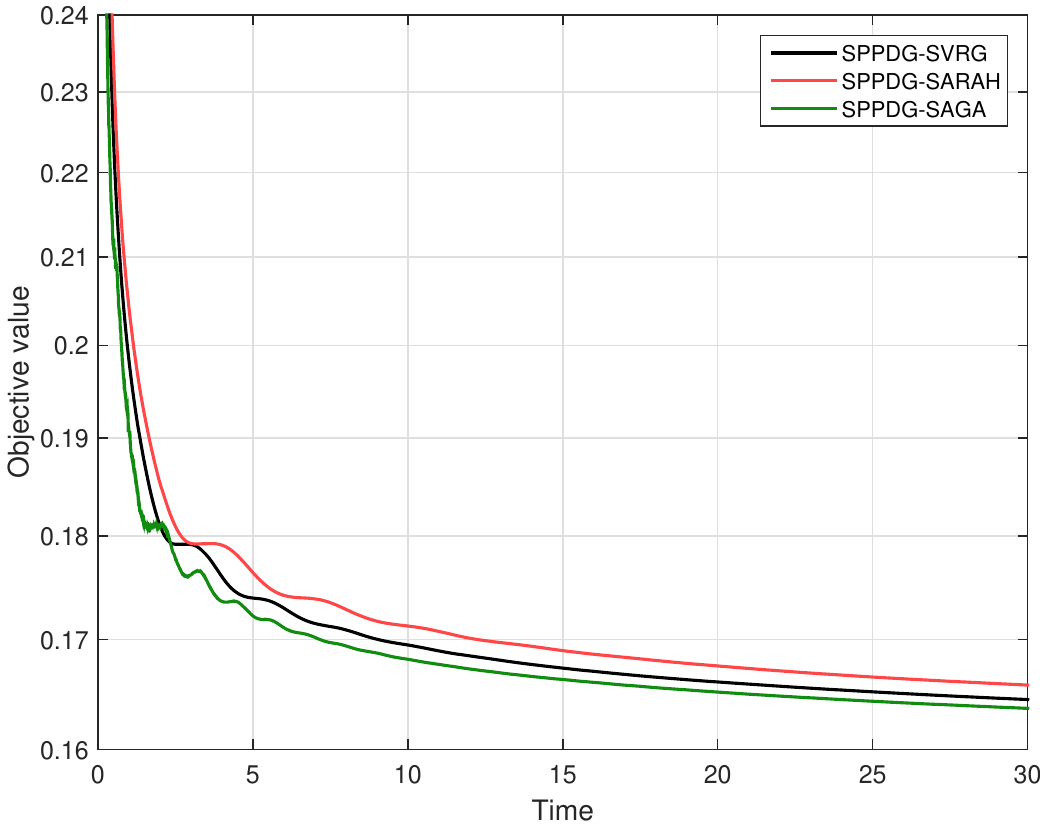}}}\hspace{5pt}\\
\subfloat[Obj. vs Iter. (MNIST)]{
	\resizebox*{5.8cm}{!}{\includegraphics{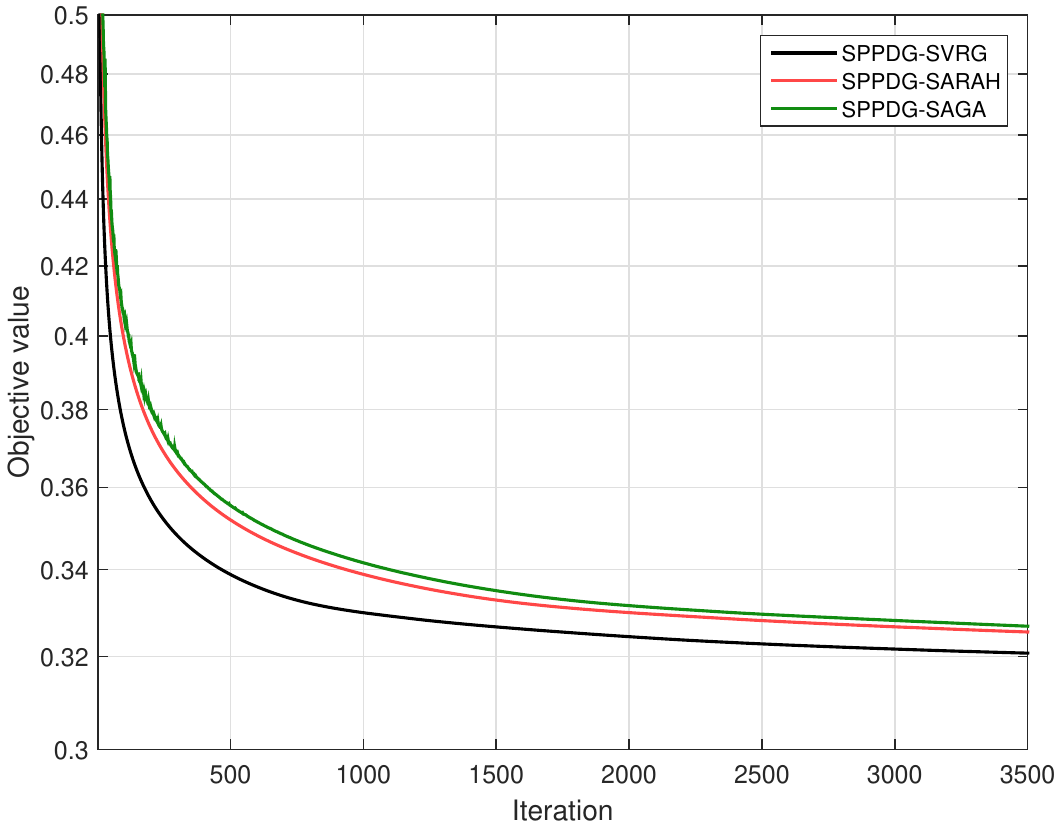}}}\hspace{5pt}	
\subfloat[Obj. vs Time (MNIST)]{
	\resizebox*{5.8cm}{!}{\includegraphics{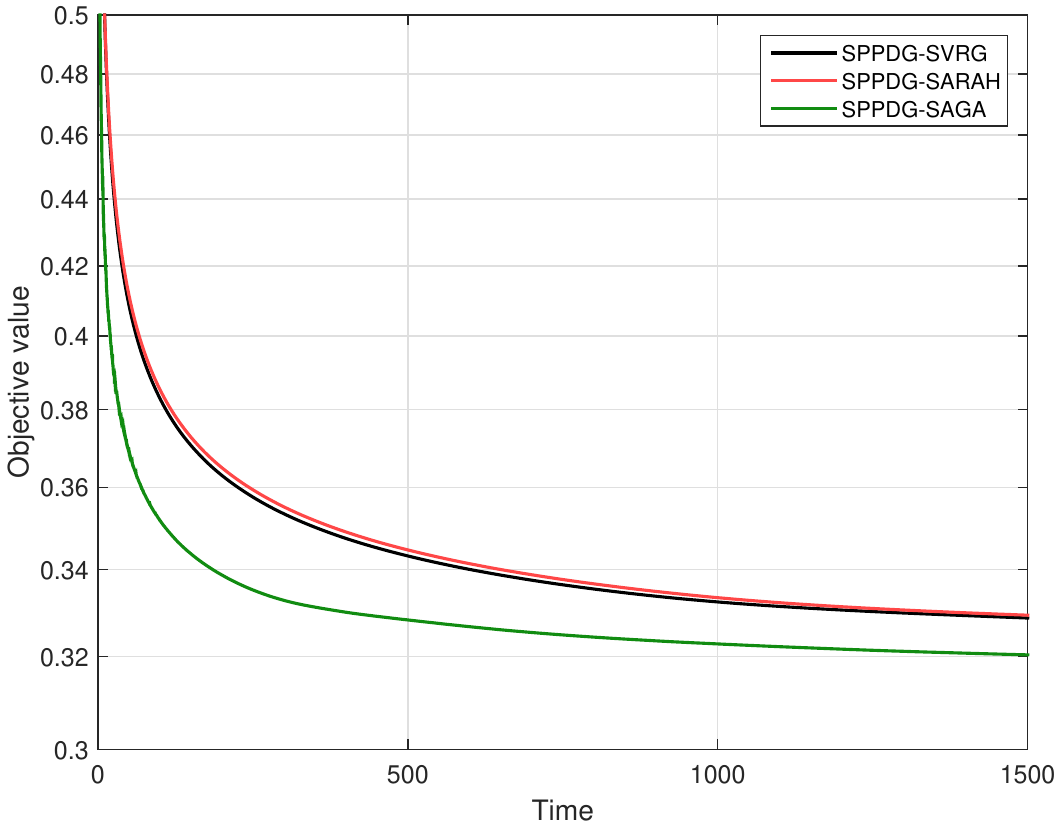}}}\hspace{5pt}\\
\subfloat[Obj. vs Iter. (gisette)]{
	\resizebox*{5.8cm}{!}{\includegraphics{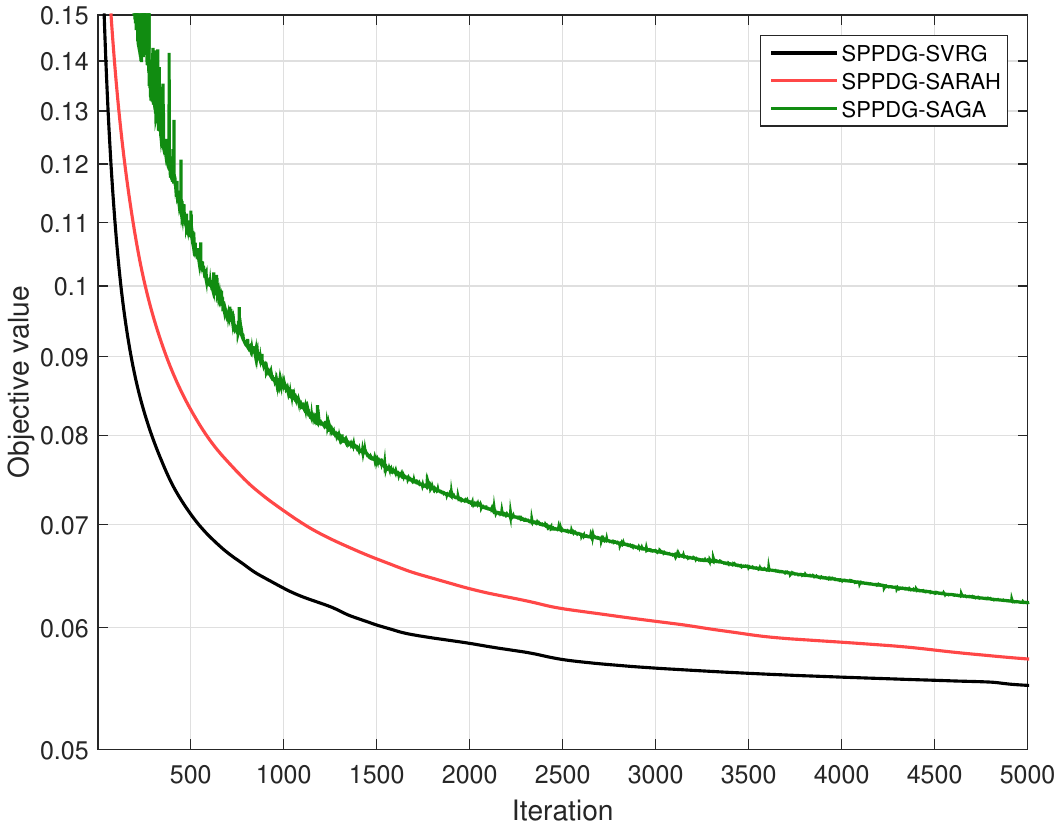}}}\hspace{5pt}		
\subfloat[Obj. vs Time (gisette)]{
	\resizebox*{5.8cm}{!}{\includegraphics{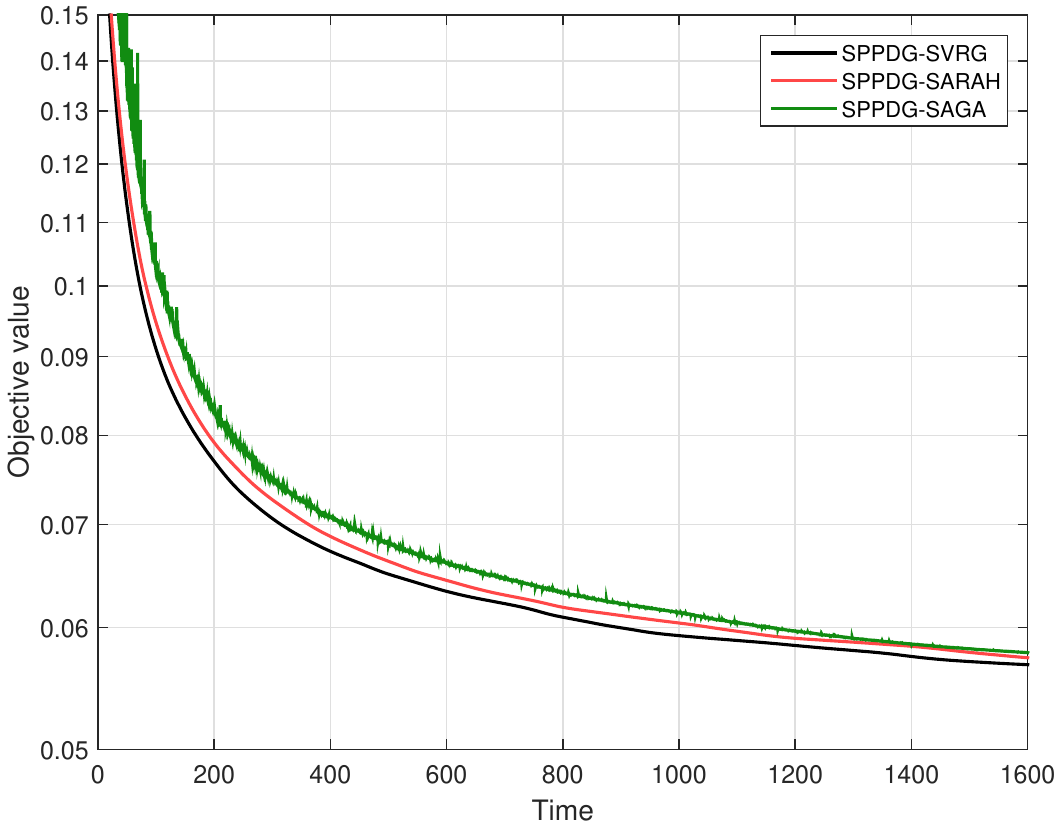}}}\hspace{5pt}\\
\subfloat[Obj. vs Iter. (covtype)]{
	\resizebox*{5.8cm}{!}{\includegraphics{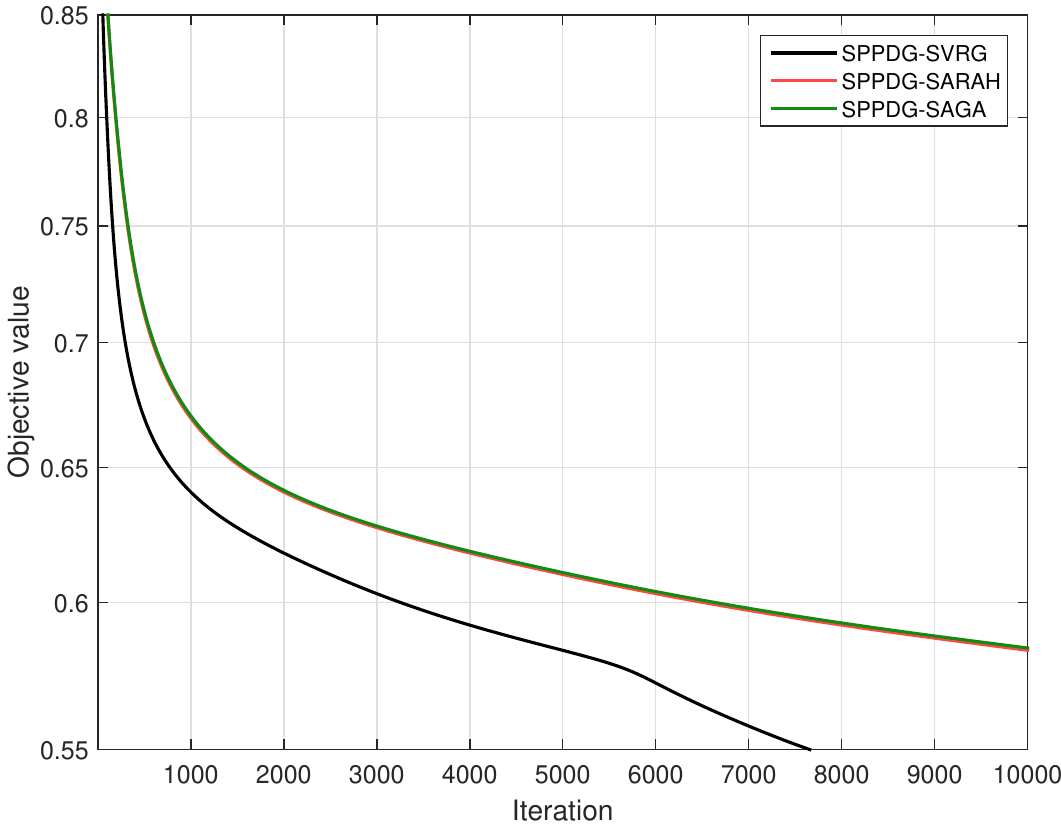}}}\hspace{5pt}		
\subfloat[Obj. vs Time (covtype)]{
	\resizebox*{5.8cm}{!}{\includegraphics{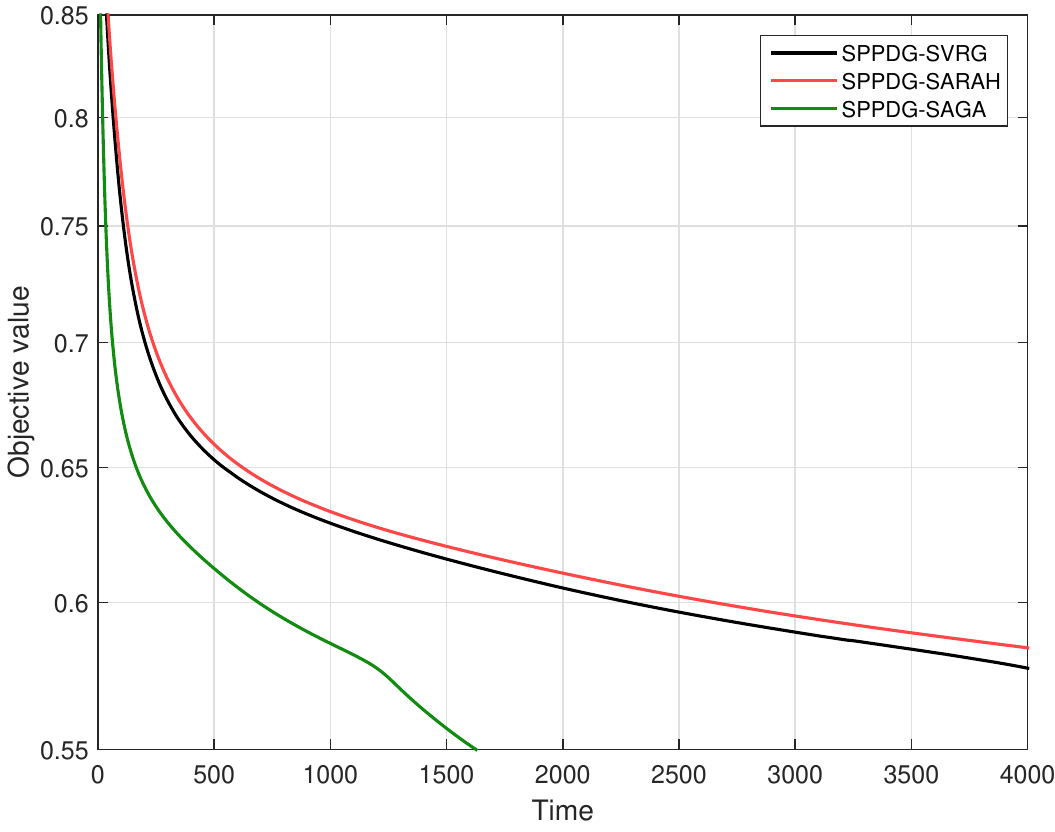}}}
	\caption{SPPDG for nonconvex graph-guided fused lasso.}
	\label{Fig-4}
\end{figure}

\section{Conclusion}
%

In this paper, we delve into the exploration of the first-order primal-dual methods  for  composite optimization  and nonconvex finite-sum optimization in the fully nonconvex setting. Inspired by the existing first-order primal-dual methods for convex optimization, with the help of conjugate duality, we propose a preconditioned primal-dual gradient method and its stochastic approximate variant.
The proposed methods are shown to be effective on a variety of nonconvex applications.  

Motivated by the rapid development of convergence analysis of various nonconvex optimization algorithms in recent years, we have derived the convergence results for the proposed algorithms  in the context of Kurdyka-\L{}ojasiewicz condition. Notably, the analysis of the stochastic algorithm for finite-sum optimization relies heavily  on  the properties of the variance reduced gradient estimators. Consequently,  it is not trivial to extend the technique in this paper to investigate the convergence of stochastic algorithms for general nonconvex stochastic optimization problems, which is left to future research.


\vspace{1cm}
\section*{Acknowledgement}
This work was partially supported by the National Key R\&D Program of China (No. 2022YFA1004000), the Major Key Project of PCL (No. PCL2022A05)  and the National Natural Science Foundation of China (Nos. 12271076 and 11271278).

\bibliographystyle{plain}
\bibliography{ref}



\appendix
\section{Appendix}	

\subsection{Examples of proximal mappings of conjugate functions}\label{Appendix-prox}
We will now list a set of examples on the proximal mappings of  conjugate functions associated with some well-known regularizers. All of them are obtained by direct calculations based on the definitions of proximal mapping and conjugate function.


\begin{example}[$\ell_1$ norm]\label{ex:l1}
For a constant $\lambda>0$, let $h(x)=\lambda \|x\|_1$. The conjugate function of $h$ is given by
\[h^*(x)=\lambda\cI_{\cC}(x/\lambda), \text{ with } \cC=\{x:\|x\|_{\infty}\leq 1\},\]
where $\cI_{\cC}(\cdot)$ is the indicator function.
Then, the proximal mapping of $h^*$ is
\[
\prox_{\beta h^*}(y)= \lambda\Pi_{\cC}(y/\lambda).
\]

\end{example}


\begin{example}[$\ell_0$ norm]\label{ex:l0}
Consider the function
\[
h(x)=\lambda\|x\|_0+\cI_{\cD}(x),
\]
where $\cD:=\{x\in\R^n:c_1\leq x_{i}\leq c_2\}$   is to guarantee that each element of $x$ is neither too large nor too small,  and $c_1<0<c_2$ are constants.

By performing a simple calculation, we obtain
 the conjugate function of $h$  as follows,
\[h^*(x)=\sum_{i=1}^{n}h^*(x_i),\]
where
\[
h^*(x_i)=\left\{
\begin{array}{lll}
	c_2x_i-\lambda, & \text{ if }  x_i>\frac{\lambda}{c_2},\\
	\\
	0, & \text{ if } \frac{\lambda}{c_1}<x_i\leq \frac{\lambda}{c_2},\\
	\\
	c_1x_i-\lambda, & \text{ if } x_i\leq \frac{\lambda}{c_1},
\end{array}
\right.
\]
and  its proximal mapping is computed by
\bee
[\prox_{\beta h^*}(y)]_i=\left\{
\begin{array}{lll}
	y_i-c_2\beta, & \text{ if }  y_i>c_2\beta+\frac{\lambda}{c_2},\\
	\\
	\frac{\lambda}{c_2}, & \text{ if } \frac{\lambda}{c_2}<y_i\leq c_2\beta+\frac{\lambda}{c_2},\\
	\\
	y_i, & \text{ if } \frac{\lambda}{c_1}<y_i\leq \frac{\lambda}{c_2},\\
	\\
	\frac{\lambda}{c_1}, & \text{ if } c_1\beta+\frac{\lambda}{c_1}<y_i\leq \frac{\lambda}{c_1},\\
	\\
	y_i-c_1\beta, & \text{ if } y_i\leq c_1\beta+\frac{\lambda}{c_1}.
\end{array}
\right.
\eee	
\end{example}

\begin{example}[$\ell_p$ norm]\label{ex:lp}
Let $\|x\|_p$ be the $\ell_p$ norm of $x$ for $p\in(0,1)$.
Consider
\[
h(x)=\lambda\|x\|_p^p+\cI_{\cD}(x),
\]
where $\cD=\{x\in\R^n:\|x\|_{\infty}\leq r\}$ for some constant $r>0$.
The conjugate function of $h$ is
\[
h^*(y)=\sum_{i=1}^n\max\{r|y_i|-\lambda r^p,0\}.
\]
Furthermore, the proximal mapping of $h^*$ is calculated by
\[
[\prox_{\beta h^*}(y)]_i=\cT(y_i,r,\lambda,\beta),\ i=1,\ldots,n,
\]
where
\bee
\cT(y_i,r,\lambda,\beta)=\left\{
\begin{array}{lll}
	y_i, & \text{ if } |y_i|< \lambda r^{p-1},\\
	\\
	\sign(y_i) \lambda r^{p-1}, & \text{ if } \lambda r^{p-1}\leq|y_i|\leq \lambda r^{p-1}+r\beta,\\
	\\
	y_i-\sign(y_i)r\beta, & \text{ if } |y_i|>\lambda r^{p-1}+r\beta.
\end{array}
\right.
\eee	
\end{example}

\begin{example}[SCAD]\label{ex:SCAD}
The SCAD regularizer is defined as follows,
\bee
p_{\lambda,\gamma}(|w|)=\left\{
\begin{array}{lll}
	\lambda|w|, & \text{ if } |w|\leq\lambda,\\
	\\
	\frac{2\gamma\lambda |w|-(w^2+\lambda^2)}{2(\gamma-1)}, & \text{ if } \lambda<|w|\leq\gamma\lambda,\\
	\\
	\frac{\lambda^2(\gamma+1)}{2}, & \text{ if } |w|>\gamma\lambda,
\end{array}
\right.
\eee
where $\lambda>0$ is a given penalty parameter and $\gamma>2$ is a tuning parameter.
Consider the following function
\[
h(x)=\sum_{j=1}^{n}p_{\lambda,\gamma}(|x_j|)+\cI_{\cD}(x),
\]
where $r>0$ is a constant and  $\cD:=\{x\in\R^n:-r\leq x_i\leq r\}$.

Through a straightforward calculation, we derive the conjugate function of $h$ as
\[h^*(x)=\sum_{i=1}^{n}h^*(x_i),\]
where $h^*(x_i)$  is given in three cases:
\begin{itemize}
	\item[(i)] when $r<\lambda$,
	\[
	h^*(x_i)=r\sign(x_i)\max\left\{|x_i|-\lambda,0\right\};
	\]
	\item[(ii)] when $\lambda\leq r<\gamma\lambda$,
	\[h^*(x_i)=
		r\max\left\{|x_i|+\frac{(r-\lambda)^2}{2r(\gamma-1)}-\lambda,0\right\};\]
	\item[(iii)] when $ r\geq\gamma\lambda$,
	\[	h^*(x_i)=
	\left\{
	\begin{array}{lll}
		rx_i-\frac{\lambda^2(\gamma+1)}{2}, & \text{ if }  x_i>\frac{\lambda^2(\gamma+1)}{2r},\\
		\\
		0, & \text{ if } -\frac{\lambda(\gamma+1)}{2\gamma}<x_i\leq\frac{\lambda^2(\gamma+1)}{2r},\\
		\\
		-\lambda \gamma x_i-\frac{\lambda^2(\gamma+1)}{2}, & \text{ if } x_i\leq -\frac{\lambda(\gamma+1)}{2\gamma}.
	\end{array}
	\right.
	\]
\end{itemize}
Then, corresponding to the above three cases, the proximal mapping of $h^*(x)$ is obtained by
 \[[\prox_{\beta h^*}(y)]_i=\cH(y_i,r,\lambda,\beta,\gamma), \ i=1,\ldots,n.\]
 Here, $\cH(y_i,r,\lambda,\beta,\gamma)$ is computed as follows:
 \begin{itemize}
 \item[(i)] when $r<\lambda$,	
  \bee
 \cH(y_i,r,\lambda,\beta,\gamma)=\left\{
 \begin{array}{lll}
 	y_i-r\beta, & \text{ if }  y_i>\lambda+r\beta,\\
 	\\
 	\lambda, & \text{ if } \lambda<y_i\leq\lambda+r\beta,\\
 	\\
 	y_i, & \text{ if } -\lambda+\frac{r\beta}{2}<y_i\leq \lambda,\\
 	\\
 	y_i-r\beta, &  \text{ if } y_i\leq -\lambda+\frac{r\beta}{2};
 \end{array}
 \right.
 \eee
 \item[(ii)] when $\lambda\leq r<\gamma\lambda$,
  \bee
 \cH(y_i,r,\lambda,\beta,\gamma)=\left\{
 \begin{array}{lll}
 	y_i-\sign(y_i)r\beta, & \text{ if }  |y_i|>\lambda+r\beta-\frac{(r-\lambda)^2}{2r(\gamma-1)},\\
 	\\
 	\sign(y_i)\left(\lambda-\frac{(r-\lambda)^2}{2r(\gamma-1)}\right), & \text{ if } \lambda-\frac{(r-\lambda)^2}{2r(\gamma-1)}<|y_i|\leq\lambda+r\beta-\frac{(r-\lambda)^2}{2r(\gamma-1)},\\
 	\\
 	y_i, & \text{ if } |y_i|\leq \lambda-\frac{(r-\lambda)^2}{2r(\gamma-1)};
 \end{array}
 \right.
 \eee
 \item[(iii)] when $ r\geq\gamma\lambda$,
 \bee
 \cH(y_i,r,\lambda,\beta,\gamma)=\left\{
 \begin{array}{lll}
 		y_i-r\beta, & \text{ if }  y_i>r\beta+\frac{\lambda^2(\gamma+1)}{2r},\\
 	\\
 	\frac{\lambda^2(\gamma+1)}{2r}, & \text{ if } \frac{\lambda^2(\gamma+1)}{2r}<y_i\leq r\beta+\frac{\lambda^2(\gamma+1)}{2r},\\
 	\\
 	y_i, & \text{ if } -\frac{\lambda(\gamma+1)}{2\gamma}<y_i\leq\frac{\lambda^2(\gamma+1)}{2r},\\
 	\\
 	-\frac{\lambda(\gamma+1)}{2\gamma}, & \text{ if } -\frac{\lambda(\gamma+1)}{2\gamma}-\gamma\lambda\beta<y_i\leq-\frac{\lambda(\gamma+1)}{2\gamma},\\
 	\\
 	y_i+\gamma\lambda\beta, & \text{ if } y_i\leq -\frac{\lambda(\gamma+1)}{2\gamma}-\gamma\lambda\beta.
 \end{array}
 \right.
 \eee
  \end{itemize}
\end{example}

\end{document}